\documentclass[twoside, 14pt]{article}
\usepackage{amsmath,amssymb,amsthm,mathrsfs}
\usepackage[margin=2.3cm]{geometry} 
\usepackage{lipsum}
\usepackage{titlesec,hyperref}
\usepackage{fancyhdr}
\usepackage[numbers,sort&compress]{natbib}
\usepackage{color}
\usepackage[titletoc]{appendix}

\allowdisplaybreaks

\fancypagestyle{plain}
{
	\fancyhf{}
	
}

\titleformat{\subsection}{\it}{\thesubsection.\enspace}{1.5pt}{}
\titleformat{\subsubsection}{\it}{\thesubsubsection.\enspace}{1.5pt}{}

\newtheorem{theo}{Theorem}[section]
\newtheorem{lemm}[theo]{Lemma}

\newtheorem{prop}[theo]{Proposition}
\newtheorem{rema}{Remark}[section]
\numberwithin{equation}{section}

\allowdisplaybreaks

\def\beq{\begin{equation}}
	\def\bal{\begin{aligned}}
		\def\dal{\end{aligned}}
	\def\deq{\end{equation}}
\def\beqq{\begin{equation*}}
	\def\deqq{\end{equation*}}

\def\p{\partial}

\def\al{\alpha}
\def \f{\frac}
\def \var{\varepsilon}
\def \vf{\varphi}

\def \i {\int_{\mathbb{R}^3_+}}
\def \be {\beta}

\def \ah{\alpha_h}
\def\me{\mathcal{E}}
\def\md{\mathcal{D}}
\def\bme{\bar{\mathcal{E}}}
\def\bmd{\bar{\mathcal{D}}}
\def \wu{\partial_3 u}

\def \wr{\partial_3 \vro}
\def \du{{\rm{div}} u}
\def \hs{\Lambda_h^{-s}}
\def \vro {\varrho}
\def \dl{\delta}

\begin{document}
	\begin{sloppypar}
		\title{Global uniform regularity and vanishing vertical viscosity limit for the compressible Navier--Stokes equations in the half-space  \hspace{-4mm}}
		\author{$\mbox{Jincheng Gao}^1$ \footnote{Email: gaojch5@mail.sysu.edu.cn}, \quad
			$\mbox{Lianyun Peng}^2$ \footnote{Corresponding author. Email: lianyun.peng@polyu.edu.hk}, \quad
			$\mbox{Jiahong Wu}^3$ \footnote{Email: jwu29@nd.edu},\quad 
            $\mbox{Zheng-an Yao}^{1,4}$ \footnote{Email: mcsyao@mail.sysu.edu.cn},\\
			\quad
			$^1\mbox{School}$ of Mathematics, Sun Yat-sen University,\\
			Guangzhou 510275, China\\
			$^2\mbox{Department}$ of Applied Mathematics, 
			The Hong Kong Polytechnic University,\\
			Hong Kong 999077, China\\
			$^3\mbox{Department}$ of Mathematics,
			University of Notre Dame, \\
			Notre Dame 46556, USA\\
            $^4 \mbox{Sun}$ Yat-sen University Institute of Advanced Studies Hong Kong,\\
            Hong Kong 999077, China\\
		}
		
		\date{}
		
		\maketitle

\begin{abstract}
In geophysical flows such as large-scale ocean dynamics, the
vertical viscosity is often much smaller than the horizontal viscosity. This anisotropy makes it natural to ask whether solutions of the full anisotropic
compressible Navier--Stokes equations converge, as the vertical
viscosity coefficient $\varepsilon \to 0$, to solutions of a
horizontally dissipative limit system, and whether this limit
can be justified globally in time.
Prior work has answered this question locally in time or in
the incompressible setting. We resolve this problem for the three-dimensional compressible Navier--Stokes equations in the upper half-space with the Navier
slip boundary condition.  This paper establishes two main results for
small perturbations of a constant equilibrium state. First, we prove the existence of a unique global-in-time solution whose
conormal Sobolev norm remains uniformly bounded for all $t \ge 0$
and all $\varepsilon \in (0,1)$. Second, we justify the global vanishing vertical viscosity limit. More precisely, we show that the solutions converge to a global solution of the horizontally dissipative
compressible Navier--Stokes system. This provides the first rigorous justification of the anisotropic viscosity limit for compressible flows on an infinite time interval. 
			
		\end{abstract}

		\tableofcontents

\section{Introduction}

In many geophysical flows such as large-scale ocean dynamics,
the dissipation is strongly anisotropic: the vertical
(cross-layer) viscosity is orders of magnitude smaller than the
horizontal (along-layer) viscosity.  According to classical
geophysical fluid dynamics (cf.\ \cite[Chapter~4]{1987Geophysical}),
the vertical viscosity coefficient ranges from $1$ to
$10^3\ \mathrm{cm}^2/\mathrm{sec}$, while the horizontal viscosity
ranges from $10^5$ to $10^8\ \mathrm{cm}^2/\mathrm{sec}$.
In such regimes it is natural to ask whether solutions of the full
anisotropic system converge, as the vertical viscosity $\varepsilon
\to 0$, to solutions of a horizontally dissipative limit equation,
and whether this limit can be justified globally in time.

\medskip 
The present paper addresses this question for the
three-dimensional compressible Navier--Stokes equations
in the upper half-space $\mathbb{R}^3_+ := \{x = (x_1, x_2, x_3)
: x_3 > 0\}$:
\begin{equation}\label{eq1}
	\left\{
	\begin{aligned}
		&\p_t \rho^\var + {\rm div}(\rho^\var u^\var) = 0,\\
		&\rho^\var \p_t u^\var + \rho^\var u^\var \cdot \nabla u^\var
		- \Delta_h u^{\var} - \var \p_3^2 u^{\var}
		- \nabla {\rm div}\, u^\var + \nabla P(\rho^\var) = 0,
	\end{aligned}\right.
\end{equation}
where $(t,x) \in \mathbb{R}_+ \times \mathbb{R}^3_+$,
$\Delta_h := \p_1^2 + \p_2^2$ is the horizontal Laplacian,
and the small parameter $\varepsilon \in (0,1)$ is the vertical
viscosity coefficient.  The system is supplemented with the
Navier slip boundary condition
\begin{equation}\label{bc1}
	\bigl(u_3^\var,\, \p_3 u_h^\var\bigr)\big|_{x_3=0} = 0,
\end{equation}
where $\rho^\var$ is the density and $u^\var = (u^\var_h,
u^\var_3)$ is the velocity field with horizontal component
$u^\var_h = (u^\var_1, u^\var_2)$.
The Navier slip condition \eqref{bc1} is physically appropriate
for geophysical flows at a rigid lower boundary. It imposes no
normal flow through the boundary ($u_3^\var|_{x_3=0}=0$) while
allowing tangential sliding, and it avoids the Prandtl boundary
layer that arises under no-slip conditions, making the vanishing
viscosity analysis tractable.
For simplicity we take the pressure law
\begin{equation*}
	P(\rho^\var) = \frac{(\rho^\var)^3}{3}.
\end{equation*}
The qualitative conclusions of this paper extend to the general
$\gamma$-law $P(\rho) = \rho^\gamma/\gamma$ ($\gamma > 1$) with
minor modifications.

\medskip

The system \eqref{eq1} admits the physically important steady state
$\rho^* \equiv 1$, $u^* \equiv 0$.
Introducing the density perturbation
$\vro^{\var} := \rho^{\var} - 1$
and substituting $\rho^{\var} = \vro^{\var} + 1$ into \eqref{eq1},
we obtain the system governing $(\vro^{\var}, u^{\var})$:
\begin{equation}\label{eqr}
	\left\{\begin{aligned}
		&\p_t \vro^{\var} + {\rm div}\, u^{\var}
		= -(u^{\var} \cdot \nabla \vro^{\var}
		+ \vro^{\var}\, {\rm div}\, u^{\var}),\\
		&\p_t u^{\var}
		- \Delta_h u^{\var} - \var \p_3^2 u^{\var}
		- \nabla {\rm div}\, u^{\var} + \nabla \vro^{\var}
		= - u^{\var} \cdot \nabla u^{\var}
		- \frac{\vro^{\var}}{1+\vro^{\var}}
		\bigl(\Delta_h u^{\var} + \var \p_3^2 u^{\var}
		+ \nabla {\rm div}\, u^{\var}\bigr)
		- \vro^{\var} \nabla \vro^{\var},\\
		&\bigl(u^{\var}_3,\, \p_3 u^{\var}_h\bigr)\big|_{x_3=0} = 0,
	\end{aligned}\right.
\end{equation}
with initial data
\begin{equation}\label{eqr-i}
	(\vro^{\var}, u^{\var})\big|_{t=0} = (\vro_0, u_0).
\end{equation}
As $\varepsilon \to 0$, system \eqref{eqr} formally converges to the
horizontally dissipative system
\begin{equation}\label{eqr0}
	\left\{\begin{aligned}
		&\p_t \varrho^0 + {\rm div}\, u^0
		= -(u^0 \cdot \nabla \varrho^0 + \varrho^0\, {\rm div}\, u^0),\\
		&\p_t u^0 - \Delta_h u^0 - \nabla {\rm div}\, u^0 + \nabla \varrho^0
		= -u^0 \cdot \nabla u^0
		- \frac{\varrho^0}{1+\varrho^0}
		\bigl(\Delta_h u^0 + \nabla {\rm div}\, u^0\bigr)
		- \varrho^0 \nabla \varrho^0,\\
		&u^0_3\big|_{x_3=0} = 0,
	\end{aligned}\right.
\end{equation}
with initial data
\begin{equation}\label{eqr-i2}
	(\varrho^0, u^0)\big|_{t=0} = (\varrho_0, u_0).
\end{equation}
The goal of this paper is to rigorously justify the limit from
\eqref{eqr}--\eqref{eqr-i} to \eqref{eqr0}--\eqref{eqr-i2}
globally in time.

\medskip
The mismatch between the boundary conditions of \eqref{eqr} and
\eqref{eqr0} precludes the use of standard
Sobolev spaces.  Following the approach introduced for the
incompressible setting in \cite{Masmoudi2012ARMA, Gao2024},
we work in the conormal Sobolev framework.
Let $(Z_k)_{1 \le k \le 3}$ be the generators of vector fields
tangent to $\p\mathbb{R}^3_+ = \mathbb{R}^2 \times \{x_3=0\}$:
$Z_k = \p_k$ for $k = 1, 2$ and $Z_3 = \varphi(x_3)\p_3$,
where $\varphi \in C^\infty(\mathbb{R}_+)$ satisfies $\varphi(0)=0$,
$\varphi'(0) \neq 0$, and $\varphi(x_3) > 0$ for $x_3 > 0$
(for example, $\varphi(x_3) = x_3(1+x_3)^{-1}$).
Set $Z^\alpha := Z_1^{\alpha_1} Z_2^{\alpha_2} Z_3^{\alpha_3}$
and define the conormal Sobolev space
\begin{equation*}
	H_{co}^m(\mathbb{R}_+^3) :=
	\bigl\{ f \in L^2(\mathbb{R}_+^3)
	\mid Z^\alpha f \in L^2(\mathbb{R}_+^3),\;
	\forall\, 0 \le |\alpha| \le m \bigr\}.
\end{equation*}
We also define $\nabla_h := (\p_1, \p_2, 0)$,
$\alpha_h := (\alpha_1, \alpha_2, 0)$, and the norms
\begin{equation*}
	\begin{aligned}
		&\| f \|_{H^m_{tan}}^2
		:= \sum_{0 \le |\alpha_h| \le m} \| Z^{\alpha_h} f \|_{L^2}^2,
		\qquad
		\| f \|_{\dot{H}_{tan}^m}^2
		:= \sum_{|\alpha_h|=m} \| Z^{\alpha_h} f \|_{L^2}^2,\\
		&\| f \|_{H^m_{co}}^2
		:= \sum_{0 \le |\alpha| \le m} \| Z^\alpha f \|_{L^2}^2,
		\qquad
		\| f \|_{\dot{H}^m_{co}}^2
		:= \sum_{|\alpha|=m} \| Z^\alpha f \|_{L^2}^2.
	\end{aligned}
\end{equation*}
Finally, the horizontal Fourier multiplier $\Lambda_h^s$, $s \in
\mathbb{R}$, is defined by
\begin{equation*}
	\Lambda_h^s f(x_h) :=
	\int_{\mathbb{R}^2} |\xi_h|^s \hat{f}(\xi_h)\,
	e^{2\pi i x_h \cdot \xi_h}\, d\xi_h,
\end{equation*}
where $x_h := (x_1, x_2, 0)$ and $\hat{f}$ denotes the horizontal
Fourier transform.

\medskip

Let $w_h^\var := (\nabla \times u^\var)_h$ denote the horizontal
vorticity, and fix a constant $\zeta \in (0, \frac{1}{34}]$.
We define the energy functional
\begin{equation*}
	\begin{aligned}
		\mathcal{E}(t) :=\;
		&\| (\vro^\var, u^\var)(t) \|_{H^m_{co}}^2
		+ \| (\p_3 \vro^\var, \p_3 u^\var)(t) \|_{H^{m-1}_{co}}^2\\
		&+ \| w_h^\var(t) \|_{L^\infty}^3
		+ \| \p_3 \vro^\var(t) \|_{L^\infty}^3\\
		&+ \| \Lambda_h^{-(1-\zeta)}(\vro^\var, u^\var)(t) \|_{L^2}^2
		+ \| \Lambda_h^{-(1-\zeta)} \p_3(\vro^\var, u^\var)(t) \|_{L^2}^2\\
		&+ \| \p_t(u^\var, w_h^\var, \p_3 u_3^\var)(t) \|_{H_{tan}^2}^2
		+ \| Z_3 \p_t(w_h^\var, \p_3 u_3^\var)(t) \|_{L^2}^2
		+ \varepsilon \| \p_3^2 \vro^\var(t) \|_{H_{co}^2}^2,
	\end{aligned}
\end{equation*}
and the dissipation functional
\begin{equation*}
	\begin{aligned}
		\mathcal{D}(t) :=\;
		&\| \nabla_h u^\var(t) \|_{H^m_{co}}^2
		+ \| {\rm div}\, u^\var(t) \|_{H^m_{co}}^2
		+ \varepsilon \| \p_3 u^\var(t) \|_{H^m_{co}}^2\\
		&+ \| \p_3 \nabla_h u^\var(t) \|_{H^{m-1}_{co}}^2
		+ \| \p_3\, {\rm div}\, u^\var(t) \|_{H^{m-1}_{co}}^2
		+ \varepsilon \| \p_3^2 u^\var(t) \|_{H^{m-1}_{co}}^2\\
		&+ \| \p_t(\nabla_h u^\var, {\rm div}\, u^\var,
		\nabla_h w_h^\var,
		\nabla \p_3 u_3^\var)(t) \|_{H^2_{tan}}^2
		+ \| Z_3 \p_t(\nabla_h w_h^\var,
		\nabla \p_3 u_3^\var)(t) \|_{L^2}^2\\
		&+ \varepsilon\bigl(
		\| \p_3 \p_t u^\var(t) \|_{H^2_{tan}}^2
		+\| \p_3 \p_t w_h^\var(t) \|_{H^2_{tan}}^2
		+\| \p_3 Z_3 \p_t w_h^\var(t) \|_{L^2}^2
		\bigr)\\
		&+ \| \nabla \vro^\var(t) \|_{H^{m-1}_{co}}^2
		+ \varepsilon \| \p_3^2 \vro^\var(t) \|_{H_{co}^2}^2.
	\end{aligned}
\end{equation*}

Several features of this energy norm deserve some comments.
First,  because of the mismatch of boundary conditions, only one
order of vertical derivative can be controlled in conormal space,
hence the $H^{m-1}_{co}$ regularity for $\p_3(\vro^\var, u^\var)$.
Second, although the density equation has no dissipation,
the coupling with the momentum equation transfers part of the
velocity dissipation to the density via the pressure term,
yielding the $H^{m-1}_{co}$ control of $\nabla \vro^\var$.
Third, to close the energy estimate it is essential to control
the $L^\infty$-norms of $(w_h^\var, \p_3\vro^\var)$.
Obtaining these requires higher-order vertical derivatives,
which are not available directly. We trade them for
time derivatives via the momentum equation, explaining the
presence of $\p_t$-terms in $\mathcal{E}(t)$.
Fourth, the cubic powers $\|w_h^\var\|_{L^\infty}^3$ and
$\|\p_3\vro^\var\|_{L^\infty}^3$ are required in order that
the power of $\|(w_h^\var, \p_3\vro^\var)\|_{L^\infty}$ in
$\mathcal{E}(t)$ dominates the linear contributions from
equation \eqref{25} when applying the maximum principle (see Remark~\ref{rema-Linfty} below).

\medskip

For the time decay analysis we define the tangential energy and
dissipation functionals
\begin{equation*}
	\begin{aligned}
		\mathcal{E}_{tan}(t)
		&:= \| (\vro^\var, u^\var)(t) \|_{H^m_{tan}}^2
		+ \| \p_3(\vro^\var, u^\var)(t) \|_{H^{m-1}_{tan}}^2,\\
		\bar{\mathcal{E}}_{tan}(t)
		&:= \| \nabla_h(\vro^\var, u^\var)(t) \|_{H^{m-2}_{tan}}^2
		+ \| \nabla_h \p_3(\vro^\var, u^\var)(t) \|_{H^{m-3}_{tan}}^2,
	\end{aligned}
\end{equation*}
\begin{equation*}
	\begin{aligned}
		\mathcal{D}_{tan}(t)
		&:= \| (\nabla_h u^\var, {\rm div}\, u^\var)(t) \|_{H^m_{tan}}^2
		+ \varepsilon \| \p_3 u^\var(t) \|_{H^m_{tan}}^2\\
		&\quad + \| \p_3(\nabla_h u^\var, {\rm div}\, u^\var)(t)
		\|_{H^{m-1}_{tan}}^2
		+ \varepsilon \| \p_3^2 u^\var(t) \|_{H^{m-1}_{tan}}^2
		+ \| \nabla \vro^\var(t) \|_{H^{m-1}_{tan}}^2,\\
		\bar{\mathcal{D}}_{tan}(t)
		&:= \| \nabla_h(\nabla_h u^\var, {\rm div}\, u^\var)(t)
		\|_{H^{m-2}_{tan}}^2
		+ \varepsilon \| \nabla_h \p_3 u^\var(t) \|_{H^{m-2}_{tan}}^2\\
		&\quad
		+ \| \nabla_h \p_3(\nabla_h u^\var, {\rm div}\, u^\var)(t)
		\|_{H^{m-3}_{tan}}^2
		+ \varepsilon \| \nabla_h \p_3^2 u^\var(t) \|_{H^{m-3}_{tan}}^2
		+ \| \nabla_h \nabla \vro^\var(t) \|_{H^{m-3}_{tan}}^2.
	\end{aligned}
\end{equation*}
The functional $\bar{\mathcal{E}}_{tan}(t)$ enjoys a faster
time decay rate than $\mathcal{E}_{tan}(t)$; this two-tier decay
structure is a key mechanism for controlling $L^1$-in-time
integrals in the energy argument.

\medskip
We can now state our two main theorems.

\begin{theo}\label{main_result_one}
	Let $m \ge 5$ be an integer and $\zeta \in (0, \frac{1}{34}]$.
	Assume the initial data $(\vro_0, u_0)$ satisfies the smallness
	condition
	\begin{equation}\label{condition}
		\mathcal{E}(0) \le \delta_0,
	\end{equation}
	for some sufficiently small constant $\delta_0 > 0$.
	Then system \eqref{eqr}--\eqref{eqr-i} has a unique global-in-time
	solution $(\vro^\var, u^\var)$ satisfying the uniform estimate
	\begin{equation}\label{uniform_estimate}
		\mathcal{E}(t) + \int_0^t \mathcal{D}(\tau)\, d\tau \le C\delta_0,
	\end{equation}
	and the time decay estimate
	\begin{equation}\label{decay_estimate}
		\begin{aligned}
			&(1+t)^{1-\zeta} \mathcal{E}_{tan}(t)
			+ (1+t)^{2-2\zeta} \bar{\mathcal{E}}_{tan}(t)\\
			&\quad
			+ \int_0^t (1+\tau)^{1-2\zeta} \mathcal{D}_{tan}(\tau)\, d\tau
			+ \int_0^t (1+\tau)^{2-2\zeta} \bar{\mathcal{D}}_{tan}(\tau)\, d\tau
			\le C\delta_0,
		\end{aligned}
	\end{equation}
	where the constant $C > 0$ is independent of $t \ge 0$ and
	$\varepsilon \in (0,1)$.
\end{theo}

\begin{rema}\label{rema-zeta}
	The restriction $\zeta \in (0, \frac{1}{34}]$ arises from three
	independent requirements.
	First, Hardy--Littlewood--Sobolev applied to the negative-order
	norm $\|\Lambda_h^{-(1-\zeta)}(\vro^\var, u^\var)\|_{L^2}$ requires
	$\zeta \in (0,1)$.
	Second, for the time integral of $\mathcal{D}_{tan}$ to converge,
	the decay rate of $\mathcal{E}_{tan}$ must be strictly faster than
	that of $\mathcal{D}_{tan}$, i.e.\ $1-\zeta > 1-2\zeta$, which
	gives $\zeta > 0$.
	Third, to obtain the improved decay rate of
	$\bar{\mathcal{E}}_{tan}$, an interpolation inequality
	(cf.\ \eqref{inter-2}) forces $\zeta \le \frac{1}{34}$;
	see \eqref{est-s-16/17} for the precise calculation.
\end{rema}

\begin{rema}\label{rema-vert}
	The absence of vertical dissipation in \eqref{eqr} forces us to
	include one order of vertical derivative in the decay estimate.
	The convective term in \eqref{eqr} prevents obtaining the same
	decay rate for the top-order ($m$-th order) horizontal derivative
	as for lower orders; see Lemma~\ref{lemma33} for details.
\end{rema}

\begin{rema}\label{rema-Linfty}
	The uniform control of $\|\p_3 \vro^\var\|_{L^\infty}$ is obtained
	by applying the maximum principle to the equation for
	$\p_3\vro^\var$.  The resulting equation contains linear terms
	proportional to $(1+\vro^\var)\p_t u_3^\var$, $\Delta_h u_3^\var$,
	and $\varepsilon \p_3 \nabla_h \cdot u_h^\var$
	(see equation \eqref{25}), which can only be absorbed provided
	the exponent of $\|(w_h^\var, \p_3\vro^\var)\|_{L^\infty}$
	in $\mathcal{E}(t)$ is strictly larger than one.  This forces
	the cubic exponent in the definition of $\mathcal{E}(t)$.
\end{rema}

\begin{rema}\label{rema-div}
	In the pioneering work of Wang, Xin, and Yong \cite{Wang-Xin-Yan2015}, 
	uniform regularity and the vanishing viscosity limit were established 
	locally in time for the full compressible system
	\begin{equation}\label{cns}
		\left\{\begin{aligned}
			&\p_t \rho^\var + {\rm div}(\rho^\var u^\var) = 0,\\
			&\rho^\var \p_t u^\var + \rho^\var u^\var \cdot \nabla u^\var
			- \var(\Delta_h u^\var + \p_3^2 u^\var
			+ \nabla {\rm div}\, u^\var)
			+ \nabla P(\rho^\var) = 0,
		\end{aligned}\right.
	\end{equation}
	with general Navier-slip boundary conditions. A key structural 
	difference between \eqref{cns} and the system \eqref{eq1} studied 
	here is the role of the divergence term $-\nabla {\rm div}\, u^\var$, 
	which has a dual effect.  On one hand, it weakens the
	vertical dissipation for the density: for system \eqref{cns} one
	obtains $\|\p_3^2 \vro^\var\|_{L^2}^2$, while for \eqref{eq1}
	only the weaker $\varepsilon\|\p_3^2 \vro^\var\|_{L^2}^2$ is
	available, which forces $\|\p_3\vro^\var\|_{L^\infty}$ into the
	energy functional.  On the other hand, the same divergence term
	provides the strong dissipation $\|{\rm div}\, u^\var\|_{L^2}^2$,
	which is essential for controlling the nonlinear convection term
	$u^\var \cdot \nabla u^\var$.
\end{rema}

\medskip

Our second main result justifies the vanishing vertical viscosity
limit globally in time.

\begin{theo}\label{main_result_two}
	Under the conditions of Theorem~\ref{main_result_one}, the global
	solution $(\vro^\var, u^\var)$ of \eqref{eqr}--\eqref{eqr-i}
	contains a subsequence $(\vro^{\var_n}, u^{\var_n})$ with
	$\var_n \to 0$ such that, for every $T > 0$,
	\begin{equation*}
		(\vro^{\var_n}, u^{\var_n})
		\longrightarrow (\vro^0, u^0)
		\quad \text{in } C\bigl([0,T];\, H^{m-1}_{co,\mathrm{loc}}\bigr),
	\end{equation*}
	 where $(\vro^0, u^0)$ is a global solution of
	the limit system \eqref{eqr0}--\eqref{eqr-i2}.
\end{theo}

\begin{rema}
	The convergence in Theorem~\ref{main_result_two} holds along a
	subsequence; full convergence as $\varepsilon \to 0$ would follow
	from uniqueness of global solutions to the limit system
	\eqref{eqr0}, which to our knowledge remains an open problem in
	general.
\end{rema}

We now place these results in the context of the existing literature
and explain in what respects they go beyond it. For the isotropic compressible Navier--Stokes equations (formally
\eqref{eq1} with $\varepsilon=1$), Matsumura and Nishida
\cite{Matsumura-Nishida1980, Matsumura-Nishida1983} established
global existence of classical solutions near a constant equilibrium
in the $H^3$ framework.  Huang, Li, and Xin \cite{Huang-Li-Xin2012}
later obtained global strong solutions for small but
possibly highly oscillatory initial data; see also
\cite{Fan-Li-Li2022, Huang-Li2018, Hong-Hou-Peng-Zhu2024,
	Wen-2017} for further developments.  Wen \cite{Wen-2025} recently
proved global well-posedness under a scaling-invariant smallness
condition, even allowing far-field vacuum.  The large-time behavior
of solutions has been extensively studied; we refer to
\cite{Guo-Wang2012, Hoff-Zumbrun1995, Liu-Wang1998,
	Kagei-Kobayashi2002} for pointwise decay and energy decay results.
Some blow-up criterions for the strong solution as the time variable approaches the maximum time of existence, see also
\cite{{Huang-Li-Xin-Siam},{Huang-Li-Xin-CMP},{Wen-Zhu-Adv},{Sun-Wang-Zhang-JMPA}}.
The global well-posedness and optimal decay rates for the limit
system \eqref{eqr0} in the three-dimensional whole space
were established in \cite{FHWW2025} for small initial perturbations.

\medskip
The most direct predecessor of the present work is the incompressible
analogue of our problem. The incompressible counterpart of our problem concerns the equation
\begin{equation}\label{eq3}
	\p_t u + u \cdot \nabla u
	- \p_1^2 u - \p_2^2 u - \var \p_3^2 u + \nabla P = 0,
	\quad {\rm div}\, u = 0,
\end{equation}
which converges formally as $\varepsilon \to 0$ to
\begin{equation}\label{eq4}
	\p_t u + u \cdot \nabla u
	- \p_1^2 u - \p_2^2 u + \nabla P = 0,
	\quad {\rm div}\, u = 0.
\end{equation}
The well-posedness of \eqref{eq4} in Sobolev and Besov spaces under
small data has been widely studied
\cite{ani-CDGG2000, Chemin2007, Iftimie2002, Liu-Paica-Zhang-2020,
	Paicu2005, Zhang-2009, new-Cao-Wu2025}.
Large-time behavior on $\mathbb{T}^2 \times \mathbb{R}$ and
$\mathbb{R}^3$ is addressed in \cite{jWy2021, JTW2023}.
For the vanishing vertical viscosity limit of \eqref{eq3}
in the whole space, see \cite{ani-CDGG2000}; for the half-space
with Navier slip conditions, Iftimie and Planas
\cite{ani-Iftimie-Planas2016} proved convergence by a compactness
argument under weak initial data assumptions.
Global-in-time uniform estimates in the conormal framework for
\eqref{eq3} in the half-space with slip boundary conditions,
together with a time-independent convergence rate to \eqref{eq4},
were established in \cite{Gao2024}.
The present paper extends this incompressible analysis to the
compressible setting, which introduces substantial new
difficulties owing to the coupling between the density and velocity
equations and the weaker dissipation structure.

\medskip
Regarding the vanishing viscosity limit for compressible flows,
Wang, Xin, and Yong \cite{Wang-Xin-Yan2015} established uniform
regularity and convergence for the full (isotropic) viscosity limit
in a bounded domain with Navier-slip conditions, but only locally
in time; the key structural difference from \eqref{eq1} is
discussed in Remark~\ref{rema-div}.
For the simultaneous incompressible and vanishing vertical viscosity
limits of strong solutions to \eqref{eq1} with ill-prepared data
in a Dirichlet boundary condition setting, see the recent work of
Masmoudi, Sun, Wang, and Zhang \cite{Masmoudi-Sun-Wang-Zhang}.
More broadly, the vanishing viscosity limit for the incompressible
Navier--Stokes equations is one of the central open problems in
fluid mechanics; we refer to \cite{Constantin1986, Constantin1988,
	Kato1972, Masmoudi2007CMP, ConstantinWu1995, ConstantinWu1996,
	BeiraodaVeiga2010, BeiraodaVeiga2011, Xiao2007, Iftimie2011ARMA,
	Masmoudi2012ARMA, Xiao2013, Wang-2016, Wang-Xin-Yan2015} and the
references therein for the extensive literature on this topic.

\medskip
In this context, the present paper makes three main contributions. First, while all prior results for the vanishing vertical viscosity
limit of compressible Navier--Stokes \cite{Wang-Xin-Yan2015,
	Masmoudi-Sun-Wang-Zhang} are either local in time or restricted to
the incompressible regime, this paper is the first to establish
global-in-time uniform estimates and justify the anisotropic limit
on an infinite time interval for the compressible system \eqref{eq1}
with slip boundary conditions.
Second, compared with the incompressible analysis of \cite{Gao2024},
the compressible setting introduces the density perturbation
$\vro^\var$ as a new unknown with no vertical dissipation;
its coupling to the velocity via the pressure generates additional
nonlinear terms, and the resulting weaker dissipation structure
(only $\varepsilon\|\p_3^2\vro^\var\|_{L^2}^2$ is available,
compared with $\|\p_3^2\vro^\var\|_{L^2}^2$ for \eqref{cns};
see Remark~\ref{rema-div}) forces the $L^\infty$ control of
$\p_3\vro^\var$ to be carried as an independent component of the
energy throughout the entire argument.
Third, the two-tier energy method developed here constitutes a
genuinely new contribution even relative to the incompressible
case of \cite{Gao2024}. The
faster algebraic decay of $\bar{\mathcal{E}}_{tan}$ at rate
$(1+t)^{-(2-2\zeta)}$ feeds back into the $L^1$-in-time control
of the critical nonlinear terms needed to close the $L^\infty$
estimates uniformly in $t$ and $\varepsilon$.

\medskip
The proof of Theorem~\ref{main_result_one} is built on a
bootstrap argument in which energy and dissipation estimates
are closed simultaneously.  We briefly outline the two main
steps.  A detailed technical discussion is given in
Section~\ref{section-approach}.

\textbf{Step A: Horizontal derivatives and time decay.}
We first establish uniform estimates for tangential derivatives
and derive enhanced dissipation for the density via the pressure
coupling.  Incorporating the negative-order horizontal Sobolev
norms $\|\Lambda_h^{-(1-\zeta)}(\vro^\var, u^\var)\|_{L^2}^2$
into the energy, we derive the nonlinear differential inequality
\begin{equation*}
	\frac{d}{dt}\widehat{\mathcal{E}}_{tan}(t)
	+ \kappa\, C_0^{-\frac{1}{1-\zeta}}
	\widehat{\mathcal{E}}_{tan}(t)^{1+\frac{1}{1-\zeta}} \le 0,
\end{equation*}
which yields the algebraic decay
$\widehat{\mathcal{E}}_{tan}(t) \lesssim C_0(1+t)^{-(1-\zeta)}$.
A second level of estimates then gives the faster decay
$(1+t)^{-(2-2\zeta)}$ for $\bar{\mathcal{E}}_{tan}(t)$,
establishing \eqref{decay_estimate}.

\textbf{Step B: Vertical $L^\infty$ estimates and closure.}
The maximum principle applied to the equations for $w_h^\var$
and $\p_3\vro^\var$ yields integral-in-time bounds on their
$L^\infty$ norms, provided the $L^1$-in-time integrability of
critical quantities is available.  The latter is exactly what the
two-tier decay from Step A supplies.  Combining Steps A and B
within the bootstrap scheme closes the energy inequality
\eqref{uniform_estimate} and completes the proof.

Theorem~\ref{main_result_two} then follows from a classical
strong compactness argument: the uniform bounds in
Theorem~\ref{main_result_one} imply equicontinuity in time and
compactness in space, so Aubin--Lions type results apply to
extract a convergent subsequence.

\medskip
Throughout the paper we write $A \lesssim B$ to mean $A \le CB$
for a constant $C > 0$ that may change from line to line and is
independent of $t \ge 0$ and $\varepsilon \in (0,1)$.
We write $A \sim B$ when $A \lesssim B$ and $B \lesssim A$.

\medskip
The remainder of the paper is organized as follows. The rest of the paper is organized as follows.
In Section \ref{section-approach}, we explain the difficulties and our approach to establish
the global uniform estimate and time decay rate estimate
for the system  \eqref{eqr} and \eqref{eqr0} respectively.
In Section  \ref{global-estimate}, we apply the energy method
to establish the global uniform and decay rate estimate
for system  \eqref{eqr} under the condition of small initial data \eqref{condition}.
In Section \ref{asymptotic-behavior},
with the help of the uniform estimate in Theorem \ref{main_result_one}, we obtain the vanishing vertical viscosity limit to the equation \eqref{eqr}, which proves Theorem \ref{main_result_two}.
Finally, we introduce some useful inequalities in Appendix \ref{usefull-inequality}.

\vskip .2in 
\section{Difficulties and outline of our approach}
\label{section-approach}

The proof of Theorem~\ref{main_result_one} is divided into four major steps.
Steps~1--3 establish the $L^\infty$ control of the horizontal
vorticity $w_h^\varepsilon$ and of $\partial_3\varrho^\varepsilon$,
which is needed to close the conormal energy estimate for
one-order vertical derivatives.
Step~4 derives the tangential energy estimates and the two-tier
time decay, which in turn supplies the $L^1$-in-time
integrability required in Steps~2 and~3.
The four steps are therefore mutually dependent and are closed
simultaneously within the bootstrap scheme of
Proposition~\ref{main_pro} in Section~\ref{global-estimate}.

 {\textbf{Step 1: Vertical derivative estimate.}}  Because of the boundary layer effect induced by the mismatch
 of boundary conditions between \eqref{eqr} and \eqref{eqr0},
 the global uniform estimate must be established in the
 conormal Sobolev framework.
 In particular, it is necessary to control
 $\|(\partial_3\varrho^\varepsilon,
 \partial_3 u^\varepsilon)\|_{H^{m-1}_{co}}^2$.
 Estimating the term $III_{1,1}$ when $|\alpha| = m-1$,
 one encounters the integral
 \begin{equation*}
 	\int_{\mathbb{R}^3_+}
 	\partial_3 u_h^{\varepsilon}
 	\cdot Z^\alpha \nabla_h \varrho^{\varepsilon}
 	\cdot \partial_3 Z^\alpha \varrho^{\varepsilon}\, dx
 	\;\le\;
 	\|\partial_3 u_h^{\varepsilon}\|_{L^\infty}
 	\|Z^\alpha \nabla_h \varrho^{\varepsilon}\|_{L^2}
 	\|\partial_3 Z^\alpha \varrho^{\varepsilon}\|_{L^2},
 \end{equation*}
 which forces control of $\|\partial_3 u^\varepsilon\|_{L^\infty}$.
 By the classical Sobolev embedding, all components of
 $\partial_3 u^\varepsilon$ except the horizontal vorticity
 $w_h^\varepsilon$ are already bounded by the conormal energy
 $\mathcal{E}(t)$.
 It therefore suffices to estimate
 $\|w_h^\varepsilon\|_{L^\infty}$ separately, which is the
 subject of Step~2.

   {\textbf{Step 2: $L^{\infty}-$Estimate of $w^{\var}_h$.}} Applying the maximum principle to the equation satisfied by
   $w_h^\varepsilon$, one obtains
   \begin{equation}\label{2222}
   	\begin{aligned}
   		\|w^{\varepsilon}_h(t)\|_{L^\infty}
   		\;\leq\;&
   		\|w^{\varepsilon}_h(0)\|_{L^\infty}
   		+ \int_{0}^t
   		\|w^{\varepsilon}_h\,{\rm div}\,u^{\varepsilon}
   		+ w^{\varepsilon}\cdot\nabla u^{\varepsilon}_h
   		\|_{L^\infty}\,d\tau \\
   		&+ \int_{0}^t
   		\Bigl\|\Bigl[\nabla\Bigl(\frac{\varrho^{\varepsilon}}
   		{1+\varrho^{\varepsilon}}\Bigr)
   		\times \Delta_h u^{\varepsilon}\Bigr]_h
   		\Bigr\|_{L^\infty}\,d\tau \\
   		&+ \varepsilon\int_{0}^t
   		\Bigl\|\Bigl[\nabla\Bigl(\frac{\varrho^{\varepsilon}}
   		{1+\varrho^{\varepsilon}}\Bigr)
   		\times \partial_3^2 u^{\varepsilon}\Bigr]_h
   		\Bigr\|_{L^\infty}\,d\tau \\
   		&+ \int_{0}^t
   		\Bigl\|\Bigl[\nabla\Bigl(\frac{\varrho^{\varepsilon}}
   		{1+\varrho^{\varepsilon}}\Bigr)
   		\times \nabla\,{\rm div}\,u^{\varepsilon}\Bigr]_h
   		\Bigr\|_{L^\infty}\,d\tau.
   	\end{aligned}
   \end{equation}
The first three integrals on the right are controlled using
the time decay estimates of Step~4 together with
Cauchy--Schwarz in time.
The fourth and fifth integrals are more delicate; we address
each in turn. The critical contribution from the fourth term is
$\varepsilon\int_0^t
\|\partial_3\varrho^\varepsilon\|_{L^\infty}
\|\partial_3^2 u_h^\varepsilon\|_{L^\infty}\,d\tau$.
To handle $\|\partial_3^2 u_h^\varepsilon\|_{L^\infty}$
without losing a factor of $\varepsilon^{-1}$, we extract
one vertical derivative from the horizontal momentum equation:
\begin{equation*}
	\begin{aligned}
		\varepsilon\|\partial_3^3 u^{\varepsilon}_h\|_{H^2_{tan}}
		\;\lesssim\;&
		\bigl\|\partial_3\bigl(
		(1+\varrho^{\varepsilon})\partial_t u^{\varepsilon}_h
		+ (1+\varrho^{\varepsilon})u^{\varepsilon}\cdot
		\nabla u^{\varepsilon}_h
		- \Delta_h u^{\varepsilon}_h
		- \nabla_h\,{\rm div}\,u^{\varepsilon}
		+ (1+\varrho^{\varepsilon})^2\nabla_h\varrho^{\varepsilon}
		\bigr)\bigr\|_{H^2_{tan}} \\
		\;\lesssim\;&
		\mathcal{E}(t)^{1/2} + \mathcal{D}(t)^{1/2}.
	\end{aligned}
\end{equation*}
Applying the anisotropic Sobolev inequality and
H\"{o}lder's inequality in time then yields
\begin{equation}\label{21}
	\begin{aligned}
		&\varepsilon\int_{0}^t
		\|\partial_3\varrho^{\varepsilon}\|_{L^\infty}
		\|\partial_3^2 u^{\varepsilon}_h\|_{L^\infty}\,d\tau \\
		\lesssim\;&
		\varepsilon\int_{0}^t
		\|\partial_3\varrho^{\varepsilon}\|_{L^2}^{1/8}
		\|\partial_3\nabla_h\varrho^{\varepsilon}\|_{H^1_{tan}}^{3/8}
		\|\partial_3^2\varrho^{\varepsilon}\|_{H^2_{tan}}^{1/2}
		\|\partial_3^2 u^{\varepsilon}_h\|_{H^2_{tan}}^{1/2}
		\|\partial_3^3 u^{\varepsilon}_h\|_{H^2_{tan}}^{1/2}
		\,d\tau \\
		\lesssim\;&
		\sup_{0\le\tau\le t}\mathcal{E}(\tau)^{1/2}
		\left\{\int_{0}^t
		\|\partial_3\varrho^{\varepsilon}\|_{L^2}^2
		(1+\tau)^{1-2\zeta}\,d\tau
		\right\}^{1/16}
		\left\{\int_{0}^t
		\|\partial_3\nabla_h\varrho^{\varepsilon}\|_{H^1_{tan}}^2
		(1+\tau)^{2-2\zeta}\,d\tau
		\right\}^{3/16} \\
		&\times
		\left\{\int_{0}^t
		\varepsilon\|\partial_3^2\varrho^{\varepsilon}\|_{H^2_{tan}}^2
		\,d\tau
		\right\}^{1/4}
		\left\{\int_{0}^t
		\varepsilon\|\partial_3^2 u^{\varepsilon}_h\|_{H^2_{tan}}^2
		\,d\tau
		\right\}^{1/4}
		\left\{\int_{0}^t
		(1+\tau)^{-\frac{7-8\zeta}{4}}\,d\tau
		\right\}^{1/4}
		+ \int_{0}^t\mathcal{D}(\tau)\,d\tau.
	\end{aligned}
\end{equation}
The time integrals weighted by $(1+\tau)^{1-2\zeta}$ and
$(1+\tau)^{2-2\zeta}$ are precisely those provided by
the two-tier decay estimate \eqref{decay_estimate}.
The integral $\int_0^t(1+\tau)^{-(7-8\zeta)/4}\,d\tau$ converges
since $\zeta\le\frac{1}{34}$ implies $\frac{7-8\zeta}{4}>1$.   
   
The critical contribution from the fifth term is
$\int_0^t
\|\nabla_h\varrho^\varepsilon\|_{L^\infty}
\|\partial_3\,{\rm div}\,u^\varepsilon\|_{L^\infty}\,d\tau$.
A direct estimate of $\|\partial_3\,{\rm div}\,u^\varepsilon
\|_{L^\infty}$ would require $\partial_3^2 u^\varepsilon$
at the $L^\infty$ level, which is not available.
Instead, we extract $\partial_3\,{\rm div}\,u^\varepsilon$
directly from the vertical component of the momentum equation:
\begin{equation}\label{23}
	(1+\varepsilon)\,\partial_3\,{\rm div}\,u^{\varepsilon}
	= (1+\varrho^{\varepsilon})\partial_t u^{\varepsilon}_3
	+ (1+\varrho^{\varepsilon})u^{\varepsilon}\cdot
	\nabla u^{\varepsilon}_3
	- \Delta_h u^{\varepsilon}_3
	+ \varepsilon\,\partial_3\nabla_h\cdot u^{\varepsilon}_h
	+ (1+\varrho^{\varepsilon})^2\partial_3\varrho^{\varepsilon}.
\end{equation}
Using \eqref{23} together with the anisotropic Sobolev
inequality and the two-tier decay of Step~4, one obtains
\begin{equation}\label{22}
	\begin{aligned}
		&\int_{0}^t
		\|\nabla_h\varrho^{\varepsilon}\|_{L^\infty}
		\|\partial_3\,{\rm div}\,u^{\varepsilon}\|_{L^\infty}
		\,d\tau \\
		\lesssim\;&
		\sup_{0\le\tau\le t}
		\bigl(
		\|(\partial_t u^{\varepsilon}_3,
		\partial_3\partial_t u^{\varepsilon}_3)(\tau)
		\|_{H^2_{tan}}
		+ \|\partial_3\varrho^{\varepsilon}(\tau)\|_{L^\infty}
		\bigr)
		\left\{\int_{0}^t
		\|\nabla_h\varrho^{\varepsilon}\|_{L^2}^2
		(1+\tau)^{1-2\zeta}\,d\tau
		\right\}^{1/16} \\
		&\times
		\left\{\int_{0}^t
		\|\nabla\nabla_h\varrho^{\varepsilon}\|_{H^2_{tan}}^2
		(1+\tau)^{2-2\zeta}\,d\tau
		\right\}^{7/16}
		\left\{\int_{0}^t
		(1+\tau)^{-\frac{15-16\zeta}{8}}\,d\tau
		\right\}^{1/2} \\
		&+
		\bigl(1 + \sup_{0\le\tau\le t}
		\|u^{\varepsilon}(\tau)\|_{L^\infty}\bigr)
		\int_{0}^t\mathcal{D}_{tan}(\tau)\,d\tau.
	\end{aligned}
\end{equation}
The time exponent satisfies
$\frac{15-16\zeta}{8} > 1$ for $\zeta \le \frac{1}{34}$,
so the last time integral in \eqref{22} converges.   
   
The estimates \eqref{21} and \eqref{22} show that
controlling $\|w_h^\varepsilon\|_{L^\infty}$ requires
the additional quantities
\begin{equation*}
	\|\partial_t(u^\varepsilon, w^{\varepsilon}_h,
	\partial_3 u^{\varepsilon}_3)(t)\|_{H^2_{tan}}^2
	+ \|Z_3\partial_t(w^{\varepsilon}_h,
	\partial_3 u^{\varepsilon}_3)(t)\|_{L^2}^2
	+ \varepsilon\|\partial_3^2\varrho^{\varepsilon}(t)
	\|_{H^2_{co}}^2
\end{equation*}
to be part of the energy functional $\mathcal{E}(t)$.
This explains the presence of the mixed time--vertical
derivative terms in the definition of $\mathcal{E}(t)$
in the introduction.
Furthermore, the decay weights $(1+\tau)^{1-2\zeta}$ and
$(1+\tau)^{2-2\zeta}$ in \eqref{21}--\eqref{22} show
that the optimal time decay of the tangential energy is
indispensable for closing the $L^\infty$ estimate.

\textbf{Step 3: $L^\infty$ estimate of
	$\partial_3\varrho^\varepsilon$.}
Applying $\partial_3$ to the density equation
$\eqref{eqr}_1$ yields the transport equation
\begin{equation*}
	\partial_t\partial_3\varrho^{\varepsilon}
	+ u^{\varepsilon}\cdot\nabla\partial_3\varrho^{\varepsilon}
	+ (1+\varrho^{\varepsilon})\partial_3\,{\rm div}\,u^{\varepsilon}
	= -\partial_3\varrho^{\varepsilon}\,{\rm div}\,u^{\varepsilon}
	- \partial_3 u^{\varepsilon}\cdot\nabla\varrho^{\varepsilon}.
\end{equation*}
The term $(1+\varrho^\varepsilon)\partial_3\,{\rm div}\,
u^\varepsilon$ on the left is linear in
$\partial_3 u^\varepsilon$ and cannot be directly absorbed
by the energy.
Substituting \eqref{23} to eliminate
$\partial_3\,{\rm div}\,u^\varepsilon$, we arrive at the
equation
\begin{equation}\label{25}
	\begin{aligned}
		&\partial_t\partial_3\varrho^{\varepsilon}
		+ u^{\varepsilon}\cdot\nabla\partial_3\varrho^{\varepsilon}
		+ \frac{1}{1+\varepsilon}\partial_3\varrho^{\varepsilon} \\
		=\;&
		-\frac{1+\varrho^{\varepsilon}}{1+\varepsilon}
		\Bigl(
		(1+\varrho^{\varepsilon})\partial_t u^{\varepsilon}_3
		- \Delta_h u^{\varepsilon}_3
		+ \varepsilon\,\partial_3\nabla_h\cdot u^{\varepsilon}_h
		+ (1+\varrho^{\varepsilon})u^{\varepsilon}\cdot
		\nabla u^{\varepsilon}_3
		\Bigr) \\
		&- \partial_3\varrho^{\varepsilon}\,{\rm div}\,u^{\varepsilon}
		- \partial_3 u^{\varepsilon}\cdot\nabla\varrho^{\varepsilon}
		- \frac{1}{1+\varepsilon}
		\bigl((1+\varrho^{\varepsilon})^3 - 1\bigr)
		\partial_3\varrho^{\varepsilon}.
	\end{aligned}
\end{equation}
The first line on the right of \eqref{25} contains the
linear terms $(1+\varrho^\varepsilon)\partial_t u_3^\varepsilon$,
$\Delta_h u_3^\varepsilon$, and
$\varepsilon\,\partial_3\nabla_h\cdot u_h^\varepsilon$.
When applying the maximum principle to \eqref{25}, these terms
contribute quantities proportional to
$\|(w_h^\varepsilon, \partial_3\varrho^\varepsilon)\|_{L^\infty}$
to the right-hand side.
To absorb them, it is necessary that the power of
$\|(w_h^\varepsilon, \partial_3\varrho^\varepsilon)\|_{L^\infty}$
in the energy functional $\mathcal{E}(t)$ is strictly greater
than one; this is the reason for the cubic exponents
$\|w_h^\varepsilon\|_{L^\infty}^3$ and
$\|\partial_3\varrho^\varepsilon\|_{L^\infty}^3$ in the
definition of $\mathcal{E}(t)$.
With this choice, the maximum principle applied to \eqref{25}
yields the uniform estimate for
$\|\partial_3\varrho^\varepsilon\|_{L^\infty}$.

\textbf{Step 4: Tangential energy estimates and two-tier
	time decay.}
With the $L^\infty$ bounds from Steps~2 and~3 available
as bootstrap assumptions, we turn to the energy
estimates that drive the decay.
The argument proceeds in three substeps.

\emph{(4a) Tangential and vertical derivative estimates.}
We apply conormal vector fields $Z^{\alpha_h}$ with
$|\alpha_h| \le m$ to the system \eqref{eqr} and derive
energy identities for the tangential norms
$\|(\varrho^\varepsilon, u^\varepsilon)\|_{H^m_{tan}}^2$
and $\|\partial_3(\varrho^\varepsilon, u^\varepsilon)
\|_{H^{m-1}_{tan}}^2$.
The nonlinear terms are controlled by the anisotropic
Sobolev inequality, using the vertical derivative estimate
from Step~1 and the $L^\infty$ bounds from Steps~2--3.

\emph{(4b) Enhanced dissipation for the density.}
The density equation alone provides no direct dissipation
for $\nabla\varrho^\varepsilon$.
However, taking the inner product of the momentum equation
$\eqref{eqr}_2$ with $\nabla\varrho^\varepsilon$ and using
the pressure term, one recovers the dissipation estimate
$\|\nabla\varrho^\varepsilon\|_{H^{m-1}_{tan}}^2$, which
is indispensable for closing the energy.

\emph{(4c) Two-tier decay.}
Combining substeps (4a)--(4b), and incorporating the
negative-order norm
$\|\Lambda_h^{-(1-\zeta)}(\varrho^\varepsilon,
u^\varepsilon)\|_{L^2}^2$ into a modified energy
$\widehat{\mathcal{E}}_{tan}(t)
\sim \mathcal{E}_{tan}(t)$,
one derives the differential inequality
\begin{equation}\label{202}
	\frac{d}{dt}\widehat{\mathcal{E}}_{tan}(t)
	+ \kappa\,\mathcal{D}_{tan}(t) \le 0.
\end{equation}
The key step is to pass from \eqref{202} to a
nonlinear inequality.
Using the interpolation between the negative-order norm
and $\mathcal{E}_{tan}(t)$, one shows that
$\mathcal{D}_{tan}(t)
\ge \kappa C_0^{-1/(1-\zeta)}
\widehat{\mathcal{E}}_{tan}(t)^{1+1/(1-\zeta)}$,
where $C_0 := \|\Lambda_h^{-(1-\zeta)}(\varrho_0,
u_0)\|_{L^2}^2$.
Substituting into \eqref{202} gives
\begin{equation}\label{202b}
	\frac{d}{dt}\widehat{\mathcal{E}}_{tan}(t)
	+ \kappa\,C_0^{-\frac{1}{1-\zeta}}
	\widehat{\mathcal{E}}_{tan}(t)^{1+\frac{1}{1-\zeta}}
	\le 0,
\end{equation}
from which the ODE comparison principle yields the
algebraic decay
\begin{equation*}
	\widehat{\mathcal{E}}_{tan}(t)
	\lesssim C_0(1+t)^{-(1-\zeta)}.
\end{equation*}
Inserting this back into \eqref{202} and integrating with
the time weight $(1+t)^{1-\zeta}$ establishes the first
tier of the decay estimate:
\begin{equation}\label{203}
	(1+t)^{1-\zeta}\mathcal{E}_{tan}(t)
	+ \int_0^t(1+\tau)^{1-2\zeta}
	\mathcal{D}_{tan}(\tau)\,d\tau
	\lesssim C_0.
\end{equation}
To obtain the faster decay of
$\bar{\mathcal{E}}_{tan}(t)
= \|\nabla_h(\varrho^\varepsilon, u^\varepsilon)
\|_{H^{m-2}_{tan}}^2
+ \|\nabla_h\partial_3(\varrho^\varepsilon, u^\varepsilon)
\|_{H^{m-3}_{tan}}^2$,
we repeat the energy argument at this higher level and
use \eqref{203} to absorb the lower-order remainders.
This yields the modified differential inequality
\begin{equation*}
	\frac{d}{dt}\widehat{\bar{\mathcal{E}}}_{tan}(t)
	+ \kappa\,\bar{\mathcal{D}}_{tan}(t) \le 0,
	\quad
	\widehat{\bar{\mathcal{E}}}_{tan}(t)
	\sim \bar{\mathcal{E}}_{tan}(t),
\end{equation*}
and, integrating with the weight $(1+t)^{2-2\zeta}$ and
using \eqref{203}, the second tier of the decay estimate:
\begin{equation}\label{24}
	(1+t)^{2-2\zeta}\bar{\mathcal{E}}_{tan}(t)
	+ \int_0^t(1+\tau)^{2-2\zeta}
	\bar{\mathcal{D}}_{tan}(\tau)\,d\tau
	\lesssim C_0.
\end{equation}
The two-tier decay estimates \eqref{203}--\eqref{24}
supply exactly the $L^1$-in-time integrability
exploited in Steps~2--3 above.
Inserting \eqref{203}--\eqref{24} back into the estimates
of Steps~1--3 within the bootstrap scheme of
Proposition~\ref{main_pro} closes the argument and
establishes the global-in-time uniform estimates
\eqref{uniform_estimate}--\eqref{decay_estimate}
with constants independent of $t$ and $\varepsilon$.
The rigorous implementation is carried out in
Section~\ref{global-estimate}.

\vskip .2in        
\section{Global in time uniform regularity}\label{global-estimate}
			
	In this section, we establish the global-in-time well-posedness of system \eqref{eqr} in conormal Sobolev spaces under the smallness assumption \eqref{condition} on the initial data. Owing to the horizontally dissipative structure of the equations, one can first derive a local-in-time existence and uniqueness theory with estimates that are uniform with respect to $\var$. The main task is then to extend these local solutions globally in time by establishing suitable uniform a priori estimates. Therefore, the primary objective of this section is to prove global-in-time uniform regularity for solutions arising from small initial data satisfying \eqref{condition}. For notational simplicity, we suppress the superscript $\var$ throughout this section. We now state the corresponding global uniform estimate.
	
			\begin{prop}\label{main_pro}
			Let $m \ge 5$ be an integer  and constants
				$(\sigma, s)$ satisfy $\frac{16}{17}\le \sigma<s<1$, assume the initial data $(\vro_0, u_0)$ satisfying the condition stated in Theorem \ref{main_result_one}.
				For the solution $(\vro,u)$ of equation \eqref{eqr}
				defined on $ [0,T] \times \mathbb{R}_+^3$, assume there exists
				a small positive constant $\delta$ such that for any $t \in (0, T]$
				\begin{equation}\label{assumption} \bal
					&\; \underset{0\le \tau \le t}{\sup}\mathcal{E}(\tau)
					+\underset{0\le \tau \le t}{\sup}[(1+\tau)^s \mathcal{E}_{tan}(\tau)] +\underset{0\le \tau \le t}{\sup}[(1+\tau)^{1+\sigma} \bme_{tan}(\tau) ]
					 \\
                    &\;+\int_0^t \md(\tau)d\tau  +\int_0^t (1+\tau)^{\sigma} \md_{tan}(\tau)d\tau
					+ \int_0^t (1+\tau)^{1+\sigma} \bmd_{tan}(\tau)d\tau \le \delta,
				\dal \end{equation}
				then the  solution of equation \eqref{eqr} has the estimate
				\begin{equation}\label{close_assumption}\bal
                &\; \underset{0\le \tau \le t}{\sup}\mathcal{E}(\tau)
					+\underset{0\le \tau \le t}{\sup}[(1+\tau)^s \mathcal{E}_{tan}(\tau)] +\underset{0\le \tau \le t}{\sup}[(1+\tau)^{1+\sigma} \bme_{tan}(\tau) ]
					 \\
                    &\;+\int_0^t \md(\tau)d\tau  +\int_0^t (1+\tau)^{\sigma} \md_{tan}(\tau)d\tau
					+ \int_0^t (1+\tau)^{1+\sigma} \bmd_{tan}(\tau)d\tau  \le \frac{\delta}{2},
				\dal \end{equation}
				where the small positive constant $\delta:=12 C \mathcal{E}(0)$ and $s:=1-\zeta$, $\sigma:=1-2\zeta$.
			Here $C$ is a positive constant independent of time $t$ and
				parameter $\var$.
			\end{prop}

Since the proof of Proposition~\ref{main_pro} is rather lengthy and technically involved, we first provide a brief overview of its structure. The proof is divided into seven groups of lemmas, each devoted to estimating a specific component of the energy functional $\mathcal{E}(t)$. We now summarize the role of each group and explain the logical order in which the estimates are derived.

\medskip
\noindent\textbf{Subsection~\ref{sec:pp}:  Preparatory tools (Lemma~\ref{lemma-help}).}
Before entering the main energy argument, we record two
preparatory estimates that will be used repeatedly throughout
the section.
By expressing $\partial_3^2 u_3$ and $\varepsilon\partial_3^2 u_h$
directly from the momentum equation, we show that
$\|\partial_3^2 u_3\|_{H^2_{tan}}^2
+ \varepsilon\|\partial_3^2 u_h\|_{H^2_{tan}}^2
\lesssim \mathcal{E}(t)$,
which upgrades the available vertical regularity
without paying extra dissipation.
We also derive $H^k_{co}$ and $H^k_{tan}$ bounds for
$\partial_t\varrho$ and $\partial_3\partial_t\varrho$
from the density equation, eliminating time derivatives of
the density from subsequent nonlinear estimates.

\medskip
\noindent\textbf{Subsection~\ref{sec-tan-vertical}:
	Tangential and vertical derivative estimates
	(Lemmas~\ref{lemma32}--\ref{lemma34}).}
This subsection establishes the core energy identities
for the conormal norm
$\|(\varrho, u)\|_{H^m_{co}}^2
+ \|\partial_3(\varrho, u)\|_{H^{m-1}_{co}}^2$,
which controls the first two components of $\mathcal{E}(t)$.

\emph{Lemma~\ref{lemma32}} derives the differential inequalities
\eqref{3101-1}--\eqref{3101-2} for the tangential norms
$\|(\varrho, u)\|_{H^m_{tan}}^2$ and
$\|\nabla_h(\varrho, u)\|_{H^{m-2}_{tan}}^2$.
The key step is to control the nonlinear commutator terms using
the anisotropic Sobolev inequality, the interpolation
inequalities \eqref{inter-1}--\eqref{inter-2}, and the
condition $s \ge \frac{16}{17}$; the resulting bounds involve
$\sqrt{\mathcal{E}(t)}\bar{\mathcal{D}}_{tan}(t)$,
which is the correct structure for the two-tier decay argument.

\emph{Lemma~\ref{lemma33}} handles the one-order vertical
derivative estimates for the tangential norms,
i.e.\ the differential inequalities for
$\|\partial_3(\varrho, u)\|_{H^{m-1}_{tan}}^2$ and
$\|\nabla_h\partial_3(\varrho, u)\|_{H^{m-3}_{tan}}^2$.
The principal difficulty is the term $II_{1,4}$ (respectively
$II_{3,5}$), involving $\partial_3^2\varrho$ (respectively
$\partial_3^2 u_h$), which lies outside the energy framework.
This term is handled by decomposing the domain into a region
near the boundary (where $u_3 = \int_0^{x_3}\partial_3 u_3\,dz$
provides extra decay) and a region away from the boundary
(where the conormal norm is equivalent to the standard Sobolev
norm), yielding the critical estimate \eqref{est-s-16/17} and
the constraint $s \ge \frac{16}{17}$.

\emph{Lemma~\ref{lemma34}} establishes the full conormal estimate
$\|\partial_3(\varrho, u)\|_{H^k_{co}}^2$ for $1\le k \le m-1$,
where the multi-index $\alpha$ includes vertical components.
The commutator $[Z^\alpha, \partial_3]$ introduces terms
proportional to $Z_3^{\beta_3}(1/\varphi)$, bounded by $1/x_3$.
These are absorbed using the Hardy-type inequality
$u_3 = \int_0^{x_3}\partial_3 u_3\,dz$ near the boundary
and the equivalence of conormal and standard norms away from it.

\medskip
\noindent\textbf{Subsection~\ref{sec-density}:
	Dissipation estimate for the density
	(Lemmas~\ref{lemma35}--\ref{lemma37}).}
The density equation provides no direct dissipation for
$\nabla\varrho$.
The three lemmas in this subsection recover it via the
momentum equation.

\emph{Lemma~\ref{lemma35}} establishes the tangential dissipation
estimate $\|Z^{\alpha_h}\nabla\varrho\|_{L^2}^2$
for $|\alpha_h| \le m-1$ by taking the inner product of the
momentum equation with $Z^{\alpha_h}\nabla\varrho$,
using the time derivative of the density equation to rewrite
the resulting $\partial_t$ term as a full time derivative of
$\int Z^{\alpha_h} u \cdot Z^{\alpha_h}\nabla\varrho\,dx$
plus controlled remainders.
The estimate splits into four cases ($|\alpha_h| = 0, 1,
m-2, m-1$) according to whether the nonlinear terms require
the tangential or the higher-order tangential dissipation.

\emph{Lemma~\ref{lemma36}} extends the density dissipation
to the full conormal norm $\|Z^\alpha\nabla\varrho\|_{L^2}^2$
for $|\alpha| \le m-1$, including vertical components.
When $\alpha_3 \ge 1$, the commutator $[Z^\alpha, \partial_3]$
must be handled via the property \eqref{properity-p3},
and the resulting boundary terms are controlled using
the Hardy estimate and Lemma~\ref{lemma-help}.

\emph{Lemma~\ref{lemma37}} controls
$\varepsilon\|\partial_3^2\varrho\|_{H^2_{co}}^2$,
the last component of $\mathcal{E}(t)$.
Taking the divergence of the momentum equation and substituting
the density equation gives an ODE-type equation for
$\partial_3^2\ln(1+\varrho)$; multiplying by
$\varepsilon Z^\alpha\partial_3^2\varrho$
and integrating yields the estimate \eqref{3901},
whose right-hand side is controlled by the time derivative
estimates established in the next subsection.

\medskip
\noindent\textbf{Subsection~\ref{lines}}:
\noindent\textbf{$L^\infty$ estimates
	(Lemmas~\ref{lemma-wh}--\ref{lemma39}).}
These two lemmas provide the $L^\infty$ bounds for the horizontal
vorticity and the vertical density derivative,
which appear with cubic powers in $\mathcal{E}(t)$.
The maximum principle applied to the $w_h$-equation
gives an integral inequality for $\|w_h(t)\|_{L^\infty}$;
the four integral terms on the right-hand side of \eqref{2222}
are bounded by $\delta^{5/6}$ using the two-tier decay from
the tangential estimates of the previous step.
The key inputs are the $L^1$-in-time bounds on
$\|\partial_t u_3\|_{L^\infty}$ and
$\|\nabla_h\varrho\|_{L^\infty}$, controlled by the time
decay rates \eqref{203}--\eqref{24}.
Similarly, the maximum principle applied to equation \eqref{25}
for $\partial_3\varrho$ gives the bound \eqref{vro01};
the damping term $\frac{1}{1+\varepsilon}\partial_3\varrho$
on the left-hand side of \eqref{25} ensures global integrability
of the forcing term $G(t)$ in time.

\medskip
\noindent\textbf{Subsection~\ref{timer}:
	Time derivative estimates (Lemma~\ref{lemma310}).}
This subsection establishes the three differential inequalities
\eqref{3601}--\eqref{3801} for the time derivative components
$\|\partial_t u\|_{H^2_{tan}}^2$,
$\|\partial_t w_h\|_{H^2_{tan}}^2 + \|Z_3\partial_t w_h\|_{L^2}^2$,
and
$\|\partial_t\partial_3 u_3\|_{H^2_{tan}}^2
+ \|Z_3\partial_t\partial_3 u_3\|_{L^2}^2$,
all of which appear in $\mathcal{E}(t)$.
These terms arise because $\varepsilon\partial_3^2 u_h$
and $\partial_3\,\mathrm{div}\,u$ can only be controlled at
the $L^\infty$ level by trading vertical derivatives for
time derivatives via the momentum equation (see Section~2).
The proofs differentiate the system in time, apply energy
estimates to the resulting linear equations,
and absorb the resulting commutator terms using
Lemma~\ref{lemma-help} and the bootstrap assumption
\eqref{assumption}.

\medskip
\noindent\textbf{Subsection~\ref{negg}: Negative Sobolev estimates (Lemmas~\ref{lemma311}--\ref{lemma312}).}
Lemma~\ref{lemma311} controls
$\|\Lambda_h^{-(1-\zeta)}(\varrho, u)(t)\|_{L^2}^2$
and Lemma~\ref{lemma312} controls
$\|\Lambda_h^{-(1-\zeta)}\partial_3(\varrho, u)(t)\|_{L^2}^2$.
These negative-order norms capture the low-frequency content of
the solution and are the essential input for deriving the
nonlinear decay inequality \eqref{202b}.
The Hardy--Littlewood--Sobolev inequality (Lemma~\ref{H-L})
is used to bound the nonlinear terms in $L^{2/(1+s)}$
and transfer them to the negative norm via duality.
The condition $\zeta \in (0,1)$ enters here to ensure
the Hardy--Littlewood--Sobolev inequality applies.

\medskip
\noindent\textbf{Subsection~\ref{dect}: Decay in time estimates
	(Lemmas~\ref{lemma313}--\ref{lemma315}).}
This subsection contains three lemmas that together drive the
time decay and assemble the full bootstrap bound.

\emph{Lemma~\ref{lemma313}} establishes the
global uniform bound
$\mathcal{E}(t) + \int_0^t \mathcal{D}(\tau)\,d\tau \lesssim C_0$,
where $C_0 := \mathcal{E}(0) + \delta^{4/3}$ is defined in
\eqref{co}.
The proof combines the conormal energy estimates from the tangential
and vertical derivative lemmas (via induction and the equivalence
\eqref{est-equivalence}), adds the density dissipation via the
density dissipation lemmas with a small coupling constant $\kappa_1$,
and then incorporates the time derivative estimates
\eqref{3601}--\eqref{3801}, the $L^\infty$ bounds
\eqref{wh01} and \eqref{vro01}, and the negative-order bounds
\eqref{J01}--\eqref{K01} to produce the combined bound
\eqref{31002}--\eqref{31003}.

\emph{Lemma~\ref{lemma314}} establishes the
first tier of time decay:
$(1+t)^s\mathcal{E}_{tan}(t)
+ \int_0^t(1+\tau)^\sigma\mathcal{D}_{tan}(\tau)\,d\tau
\lesssim C_0$.
The three differential inequalities \eqref{31102}--\eqref{31104}
are combined into a single inequality \eqref{31105} for the
modified energy $\widetilde{\mathcal{E}}_{tan}(t)
\sim \mathcal{E}_{tan}(t)$.
Under the bootstrap assumption \eqref{assumption} this simplifies
to the linear dissipation inequality \eqref{31106}:
$\frac{d}{dt}\widetilde{\mathcal{E}}_{tan}(t)
+ \kappa_2\mathcal{D}_{tan}(t) \le 0$.
The interpolation inequality \eqref{31107} then gives
$\widetilde{\mathcal{E}}_{tan}(t)
\lesssim C_0^{1/(1+s)}\mathcal{D}_{tan}(t)^{s/(1+s)}$,
converting \eqref{31106} into the nonlinear ODE
$\frac{d}{dt}\widetilde{\mathcal{E}}_{tan}
+ \kappa_2 C_0^{-1/s}
\widetilde{\mathcal{E}}_{tan}^{1+1/s} \le 0$.
The ODE comparison principle then yields the algebraic decay
\eqref{31108}: $\mathcal{E}_{tan}(t) \lesssim C_0(1+t)^{-s}$.
Multiplying \eqref{31106} by $(1+t)^\sigma$ and integrating
gives the weighted dissipation integral via \eqref{est-sigma-s},
completing the proof of \eqref{31101}.

\emph{Lemma~\ref{lemma315}} establishes the
faster second tier of time decay:
$(1+t)^{1+\sigma}\bar{\mathcal{E}}_{tan}(t)
+ \int_0^t(1+\tau)^{1+\sigma}\bar{\mathcal{D}}_{tan}(\tau)\,d\tau
\lesssim C_0$.
The higher-order differential inequalities
\eqref{31202}--\eqref{31204} are combined into a single inequality
for the modified energy $\widetilde{\bar{\mathcal{E}}}_{tan}(t)
\sim \bar{\mathcal{E}}_{tan}(t)$, satisfying
$\frac{d}{dt}\widetilde{\bar{\mathcal{E}}}_{tan}(t)
+ \kappa_3\bar{\mathcal{D}}_{tan}(t) \le 0$
under the bootstrap assumption.
Multiplying by $(1+t)^{1+\sigma}$ and integrating, and using
the bound $\widetilde{\bar{\mathcal{E}}}_{tan}(t)
\lesssim \mathcal{D}_{tan}(t)$ together with the first-tier
decay \eqref{31101} to control the right-hand side, yields
\eqref{31201}.

\noindent\textbf{Subsection~\ref{lastt}: Global in time uniform
	regularity (closing of Proposition~\ref{main_pro} and proof of
	Theorem~\ref{main_result_one}).}
With Lemmas~\ref{lemma313}--\ref{lemma315} in hand, adding their
three estimates gives
\begin{equation*}
	\begin{aligned}
		&\sup_{0\le\tau\le t}\mathcal{E}(\tau)
		+ \sup_{0\le\tau\le t}\big[(1+\tau)^s\mathcal{E}_{tan}(\tau)\big]
		+ \sup_{0\le\tau\le t}\big[(1+\tau)^{1+\sigma}\bar{\mathcal{E}}_{tan}(\tau)\big] \\
		&\quad
		+ \int_0^t\mathcal{D}(\tau)\,d\tau
		+ \int_0^t(1+\tau)^\sigma\mathcal{D}_{tan}(\tau)\,d\tau
		+ \int_0^t(1+\tau)^{1+\sigma}\bar{\mathcal{D}}_{tan}(\tau)\,d\tau  \\
		&\le 3C\big(\mathcal{E}(0)+\delta^{4/3}\big).
	\end{aligned}
\end{equation*}
where $(\sigma,s)$ satisfy $\frac{16}{17}\le\sigma<s<1$.
Choosing $\delta = 12C\mathcal{E}(0) \le \min\{1,(12C)^{-3}\}$
ensures $3C\delta^{4/3} \le \delta/4$
and $3C\mathcal{E}(0) \le \delta/4$,
so the left-hand side is bounded by $\delta/2$.
This is the improved estimate \eqref{close_assumption}, which
is strictly sharper than the bootstrap assumption
\eqref{assumption} by a factor of $2$, closing the bootstrap.

\medskip
The proof of Theorem~\ref{main_result_one} then follows by
a standard continuity argument.
Define $T^*$ as the supremum over all $T_0$ such that
\eqref{assumption} holds on $[0,T_0)$ (see \eqref{criterion}).
Local well-posedness gives \eqref{assumption} for small $t$,
so $T^* > 0$.
If $T^* < +\infty$, Proposition~\ref{main_pro} gives
\eqref{close_assumption} on $[0,T^*)$, and local existence
extends the solution to $[0,T^{**})$ for some $T^{**}>T^*$
with \eqref{assumption} still holding, contradicting the
maximality of $T^*$.
Hence $T^* = +\infty$ and the global solution satisfies
\eqref{uniform_estimate}--\eqref{decay_estimate} for all $t\ge 0$,
uniformly in $\varepsilon\in(0,1)$.

\subsection{Preparatory tools} \label{sec:pp}

\begin{lemm}\label{lemma-help}
    For any smooth solution $(\vro,u)$ of equation \eqref{eqr}, under the assumption \eqref{assumption},
				it holds,
                \beq\label{est-p32-u3-uh} \bal
\|\p_3^2 u_3\|_{H_{tan}^2}^2
+ \var \|\p_3^2 u_h\|_{H_{tan}^2}^2 \lesssim \me(t).
                \dal \deq
               Moreover, for positive integer $ 0 \le k \le 2$, 
               \beq \label{pro-ptvro} \bal
\| \p_t \vro\|_{H_{co}^k}^2 
\lesssim  \| \du\|_{H_{co}^k}^2 + \me(t) (\|\du \|_{H_{co}^k}^2 + \|\nabla \vro  \|_{H_{co}^k}^2),  ~~~\| \p_t \vro\|_{H_{tan}^k}^2 
\lesssim \me(t),
              \dal \deq
              and
                \beq \label{pro-p3tvro} \bal
               \| \p_3 \p_t \vro\|_{H_{co}^k}^2 
\lesssim &\; \| \p_3  \du\|_{H_{co}^k}^2 + \me(t) (\|(\du, \p_3 \du)\|_{H_{co}^k}^2 + \|\nabla \vro  \|_{H_{co}^{k+1}}^2), 
 ~~~ \| \p_3 \p_t \vro\|_{H_{tan}^k}^2 
\lesssim  \me(t).
              \dal \deq
\end{lemm}
\begin{proof}
Due to the relation $\p_3 \du = \p_3^2 u_3 + \p_3 \nabla_h \cdot u_h$, we can check that
\beq\label{equ-p32u3} \bal
(1+\var) \p_3^2 u_3 &= ( 1+ \vro) \p_t u_3 + ( 1+ \vro) u \cdot \nabla u_3 - \Delta_h u_3 -\p_3 \nabla_h \cdot u_h + ( 1+ \vro)^2 \p_3 \vro,\\
\var \p_3^2 u_h & =  ( 1+ \vro) \p_t u_h + ( 1+ \vro) u \cdot \nabla u_h - \Delta_h u_h  - \nabla_h \du +( 1+ \vro)^2 \nabla_h \vro.
\dal \deq
Then, applying  the anisotropic
				type inequality \eqref{ie:Sobolev}, we have
		\beqq \bal
		 \|\p_3^2 u_3\|_{H_{tan}^2}^2 \lesssim &\; \| ( 1+ \vro) \p_t u_3\|_{H_{tan}^2}^2 +  \| ( 1+ \vro) u \cdot \nabla u_3 \|_{H_{tan}^2}^2 +\| \Delta_h u_3\|_{H_{tan}^2}^2 \\&\; + \|\p_3 \nabla_h \cdot u_h\|_{H_{tan}^2}^2 
		 + \| ( 1+ \vro)^2 \p_3 \vro\|_{H_{tan}^2}^2 \lesssim  \me(t),
		\dal \deqq 
		and similarly, we have
		\beqq 
		 \var \|\p_3^2 u_h\|_{H_{tan}^2}^2 \lesssim \me(t).
		\deqq
        Thus we finish the proof of \eqref{est-p32-u3-uh}.  Next, with the help of \eqref{est-p32-u3-uh}, using the equation $\eqref{eqr}_1$ and the anisotropic inequalities \eqref{ie:Sobolev}, we can obtain the estimates
        \eqref{pro-ptvro} and \eqref{pro-p3tvro}, here we omit the proof for simplicity.
 		Thus, we finish the proof of this lemma. 
\end{proof}

			\subsection{Tangential and vertical derivative estimates}\label{sec-tan-vertical}
            It is easy to check that
                \beq\label{est-equivalence} \bal
              \|(\vro,u)\|_{H_{co}^m}^2 + \| \p_3 (\vro, u)(t)\|_{H_{co}^{m-1}}^2 \thicksim  &\; \|(\vro,u)\|_{H_{tan}^m}^2 + \| \p_3 (\vro, u)(t)\|_{H_{co}^{m-1}}^2 \\
               \thicksim  &\; \|(\vro,u)\|_{L^2}^2 + \|(\vro,u)\|_{\dot{H}_{tan}^m}^2
                  + \| \p_3 (\vro, u)(t)\|_{H_{co}^{m-1}}^2.
               \dal \deq
            Therefore, in this subsection, we will establish the estimate for the tangential derivative and one-order vertical derivative estimate of density and velocity field.
			First of all, we establish the estimate for the tangential derivative
			of density and velocity field, and in order to obtain the faster decay of the solution, we will estimate the higher tangential derivative at the same time, see the following lemma in detail.
			\begin{lemm}\label{lemma32}
				For any smooth solution $(\vro,u)$ of equation \eqref{eqr},
				it holds that
				\beq\label{3101-1} \bal
				&\; \f12 \frac{d}{dt} (\|(\vro,u)(t)\|_{L^2}^2 +\|(\vro, u)(t)\|_{\dot{H}_{tan}^m}^2 )
				+\|(\nabla_h u, \du)(t)\|_{H^m_{tan}}^2 +\var \|\p_3 u(t)\|_{H^m_{tan}}^2
				\lesssim \sqrt{\me(t)}\md_{tan}(t),
                \dal \deq
                and 
                \beq\label{3101-2} \bal
				&\; \f12 \frac{d}{dt} (\|\nabla_h(\vro,u)(t)\|_{L^2}^2 + \|\nabla_h (\vro, u)(t)\|_{\dot{H}_{tan}^{m-2}}^2 ) \\
				&+\|\nabla_h(\nabla_h u, \du)(t)\|_{H^{m-2}_{tan}}^2
                +\var \|\nabla_h \p_3  u(t)\|_{H^{m-2}_{tan}}^2
				\lesssim 
                \sqrt{\me(t)}\bar{\md}_{tan}(t).
				\dal \deq
			\end{lemm}
			\begin{proof}
				For any $|\al_h|=k\le m$,  the equations $\eqref{eqr}$ yields directly
				\begin{equation} \label{31-equ}
					\begin{aligned}
						&\frac{d}{dt}\frac{1}{2}\i(|Z^{\ah} \vro|^2+|Z^{\ah}u|^2)dx
						+\|\nabla_h Z^{\ah}u\|_{L^2}^2 +\| Z^{\ah} \du\|_{L^2}^2 + \var \|\p_3 Z^{\ah}u\|_{L^2}^2 \\
						=&\i Z^{\ah}(-u\cdot \nabla \vro -\vro \, \du )\cdot Z^{\ah} \vro \ dx
						-\i Z^{\ah} (\vro\, \nabla \vro)\cdot Z^{\ah} u\ dx\\
						& -\i Z^{\ah} (u\cdot \nabla u)\cdot Z^{\ah} u\ dx-\i Z^{\ah} (\f{\vro}{1+\vro} \Delta_h u)\cdot Z^{\ah} u\ dx  \\
                        &-\i Z^{\ah} (\f{\vro}{1+\vro} \nabla \du)\cdot Z^{\ah} u\ dx -\var \i Z^{\ah} (\f{\vro}{1+\vro}  \p_3^2 u)\cdot Z^{\ah} u\ dx\\
				:= & \sum_{i=1}^{6} I_i,
					\end{aligned}
				\end{equation}
				where we have used the basic fact
				\beqq \bal
				& \i Z^{\ah} \nabla \vro \cdot Z^{\ah} u \ dx
				+\i Z^{\ah} \du \cdot Z^{\ah} \vro \ dx=0.
				\dal \deqq
				If $|\alpha_h|=0$, integrating by part,  we can obtain
				\beqq \bal
				I_1 + I_2 & = - \i u \cdot \nabla( \vro^2) \ dx- \i \vro^2 \, \du \ dx =0,
				\dal \deqq
				and  using the anisotropic type inequality \eqref{ie:Sobolev}, we have
				\beqq \bal
				I_3 &= \f12 \i |u|^2 \du \ dx \lesssim \| u\|_{L^2} \| \nabla_h u\|_{L^2} \| \du\|_{L^2}^{\f12}   \|\p_3 \du\|_{L^2}^{\f12}  \lesssim  \sqrt{\me(t)}\md_{tan}(t),\\
				I_4+I_5+I_6 & \lesssim ( \| \Delta_h u\|_{L^2} + \var \| \p_3^2 u\|_{L^2} +  \| \nabla \du \|_{L^2}) \| \vro\|_{L^2}^{\f12}   \|\p_3 \vro\|_{L^2}^{\f12} \| \du\|_{L^2}^{\f12}   \|u\|_{L^2}^{\f14}  \| \nabla_h u\|_{L^2}^{\f12}   \|\nabla_h^2 u\|_{L^2}^{\f14} \\
				& \lesssim  \sqrt{\me(t)}\md_{tan}(t).
				\dal \deqq
				Thus we have
				\begin{equation}\label{3102} \bal
					&\f12 \frac{d}{dt} \|(\vro, u)(t)\|_{L^2}^2
				+\|(\nabla_h u, \du)(t)\|_{L^2}^2 +\var \|\p_3 u(t)\|_{L^2}^2 \lesssim \sqrt{\me(t)}\md_{tan}(t).
				\dal \end{equation}
				\textbf{Now let us deal with the case $|\alpha_h| = 1$}.
				Integrating by part, we have
			   \beqq \bal
			   & I_1 + I_2 \\
               =& - \i \nabla_h u\cdot \nabla \vro \cdot\nabla_h \vro \ dx + \f12 \i \du \, |\nabla_h \vro|^2 \ dx - \i (\nabla_h \vro \, \du+ \vro \nabla_h \du) \cdot \nabla_h \vro \ dx \\
               &\; + \f12 \i \nabla_h ( \vro^2) \nabla_h \du \ dx\\
			    =& - \i \nabla_h u_h \, |\nabla_h \vro|^2 \ dx - \i \nabla_h u_3 \, \p_3 \vro \cdot \nabla_h \vro \ dx - \f12 \i \du \, |\nabla_h \vro|^2 \ dx.
			   \dal \deqq
			   For any function $f(x)$, we can easily obtain the following interpolation inequality
			   \beq \label{inter-1}
			   \| \nabla_h f\|_{L^2} \lesssim   \| f\|_{L^2}^{\f12}   \| \nabla_h^2 f\|_{L^2}^{\f12},
			   \deq
			   which yields that
			   \beqq \bal
			    \i \nabla_h u_h\, |\nabla_h \vro|^2 \ dx  \lesssim  &\;  \|\nabla_h \vro\|_{L^2} \|\p_3 \nabla_h \vro\|_{L^2}^{\frac12} \|\p_1 \nabla_h \vro\|_{L^2}^{\frac12}
					\|\nabla_h u\|_{L^2}^{\frac12}
					\|\p_2 \nabla_h u\|_{L^2}^{\frac12}\\
					\lesssim  &\; \| \vro\|_{L^2}^{\f12}   \|\nabla_h^2 \vro\|_{L^2}^{\f12} (\|\p_3  \vro\|_{L^2}^{\frac12} \|\p_3 \nabla_h^2 \vro\|_{L^2}^{\frac12})^{\f12} \|\p_1 \nabla_h \vro\|_{L^2}^{\frac12} (\|u\|_{L^2}^{\frac12} \|\nabla_h^2 u\|_{L^2}^{\frac12})^{\frac12}
					\|\p_2 \nabla_h u\|_{L^2}^{\frac12}\\
					\lesssim &\; \sqrt{\me(t)}\bmd_{tan}(t).
			   \dal  \deqq
			   Similarly, we have
			   	   \beqq \bal
			   	   \i \nabla_h u_3 \, \p_3 \vro \cdot\nabla_h \vro \ dx
			   	   \lesssim  &\; \|\nabla_h u_3\|_{L^2}^{\frac12}
					\|\p_3 \nabla_h u_3\|_{L^2}^{\frac12} \|\p_3 \vro\|_{L^2}^{\f12} \|\p_{13} \vro\|_{L^2}^{\frac12} \|\nabla_h \vro\|_{L^2}^{\frac12}  \|\p_2 \nabla_h \vro\|_{L^2}^{\frac12}
					\\
					 \lesssim  &\; ( \| u_3\|_{L^2}^{\frac12} \|\nabla_h^2 u_3\|_{L^2}^{\frac12})^{\frac12}
					\|\p_3 \nabla_h u_3\|_{L^2}^{\frac12} \|\p_3 \vro\|_{L^2}^{\f12} \|\p_{13} \vro\|_{L^2}^{\frac12} (\| \vro\|_{L^2}^{\frac12} \| \nabla_h^2 \vro\|_{L^2}^{\frac12})^{\frac12} \|\p_2 \nabla_h \vro\|_{L^2}^{\frac12}
					\\
					\lesssim &\; \sqrt{\me(t)}\bmd_{tan}(t),\\
					 \i \du \, |\nabla_h \vro|^2 \ dx  \lesssim &\; \|\nabla_h \vro\|_{L^2} \|\p_3 \nabla_h \vro\|_{L^2}^{\frac12} \|\p_1 \nabla_h \vro\|_{L^2}^{\frac12}
					\|\du \|_{L^2}^{\frac12}
					\| \p_2 \du \|_{L^2}^{\frac12}\\
					\lesssim &\; \|\vro\|_{L^2}^{\f12} \|\nabla_h^2 \vro\|_{L^2}^{\f12} \|\p_3 \nabla_h \vro\|_{L^2}^{\frac12} \|\p_1 \nabla_h \vro\|_{L^2}^{\frac12}
					\|\du \|_{L^2}^{\frac12}
					\| \p_2 \du \|_{L^2}^{\frac12}\\
					\lesssim &\; \sqrt{\me(t)}\bmd_{tan}(t).
			   	    \dal  \deqq
			   	    Thus, combining the above estimates, we have
			   	    \beqq
			   	    I_1 +I_2 \lesssim \sqrt{\me(t)}\bmd_{tan}(t).
			   	    \deqq
			   	    Next, we estimate the term $I_3$. Integrating by part, we have
			   	    \beqq \bal
			   	    I_3  = \;  -  \i \nabla_h u_h \,|\nabla_h u|^2 \ dx - \i \nabla_h u_3 \, \p_3 u \cdot \nabla_h u \ dx +  \f12 \i \du \,|\nabla_h u|^2\ dx
			   	     = \; \sum_{i=1}^{3} I_{3,i}.
			   	      \dal \deqq
			   	     On one hand, using the interpolation inequality \eqref{inter-1} again, it is easy to check that
			   	        \beqq \bal
			   	        I_{3,1} \lesssim &\; \|  \nabla_h u\|_{L^2}^{\f32} \|\p_1 \nabla_h u_h\|_{L^2}^{\f12}  \|  \p_2 \nabla_h u\|_{L^2}^{\f12}  \|  \p_3 \nabla_h u\|_{L^2}^{\f12}\\
			   	      \lesssim &\; (\| u\|_{L^2}^{\f12}  \|  \nabla_h^2 u\|_{L^2}^{\f12})^{\f23} \| \nabla_h^2 u\|_{L^2} (\|  \p_3  u\|_{L^2}^{\f12} \|  \p_3 \nabla_h^2 u\|_{L^2}^{\f12})^{\frac12}\\
					\lesssim &\; \sqrt{\me(t)}\bmd_{tan}(t).
			   	    \dal \deqq
			   	    On the other hand, for any smooth function $f(x)$,
                    due to the interpolation inequalities
			   	    \beq \label{inter-2} \bal
			   	    \|  \nabla_h f\|_{L^2} \lesssim &\; \| \hs f\|_{L^2}^{\f{1}{s+2}} \| \nabla_h^2  f\|_{L^2}^{\f{s+1}{s+2}},\\
			   	     \|  f\|_{L^2} \lesssim &\; \| \hs f\|_{L^2}^{\f{2}{s+2}} \| \nabla_h^2  f\|_{L^2}^{\f{s}{s+2}},
			   	   \dal \deq
			   	     we can estimate the terms $I_{3,2}$ and $I_{3,3}$ as follows:
			   	    \beqq \bal
			   	        I_{3,2} \lesssim &\;  \|  \nabla_h  u\|_{L^2} \|  \p_2 \nabla_h u\|_{L^2}^{\f12} \|\p_3 \nabla_h u_3\|_{L^2}^{\f12}  \|  \p_3 u\|_{L^2}^{\f12}    \|  \p_{13} u\|_{L^2}^{\f12}     \\
			   	      \lesssim &\; \| \hs u\|_{L^2}^{\f{1}{s+2}} \| \nabla_h^2  u\|_{L^2}^{\f{s+1}{s+2}}    \Big(\| \hs \p_3 u\|_{L^2}^{\f{2}{s+2}} \| \nabla_h^2 \p_3 u\|_{L^2}^{\f{s}{s+2}}  \Big)^{\f12}   \Big(\| \hs \p_3 u\|_{L^2}^{\f{1}{s+2}} \| \nabla_h^2 \p_3 u\|_{L^2}^{\f{s+1}{s+2}}  \Big)^{\f12}    \|(\p_2 \nabla_h u ,\p_3 \nabla_h u_3)\|_{L^2}\\
			   	      \lesssim &\; \| (\hs u, \hs \p_3 u)\|_{L^2}^{\f{5}{2(s+2)}} \| \nabla_h^2 (u, \p_3 u)\|_{L^2}^{\f{4s+3}{2(s+2)}} \|(\p_2 \nabla_h u ,\p_3 \nabla_h u_3)\|_{L^2}\\
                    \lesssim &\; \sqrt{\me(t)}\bmd_{tan}(t),
			   	    \dal \deqq
			   	    and
			   	     \beqq \bal
			   	        I_{3,3} \lesssim &\;  \|  \nabla_h  u\|_{L^2} \|\p_3 \nabla_h u\|_{L^2}^{\f12}   \|  \p_{1} \nabla_h u\|_{L^2}^{\f12}   \|  \du\|_{L^2}^{\f12}    \|  \p_2 \du\|_{L^2}^{\f12}\\
			   	      \lesssim &\;\Big( \| \hs (u,\p_3 u)\|_{L^2}^{\f{1}{s+2}} \| \nabla_h^2  (u,\p_3 u)\|_{L^2}^{\f{s+1}{s+2}} \Big)^{\f32}    \Big(\| \hs \p_3 u_3\|_{L^2}^{\f{2}{s+2}} \| \nabla_h^2 \p_3 u_3\|_{L^2}^{\f{s}{s+2}}  +  \| \hs u_h\|_{L^2}^{\f{1}{s+2}} \| \nabla_h^2 u_h\|_{L^2}^{\f{s+1}{s+2}} \Big)^{\f12}  \\
                      &\; \times \|(\p_1 \nabla_h u ,\p_2 \du)\|_{L^2}\\
                      \lesssim &\; \Big(\| (\hs u, \hs \p_3 u)\|_{L^2}^{\f{5}{2(s+2)}} \| \nabla_h^2 (u, \p_3 u)\|_{L^2}^{\f{4s+3}{2(s+2)}} + \| (\hs u, \hs \p_3 u)\|_{L^2}^{\f{2(s+1)}{s+2}} \| \nabla_h^2 (u, \p_3 u)\|_{L^2}^{\f{2}{s+2}} \Big) \\
                      &\; \times \|(\p_1 \nabla_h u ,\p_2 \du)\|_{L^2}\\
                    \lesssim &\; \sqrt{\me(t)}\bmd_{tan}(t),
			   	    \dal \deqq
                    where we have used  $s \ge \f{16}{17}$.
                    Combining the estimates from $I_{3,1}$ to $I_{3,3}$, we can obtain
                    \beqq \bal
			   	        I_{3}
					\lesssim \sqrt{\me(t)}\bmd_{tan}(t).
			   	    \dal \deqq
                    As for the term $I_5$, integrating by parts and using the anisotropic type inequalities $\eqref{ie:Sobolev}$, one can arrive that
                    \beqq \bal
                    I_5 = &\; - \i (\nabla_h (\f{\vro}{1+\vro}) \nabla \du + \f{\vro}{1+\vro} \nabla \nabla_h \du ) \cdot \nabla_h u \ dx\\
                    =  &\; \i \Big(\nabla_h u \cdot \nabla \nabla_h (\f{\vro}{1+\vro})  + \nabla_h(\f{\vro}{1+\vro})  \cdot \nabla_h \du   \Big) \, \du \ dx \\
                    &\;  +  \i \Big(\nabla_h u   \cdot \nabla (\f{\vro}{1+\vro}) + \f{\vro}{1+\vro} \nabla_h \du \Big) \cdot \nabla_h \du \ dx\\
                    \lesssim &\;  \| \du \|_{L^2}^{\f12} \| \p_1 \du \|_{L^2}^{\f12} \|\nabla \nabla_h (\f{\vro}{1+\vro})\|_{L^2} \| \nabla_h u \|_{L^2}^{\f14}
                    \| \p_3 \nabla_h u \|_{L^2}^{\f14}\| \p_2 \nabla_h u \|_{L^2}^{\f14} \| \p_{23} \nabla_h u \|_{L^2}^{\f14}\\
                    &\; + \| \du \|_{L^2}^{\f12} \| \nabla_h \du \|_{L^2}^{\f32} \| \nabla_h \vro \|_{L^2}^{\f14} \| \p_3 \nabla_h \vro \|_{L^2}^{\f14}
                    \| \p_2 \nabla_h \vro \|_{L^2}^{\f14} \| \p_{23} \nabla_h \vro \|_{L^2}^{\f14}\\
                    &\; + \| \nabla_h \du \|_{L^2} ( \| \nabla_h u \|_{L^2}^{\f14} \| \nabla_h^2 u \|_{L^2}^{\f12} \| \nabla_h^3 u \|_{L^2}^{\f14}\| \nabla_h \vro \|_{L^2}^{\f12} \| \p_3 \nabla_h \vro \|_{L^2}^{\f12}
                    \\
                    &\; + \| \nabla_h u_3 \|_{L^2}^{\f14}  \| \p_3 \nabla_h u_3 \|_{L^2}^{\f14} \| \p_1 \nabla_h u_3 \|_{L^2}^{\f14} \| \p_{13} \nabla_h u_3 \|_{L^2}^{\f14} \| \p_3 \vro \|_{L^2}^{\f12} \| \p_{13} \vro \|_{L^2}^{\f12} )
                     + \| \vro\|_{L^{\infty}} \| \nabla_h \du \|_{L^2}^{2}\\
                    \lesssim &\; \sqrt{\me(t)}\bmd_{tan}(t).
                    \dal \deqq
                    Similarly, integrating by parts, the terms $I_4$ and $I_6$ can be bounded by
                    \beqq \bal
                    I_4 = &\;  \i \f{\vro}{1+\vro} \Delta_h u\cdot \nabla_h^2 u\ dx  \lesssim \| \vro\|_{L^{\infty}} \|  \nabla_h^2 u \|_{L^2}^2 \lesssim \sqrt{\me(t)}\bmd_{tan}(t),\\
                    I_6 =&\; \var \i \p_3 u \cdot \p_3 (\nabla_h (\f{\vro}{1+\vro} ) \cdot \nabla_h u) \ dx
                    + \var \i \p_3 \nabla_h u \cdot \p_3 (\f{\vro}{1+\vro}  \,\nabla_h u) \ dx \\
                    \lesssim  &\; \sqrt{\me(t)}\bmd_{tan}(t).
                    \dal \deqq
                    Substituting the estimates from $I_1$ to $I_6$ into \eqref{31-equ}, we have
                   \begin{equation} \label{3103} \bal
					&\f12 \frac{d}{dt} \|\nabla_h(\vro, u)(t)\|_{L^2}^2
				+\|\nabla_h(\nabla_h u, \du)(t)\|_{L^2}^2 +\var \|\p_3 \nabla_h u(t)\|_{L^2}^2 \lesssim \sqrt{\me(t)}\bmd_{tan}(t).
				\dal \end{equation}
				\textbf{Next let us deal with the case $|\alpha_h| =m-1$}.
				Obviously, it holds
				\beqq \bal
				I_1
				= &\; -\sum_{0\le \beta_h \le \alpha_h}C^{\beta_h}_{\alpha_h}
				\i (Z^{\beta_h} u \cdot Z^{\al_h-\beta_h}\nabla \vro + Z^{\beta_h} \vro \cdot Z^{\al_h-\beta_h} \du)\cdot Z^{\al_h} \vro\ dx.
				\dal \deqq
				If $|\beta_h|=0$, integrating by part and using the divergence-free condition,
				we conclude
				\beqq \bal
				&-\i(u \cdot Z^{\al_h}\nabla \vro + \vro \, Z^{\al_h} \du)\cdot Z^{\al_h} \vro\ dx\\
				= &\; \frac{1}{2}\i |Z^{\al_h} \vro|^2 {\rm div} u\ dx - \i \vro \, Z^{\al_h} \du \cdot Z^{\al_h} \vro\ dx\\
                \lesssim &\;  \|\du\|_{L^{\infty}} \| Z^{\al_h} \vro\|_{L^2}^2 +  \|\vro\|_{L^{\infty}} \| Z^{\al_h} \vro\|_{L^2} \| Z^{\al_h} \du\|_{L^2} \\
                \lesssim &\; \sqrt{\me(t)}\bmd_{tan}(t).
				\dal \deqq
				If $\beta_h=\alpha_h$, the anisotropic type inequality \eqref{ie:Sobolev} yields directly
				\beqq
				\begin{aligned}
					&\i (Z^{\al_h} u \cdot \nabla \vro + Z^{\al_h} \vro \, \du )\cdot Z^{\al_h} \vro\ dx\\
					\lesssim
					&(\|Z^{\al_h} u\|_{L^2}^{\frac12}\|\p_3 Z^{\al_h} u\|_{L^2}^{\frac12}
					\|\nabla \vro\|_{H_{tan}^2}+  \|\du\|_{L^{\infty}} \|Z^{\al_h} \vro\|_{L^2})
					\|Z^{\al_h} \vro\|_{L^2}\\
					\lesssim
					&\sqrt{\me(t)}\bmd_{tan}(t).
				\end{aligned}
				\deqq
				If $1<|\beta_h|<|\alpha_h|$, we apply  the anisotropic
				type inequality \eqref{ie:Sobolev} to obtain
				\beqq
				\begin{aligned}
					&\i (Z^{\beta_h} u \cdot Z^{\al_h-\beta_h}\nabla \vro+ Z^{\beta_h} \vro \,Z^{\al_h-\beta_h}\du )\cdot Z^{\al_h} \vro \ dx\\
					\lesssim &\Big(\|Z^{\beta_h} u\|_{L^{2}}^{\f14} \|\p_3 Z^{\beta_h} u\|_{L^{2}}^{\f14} \|\p_1 Z^{\beta_h} u\|_{L^{2}}^{\f14} \|\p_{13} Z^{\beta_h} u\|_{L^{2}}^{\f14}
					\|Z^{\al_h-\beta_h}\nabla \vro\|_{L^2}^{\f12} \|\p_2 Z^{\al_h-\beta_h}\nabla \vro\|_{L^2}^{\f12}  \\
                    &\; + \| Z^{\beta_h}  \vro\|_{L^2}^{\f12}  \| \p_3 Z^{\beta_h}  \vro\|_{L^2}^{\f12}
					 \| Z^{\al_h -\beta_h}  \du\|_{H_{tan}^2}\Big) \|Z^{\al_h} \vro\|_{L^2}\\
					\lesssim
					&\sqrt{\me(t)}\bmd_{tan}(t),
				\end{aligned}
				\deqq
				and for $|\beta_h|=1$, we have
				\beqq
				\begin{aligned}
					&\i (Z^{\beta_h} u \cdot Z^{\al_h-\beta_h}\nabla \vro + Z^{\beta_h} \vro \,Z^{\al_h-\beta_h}\du )\cdot Z^{\al_h} \vro\ dx\\
					\lesssim
					&\|Z^{\beta_h} u\|_{L^{\infty}}
					\|Z^{\al_h-\beta_h}\nabla \vro\|_{L^2} \|Z^{\al_h} \vro\|_{L^2} + \| Z^{\beta_h}  \vro\|_{L^2}^{\f14}  \| \p_3 Z^{\beta_h}  \vro\|_{L^2}^{\f14}  \| \p_1 Z^{\beta_h}  \vro\|_{L^2}^{\f14}  \|\p_{13} Z^{\beta_h}  \vro\|_{L^2}^{\f14} \\
                    &\;\times
					 \| Z^{\al_h -\beta_h}  \du\|_{L^2}^{\f12} \| \p_2 Z^{\al_h -\beta_h}  \du\|_{L^2}^{\f12} \|Z^{\al_h} \vro\|_{L^2}\\
					\lesssim
					&\sqrt{\me(t)}\bmd_{tan}(t).
				\end{aligned}
				\deqq
				Thus, the combination of above estimates yields directly
				\beqq
				I_1
				\lesssim \sqrt{\me(t)}\bmd_{tan}(t).
				\deqq
                Similarly, integrating by parts, for $i=1,2$, we can conclude that
                \begin{align*}
               I_2 =  &\; \i Z^{\al_h-e_i} ( \vro \nabla \vro ) \cdot Z^{\al_h+e_i} u \ dx \lesssim \sqrt{\me(t)}\bmd_{tan}(t),\\
               I_3 = &\; - \sum_{0 < \beta_h \le \alpha_h} C^{\beta_h}_{\alpha_h}
				\i Z^{\beta_h} u \cdot Z^{\al_h-\beta_h}\nabla u\cdot Z^{\al_h} u \ dx
                - \i  u \cdot Z^{\al_h} \nabla u \cdot Z^{\al_h} u\ dx\\
                = &\; - \sum_{0 < \beta_h \le \alpha_h} C^{\beta_h}_{\alpha_h}
				\i (Z^{\beta_h} u_h \cdot Z^{\al_h-\beta_h}\nabla_h u + Z^{\beta_h} u_3 \, Z^{\al_h-\beta_h}\p_3  u)\cdot Z^{\al_h} u \ dx
                +\f12 \i  \du \,| Z^{\al_h} u|^2\ dx\\
                \lesssim &\; \sqrt{\me(t)}\bmd_{tan}(t),\\
                I_4 = &\; \i Z^{\al_h-e_i} ( \f{\vro}{1+\vro }\Delta_h u ) \cdot Z^{\al_h+e_i} u \ dx \lesssim \sqrt{\me(t)}\bmd_{tan}(t).
                \end{align*}
              Next we estimate the term $I_5$.
              \beqq \bal
  I_5 =&\;  - \i Z^{\al_h} \Big(\nabla (\f{\vro}{1+\vro }\du) - \nabla (\f{\vro}{1+\vro }) \du \Big) \cdot Z^{\al_h} u \ dx\\
  =&\;   \i Z^{\al_h} (\f{\vro}{1+\vro }\du) \cdot Z^{\al_h} \du \ dx + \i Z^{\al_h} (\nabla (\f{\vro}{1+\vro }) \du ) \cdot Z^{\al_h} u \ dx\\
  =&\;  \sum_{0 \le \beta_h \le \alpha_h} C^{\beta_h}_{\alpha_h} \i Z^{\be_h} (\f{\vro}{1+\vro })  Z^{\al_h-\be_h}  \du \cdot Z^{\al_h} \du \ dx \\
  &\; + \sum_{0 \le \beta_h \le \alpha_h} C^{\beta_h}_{\alpha_h} \i Z^{\be_h} \nabla (\f{\vro}{1+\vro })  Z^{\al_h-\be_h} \du  \cdot Z^{\al_h} u \ dx.
              \dal \deqq
              Integrating by parts and applying  the anisotropic
				type inequality \eqref{ie:Sobolev}, for $i=1,2$, we can  obtain
             \beqq \bal
&\; \sum_{0 \le \beta_h \le \alpha_h} C^{\beta_h}_{\alpha_h} \i Z^{\be_h} \nabla (\f{\vro}{1+\vro })  Z^{\al_h-\be_h} \du  \cdot Z^{\al_h} u \ dx \\
= &\; \sum_{0 \le \beta_h < \alpha_h} C^{\beta_h}_{\alpha_h} \i Z^{\be_h} \nabla (\f{\vro}{1+\vro })  Z^{\al_h-\be_h} \du  \cdot Z^{\al_h} u \ dx \\
&\; - \i Z^{\al_h} (\f{\vro}{1+\vro }) \cdot (Z^{\al_h} u \cdot \nabla \du + \du \,Z^{\al_h} \du) \ dx\\
\lesssim &\; \Big(\| \nabla (\f{\vro}{1+\vro }) \|_{H_{tan}^2} \|Z^{\al_h} \du\|_{L^2} +\sum_{0 < \beta_h  < \alpha_h-e_i}  \| Z^{\be_h}\nabla (\f{\vro}{1+\vro }) \|_{H_{tan}^1} \|Z^{\al_h-\be_h} \du\|_{H_{tan}^1}
\\
&\; +  \|Z^{\alpha_h-e_i} \nabla (\f{\vro}{1+\vro }) \|_{L^2} \|Z^{e_i} \du\|_{H_{tan}^2} \Big) \|Z^{\al_h} u\|_{L^2}^{\f12} \|\p_3 Z^{\al_h} u\|_{L^2}^{\f12} \\
&\; + \|Z^{\alpha_h}  (\f{\vro}{1+\vro })\|_{L^2}(\| Z^{\al_h} u\|_{L^2}^{\f12} \|\p_3 Z^{\al_h} u\|_{L^2}^{\f12} \|\nabla \du\|_{H_{tan}^{2}}+ \| \du\|_{L^{\infty}} \| Z^{\al_h} \du\|_{L^2})\\
\lesssim &\; \sqrt{\me(t)}\bmd_{tan}(t).
\dal \deqq
            Similarly, we have
            \beqq
 \sum_{0 \le \beta_h \le \alpha_h} C^{\beta_h}_{\alpha_h} \i Z^{\be_h} (\f{\vro}{1+\vro })  Z^{\al_h-\be_h}  \du \cdot Z^{\al_h} \du \ dx  \lesssim \sqrt{\me(t)}\bmd_{tan}(t),
            \deqq
            which, together with the above estimate,  yields that
            \beqq
           I_5 \lesssim \sqrt{\me(t)}\bmd_{tan}(t).
\deqq
It remains to estimate the term $I_6$.  Similar to the estimate of $I_5$, for $i=1,2$, we have
\begin{align*}
I_6 = &\;  - \var \i Z^{\al_h} \Big( \p_3( \f{\vro}{1+\vro } \p_3 u ) - \p_3( \f{\vro}{1+\vro }) \p_3 u  \Big)
\cdot Z^{\al_h} u \ dx \\
= &\;  \var \i Z^{\al_h} ( \f{\vro}{1+\vro } \p_3 u ) \cdot Z^{\al_h} \p_3 u\ dx - \var \sum_{0 \le \beta_h \le \alpha_h} C^{\beta_h}_{\alpha_h}  \i Z^{\be_h} \p_3( \f{\vro}{1+\vro }) Z^{\al_h-\be_h}  \p_3 u
\cdot Z^{\al_h} u\ dx \\
= &\;  \var \i Z^{\al_h} ( \f{\vro}{1+\vro } \p_3 u ) \cdot Z^{\al_h} \p_3 u\ dx - \var \sum_{0 \le \beta_h < \alpha_h} C^{\beta_h}_{\alpha_h}  \i Z^{\be_h} \p_3( \f{\vro}{1+\vro }) Z^{\al_h-\be_h}  \p_3 u
\cdot Z^{\al_h} u\ dx \\
&\; +\var \i Z^{\al_h-e_i} \p_3( \f{\vro}{1+\vro })\, Z^{e_i} (\p_3 u
\cdot Z^{\al_h} u)\ dx\\
\lesssim &\; \sqrt{\me(t)}\bmd_{tan}(t).
\end{align*}
Substituting the estimates from $I_1$ to $I_6$ into \eqref{31-equ}, we have for $|\al_h|=m-1$, we have
                   \begin{equation}\label{3104} \bal
					&\f12 \frac{d}{dt} \|Z^{\ah}(\vro, u)(t)\|_{L^2}^2
				+\|Z^{\ah}(\nabla_h u, \du)(t)\|_{L^2}^2 +\var \|\p_3 Z^{\ah} u(t)\|_{L^2}^2 \lesssim \sqrt{\me(t)}\bmd_{tan}(t).
				\dal \end{equation}
                \textbf{Finally, we deal with the case $|\ah|=m$.}
                Similar to the case $|\ah|=m-1$,  we can obtain
                \begin{equation} \label{3105}\bal
					&\f12 \frac{d}{dt}  \sum\limits_{ |\al|=m} \|Z^{\ah}(\vro, u)(t)\|_{L^2}^2
				+ \!\!\sum\limits_{ |\al|=m} \|Z^{\ah}(\nabla_h u, \du)(t)\|_{L^2}^2 +\var \!\!\sum\limits_{ |\al|=m} \|\p_3 Z^{\ah} u(t)\|_{L^2}^2 \lesssim \sqrt{\me(t)}\md_{tan}(t).
				\dal \end{equation}
                Combining the estimates \eqref{3102}, \eqref{3103}-\eqref{3105}, we complete the proof of this lemma.
			\end{proof}
            Next, in order to obtain the decay estimate, we will establish the estimate for the vertical derivative as well as the higher derivative 
			of density and velocity field with the help of Lemma \ref{lemma-help}. 	
			\begin{lemm}\label{lemma33}
				For any smooth solution $(\vro,u)$ of equation \eqref{eqr},
				it holds 
				\beq\label{3201-1} \bal
				&\; \f12 \frac{d}{dt}( \|\p_3(\vro, u)(t)\|_{L^2}^2 + \|\p_3(\vro, u)(t)\|_{\dot{H}_{tan}^{m-1}}^2 )\\
				&+\|\p_3 (\nabla_h u, \du)(t)\|_{H^{m-1}_{tan}}^2 +\var \|\p_3^2 u(t)\|_{H^{m-1}_{tan}}^2
				\lesssim   \;(\me(t)^{\f13}+\me(t)^{\f12})\md_{tan}(t),
                \dal \deq 
                and
                \beq\label{3201-2} \bal
				&\; \f12 \frac{d}{dt} (\|\p_3\nabla_h(\vro, u)(t)\|_{L^2}^2 + \|\p_3 \nabla_h (\vro, u)(t)\|_{\dot{H}_{tan}^{m-3}}^2 )\\
				&+\|\p_3 \nabla_h (\nabla_h u, \du)(t)\|_{H^{m-3}_{tan}}^2 +\var \|\nabla_h \p_3^2 u(t)\|_{H^{m-3}_{tan}}^2
				\lesssim \; (\me(t)^{\f13}+\me(t)^{\f12})\bar{\md}_{tan}(t).
				\dal \deq
			\end{lemm}
			\begin{proof}
				For any $|\ah|=k\le m-1$, applying $Z^{\ah} \p_3$-operator to $\eqref{eqr}_1$ and multiplying by $Z^{\ah} \p_3\vro$ and integrating over $\mathbb{R}_+^3$, we have
                \begin{equation*}
					\begin{aligned}
						\frac{d}{dt}\frac{1}{2}\i| Z^{\ah} \p_3\vro|^2 \ dx + \i  Z^{\ah} \p_3\du\cdot Z^{\ah} \p_3\vro \ dx
						=\i Z^{\ah} \p_3 (-u\cdot \nabla \vro -\vro \, \du )\cdot Z^{\ah} \p_3 \vro \ dx.
					\end{aligned}
				\end{equation*}
                Next, applying $Z^{\ah}$-operator to $\eqref{eqr}_2$ and multiplying by $-Z^{\ah} \p_3^2 u$ and integrating over $\mathbb{R}_+^3$, together with the above equality, we can obtain that
				\begin{equation} \label{32-equ}
					\begin{aligned}
						&\frac{d}{dt}\frac{1}{2}\i(| Z^{\ah} \p_3\vro|^2+|Z^{\ah} \p_3 u|^2)dx
						+\|Z^{\ah} \nabla_h \p_3 u\|_{L^2}^2 +\| Z^{\ah} \p_3 \du\|_{L^2}^2 + \var \|\p_3^2 Z^{\ah}u\|_{L^2}^2 \\
						=&\i Z^{\ah} \p_3 (-u\cdot \nabla \vro -\vro \, \du )\cdot Z^{\ah} \p_3 \vro \ dx
						+\i Z^{\ah} (\vro\, \nabla \vro)\cdot Z^{\ah} \p_3^2 u\ dx\\
						& +\i Z^{\ah} (u\cdot \nabla u)\cdot Z^{\ah}  \p_3^2 u\ dx+\i Z^{\ah} (\f{\vro}{1+\vro} \Delta_h u)\cdot Z^{\ah} \p_3^2  u\ dx  \\
                        &+\i Z^{\ah} (\f{\vro}{1+\vro} \nabla \du)\cdot Z^{\ah}  \p_3^2 u\ dx +\var \i Z^{\ah} (\f{\vro}{1+\vro}  \p_3^2 u)\cdot Z^{\ah}  \p_3^2 u\ dx\\
				:= & \sum_{i=1}^{6} II_i,
					\end{aligned}
				\end{equation}
				where we have used the basic fact
				\beqq \bal
				& \i Z^{\ah} \p_3 \du \cdot Z^{\ah} \p_3 \vro \ dx -\i Z^{\ah} \nabla \vro \cdot Z^{\ah} \p_3^2 u \ dx
				\\
                = &\; \i Z^{\ah} \p_3 (\nabla_h \cdot u_h) \cdot Z^{\ah} \p_3 \vro \ dx -\i Z^{\ah} \nabla_h \vro \cdot Z^{\ah} \p_3^2 u_h \ dx
                = 0.
				\dal \deqq
				Since the case $|\ah|=0$ is elementary, we can easily obtain the following estimate
   \beq \label{3202}
    \f12 \frac{d}{dt} \|\p_3(\vro, u)(t)\|_{L^2}^2
				+\|\p_3(\nabla_h u, \du)(t)\|_{L^2}^2 +\var \|\p_3^2 u(t)\|_{L^2}^2
				\lesssim  \sqrt{\me(t)}\md_{tan}(t).
   \deq 
   We restrict our attention to the case $|\ah|>0$. 
   {\textbf{Now we deal with the case $|\ah|=1$.}}
    \begin{align*}
  II_1 = &\; - \i (\nabla_h \p_3 u \cdot \nabla \vro + \nabla_h \p_3 \vro \, \du) \cdot \nabla_h \p_3 \vro \ dx - \i ( \p_3 u \cdot  \nabla \nabla_h  \vro + \p_3 \vro \, \nabla_h  \du) \cdot \nabla_h \p_3 \vro \ dx \\
  &\; - \i ( \nabla_h u \cdot \nabla \p_3  \vro +  \nabla_h \vro \,  \p_3 \du) \cdot \nabla_h \p_3 \vro \ dx - \i (  u \cdot \nabla \nabla_h\p_3  \vro  +   \vro \,  \nabla_h \p_3 \du) \cdot \nabla_h \p_3 \vro \ dx \\
   = &\; - \i (\nabla_h \p_3 u \cdot \nabla \vro + \f12 \nabla_h \p_3 \vro \, \du) \cdot \nabla_h \p_3 \vro \ dx - \i ( \p_3 u \cdot  \nabla \nabla_h  \vro + \p_3 \vro \, \nabla_h  \du) \cdot \nabla_h \p_3 \vro \ dx \\
   &\; - \i ( \nabla_h u_h \cdot \nabla_h \p_3  \vro +  \nabla_h \vro \,  \p_3 \du) \cdot \nabla_h \p_3 \vro \ dx - \i \nabla_h u_3 \, \p_3^2 \vro \cdot \nabla_h \p_3 \vro \ dx  \\
   &\; - \i  \vro \,  \nabla_h \p_3 \du \cdot \nabla_h \p_3 \vro \ dx \\
  := &\; \sum_{i=1}^{5} II_{1,i}.
   \end{align*}
We first handle the term $II_{1,4}$ involving second-order vertical derivatives, as the energy framework only involves first-order vertical derivatives.
 Define $\chi(x_3)$ is a smooth compactly supported function which takes the value one in the vicinity of $0$ and is supported in $[0,1]$. Then we rewrite $II_{1,4}$ as follows:
\beq \label{ah1-II14}
		\bal
			II_{1,4}  &\; = \i \nabla_h u_3 \, \p_3^2 \vro  \cdot \nabla_h \p_3 \vro (1-\chi+\chi)^2  \, dx\\
		&\; \lesssim  \i |\nabla_h u_3 \, \p_3^2 \vro  \cdot \nabla_h \p_3 \vro| (1-\chi)^2 \, dx + \i |\nabla_h u_3 \, \p_3^2 \vro  \cdot \nabla_h \p_3 \vro| \chi^2 \, dx.
		\dal
		\deq
Away from the boundary, the conormal $H_{co}^m$ norm is equivalent to the usual $H^m$ norm, then we can obtain
\begin{align*}
 &\; \i |\nabla_h u_3 \, \p_3^2 \vro  \cdot \nabla_h \p_3 \vro| (1-\chi)^2 \, dx  \\
 \lesssim  &\; \|  \nabla_h u_3\|_{L^2}^{\f14} \| \p_{1} \nabla_h u_3\|_{L^2}^{\f14} \| \p_{3} \nabla_h u_3\|_{L^2}^{\f14} \| \p_{13}  \nabla_h u_3\|_{L^2}^{\f14} \| \nabla_h \p_3 \vro\|_{L^2}  \| Z_3 \p_3 \vro\|_{L^2}^{\f12}  \|\p_2 Z_3 \p_3 \vro\|_{L^2}^{\f12} \\
 \lesssim  &\; \|  \nabla_h u_3\|_{L^2}^{\f14} \| \p_{1} \nabla_h u_3\|_{L^2}^{\f14} \| \p_{3} \nabla_h u_3\|_{L^2}^{\f14} \| \p_{13}  \nabla_h u_3\|_{L^2}^{\f14} \| \nabla_h \p_3 \vro\|_{L^2}  \| Z_3 \p_3 \vro\|_{L^2}^{\f12}\\
 &\quad \times ( \|\p_2 \p_3 \vro\|_{L^2}^{\f12} + \|\p_2 \p_3 \vro\|_{L^2}^{\f14} \|\p_2 Z_3 \p_3 \vro\|_{L^2}^{\f14}) \\
 \lesssim &\; \sqrt{\me(t)}\bmd_{tan}(t),
\end{align*}
where we have used the interpolation inequality \eqref{a10}.
Since $u_3 = 0$ on the boundary, we can write $u_3 = \int_{0}^{x_3} \p_3 u_3 \ dz$, thus we have
\beqq \bal
&\; \i |\nabla_h u_3 \, \p_3^2 \vro  \cdot \nabla_h \p_3 \vro| \chi^2 \, dx
=  \i |\nabla_h (\int_{0}^{x_3} \p_3 u_3 \ dz) \, \p_3^2 \vro  \cdot \nabla_h \p_3 \vro| \chi^2 \, dx\\
\lesssim &\; \| \nabla_h \p_3 u_3 \|_{L^2}^{\f14} \| \nabla_h \p_3^2 u_3 \|_{L^2}^{\f14} \| \nabla_h \p_{13} u_3 \|_{L^2}^{\f14} \| \nabla_h \p_1 \p_3^2 u_3 \|_{L^2}^{\f14}   \| \nabla_h \p_3 \vro\|_{L^2}  \| Z_3 \p_3 \vro\|_{L^2}^{\f12}  \|\p_2 Z_3 \p_3 \vro\|_{L^2}^{\f12} \\
 \lesssim &\; \sqrt{\me(t)}\bmd_{tan}(t).
\dal \deqq
Therefore, we can obtain that
\beqq
 II_{1,4} \lesssim \sqrt{\me(t)}\bmd_{tan}(t).
\deqq
Applying  the anisotropic
				type inequality \eqref{ie:Sobolev} and using the estimate \eqref{est-p32-u3-uh}, it is easy to check that
                \begin{align*}
                II_{1,1}
\lesssim &\; \|\nabla_h \p_3 \vro \|_{L^2} \| \nabla_h \p_3 u\|_{L^2}^{\f14} \| \p_3 \nabla_h \p_3 u\|_{L^2}^{\f14} \| \p_1 \nabla_h \p_3 u\|_{L^2}^{\f14} \| \p_{13} \nabla_h \p_3 u\|_{L^2}^{\f14} \| \nabla \vro\|_{L^2}^{\f12}  \| \p_2 \nabla \vro\|_{L^2}^{\f12} \\
 &\;+ (\|\p_3 u_3\|_{L^{\infty}} +\| \nabla_h u \|_{L^{\infty}}) \|  \nabla_h \p_3  \vro\|_{L^2}^2 \\
\lesssim &\; \sqrt{\me(t)}\bmd_{tan}(t) + \|(\p_3 u_3 ,\p_3^2 u_3)\|_{L^2}^{\f14} \|\nabla_h (\p_3 u_3 ,\p_3^2 u_3)\|_{H_{tan}^2}^{\f34} \|  \nabla_h \p_3  \vro\|_{L^2}^2 \\
\lesssim &\; \sqrt{\me(t)}\bmd_{tan}(t).
 \end{align*}
                Similarly, we have
                \beqq \bal
II_{1,2} + II_{1,3} + II_{1,5} \lesssim \sqrt{\me(t)}\bmd_{tan}(t),
                \dal \deqq
which, together with the estimates of $II_{1,1}$ and  $II_{1,4}$, yields that
\beqq
 II_{1} \lesssim \sqrt{\me(t)}\bmd_{tan}(t).
\deqq
Similarly, applying  the anisotropic
				type inequality \eqref{ie:Sobolev}, integrating by parts, we can obtain
\begin{align*}
II_2 = &\;\i \nabla_h ( \vro\, \p_3 \vro) \cdot \nabla_h \p_3^2 u_3 \ dx -\i \p_3 \nabla_h ( \vro\, \nabla_h \vro) \cdot \nabla_h \p_3 u_h \ dx  \\
= &\; \i \nabla_h ( \vro\, \p_3 \vro) \cdot \nabla_h \p_3^2 u_3 \ dx -\i (\p_3 \nabla_h \vro\, \nabla_h \vro + \nabla_h \vro\, \p_3 \nabla_h \vro + \p_3  \vro\, \nabla_h^2 \vro )\cdot  \nabla_h \p_3 u_h \ dx \\
&\;  + \i \p_3 \nabla_h \vro \cdot \nabla_h(  \vro \, \nabla_h \p_3 u_h) \ dx\\
\lesssim &\; \sqrt{\me(t)}\bmd_{tan}(t),\\
II_4 = &\;\i \nabla_h ( \f{\vro}{1+\vro} \, \Delta_h u_3) \cdot \nabla_h \p_3^2 u_3 \ dx -\i \p_3 \nabla_h ( \f{\vro}{1+\vro} \, \Delta_h u_h) \cdot \nabla_h \p_3 u_h \ dx  \\
= &\;\i \nabla_h ( \f{\vro}{1+\vro} \, \Delta_h u_3) \cdot \nabla_h \p_3^2 u_3 \ dx -\i ( \p_3 \nabla_h (\f{\vro}{1+\vro}) \, \Delta_h u_h +  \nabla_h( \f{\vro}{1+\vro}) \, \p_3 \Delta_h u_h) \cdot \nabla_h \p_3 u_h \ dx \\
&\; - \i \p_3(\f{\vro}{1+\vro})  \, \nabla_h \Delta_h u_h \cdot \nabla_h \p_3 u_h \ dx  + \i \p_3 \Delta_h u_h \cdot \nabla_h (\f{\vro}{1+\vro} \,  \nabla_h \p_3 u_h) \ dx\\
\lesssim &\; \sqrt{\me(t)}\bmd_{tan}(t),
\end{align*}
and using the estimate \eqref{est-p32-u3-uh}, we have
\begin{align*}
II_{5} = &\;\i \nabla_h ( \f{\vro}{1+\vro} \, \p_3 \du) \cdot \nabla_h \p_3^2 u_3 \ dx -\i \p_3 \nabla_h ( \f{\vro}{1+\vro} \, \nabla_h \du) \cdot \nabla_h \p_3 u_h \ dx  \\
\lesssim &\; \sqrt{\me(t)}\bmd_{tan}(t),\\
II_{6} = &\; \var \i \Big(\nabla_h ( \f{\vro}{1+\vro}) \, \p_3^2 u +  \f{\vro}{1+\vro} \, \nabla_h \p_3^
2 u\Big) \cdot \nabla_h \p_3^2 u \ dx\\
\lesssim &\; \sqrt{\me(t)}\bmd_{tan}(t).
\end{align*}
It remains to estimate the term $II_{3}$. Integrating by parts, we split $II_3$ as follows:
\begin{align*}
II_3 = &\; -\i \p_3 (u \cdot \nabla  \nabla_h  u + \nabla_h u_h \cdot \nabla_h u + \nabla_h u_3 \, \p_3 u) \cdot  \p_3 \nabla_h u \ dx\\
= &\;  \f12 \i \du |  \p_3 \nabla_h u |^2 \ dx - \i   \p_3 u \cdot \nabla  \nabla_h  u  \cdot  \p_3 \nabla_h u \ dx - \i \p_3 (\nabla_h u_h \cdot \nabla_h u ) \cdot  \p_3 \nabla_h u \ dx\\
&\; -\i \p_3 \nabla_h u_3 \, \p_3 u \cdot  \p_3 \nabla_h u \ dx
-\i  \nabla_h u_3 \, \p_3^2 u \cdot  \p_3 \nabla_h u \ dx\\
:= &\; \sum_{i=1}^{5} II_{3,i}.
\end{align*}
Using the interpolation inequalities \eqref{inter-1}-\eqref{inter-2} and the anisotropic type inequality \eqref{ie:Sobolev}, we can check that
\begin{align*}
&\; II_{3,1} + II_{3,2}\\
\lesssim &\; (\| \p_3 u_3 \|_{L^{\infty}} + \| \nabla_h u\|_{L^{\infty}}) \| \p_3 \nabla_h u \|_{L^2}^2 + \| \p_3 u_h \|_{L^2}^{\f12} \| \p_{13} u_h\|_{L^2}^{\f12} \|\nabla_h^2 u\|_{L^2}^{\f12} \|\p_3 \nabla_h^2 u\|_{L^2}^{\f12} \| \p_3 \nabla_h u\|_{L^2}^{\f12} \| \p_{23} \nabla_h u\|_{L^2}^{\f12}\\
\lesssim &\;  \Big(\| (\p_3 u_3, \p_3^2 u_3 )\|_{L^{2}}^{\f14} \| \nabla_h( \p_3 u_3, \p_3^2 u_3 )\|_{H_{tan}^1}^{\f34}   + \| (\nabla_h u, \p_3 \nabla_h u )\|_{L^{2}}^{\f14} \| \nabla_h( \nabla_h u, \p_3 \nabla_h u)\|_{H_{tan}^1}^{\f34} \Big) \\
&\; \times \|(  \hs \p_3 u\|_{L^2}^{\f{1}{s+2}}  \| \nabla_h^2 \p_3 u\|_{L^2}^{\f{s+1}{s+2}} )^{2}  + \| \p_3 u_h \|_{L^2}^{\f12} \|\nabla_h^2 u\|_{L^2}^{\f12} \|\p_3 \nabla_h^2 u\|_{L^2}^{\f12} \| \p_{3} u\|_{L^2}^{\f12} \| \nabla_h^2 \p_{3} u\|_{L^2}^{\f12} \| \p_{23} \nabla_h u\|_{L^2}^{\f12}\\
\lesssim &\;  \sqrt{\me(t)}\bmd_{tan}(t), 
\end{align*}
here we have  used that $s \ge \f{16}{17}$ and the estimate \eqref{est-p32-u3-uh}. In a similar way, we can arrive that
\beqq  \bal
II_{3,3} +II_{3,4} \lesssim \sqrt{\me(t)}\bmd_{tan}(t).
\dal \deqq
Next, similar to the estimate \eqref{ah1-II14} of $II_{1,4}$, we rewrite the term $II_{3,5}$ as follows:
\beqq \bal
II_{3,5} \lesssim &\;\i |\nabla_h u_3 \, \p_3^2 u_3 \cdot  \p_3 \nabla_h u_3| \ dx +  \i |\nabla_h u_3 \, \p_3^2 u_h \cdot  \p_3 \nabla_h u_h| (1-\chi)^2  \ dx  + \i |\nabla_h u_3 \, \p_3^2 u_h \cdot  \p_3 \nabla_h u_h| \chi^2  \ dx\\
\lesssim &\; \|\nabla_h u_3\|_{L^2}^{\f12} \|\p_3 \nabla_h u_3\|_{L^2}^{\f12} \|\p_3^2 u_3\|_{L^2}^{\f14}  \|\nabla_h \p_3^2 u_3\|_{L^2}^{\f12}  \|\nabla_h^2 \p_3^2 u_3\|_{L^2}^{\f14} \|\p_3 \nabla_h u_3\|_{L^2}\\
&\; +  \i |\nabla_h u_3 \, \p_3^2 u_h \cdot  \p_3 \nabla_h u_h| (1-\chi)^2  \ dx  + \i |\nabla_h u_3 \, \p_3^2 u_h \cdot  \p_3 \nabla_h u_h| \chi^2  \ dx.\\
\dal \deqq
Away from the boundary, the conormal $H_{co}^m$ norm is equivalent to the usual $H^m$ norm,  using the interpolation inequalities $\eqref{ie:Sobolev}_4$  for the term $Z_3 \p_{13} u$, $\eqref{ie:Sobolev}_5$ for the term $Z_3 \p_3 u$ as well as the interpolation inequality \eqref{inter-2}, then we can obtain that
\beq\label{est-s-16/17}\bal
 &\; \i |\nabla_h u_3 \, \p_3^2 u_h \cdot  \p_3 \nabla_h u_h| (1-\chi)^2  \ dx  \\
\lesssim &\; \|\nabla_h u_3 \|_{L^2}^{\f12} \|\p_3 \nabla_h u_3 \|_{L^2}^{\f12}
\|\p_3 \nabla_h u_h \|_{L^2}^{\f12}
\|\p_{23} \nabla_h u_h \|_{L^2}^{\f12} \|Z_3 \p_3 u_h\|_{L^2}^{\f12} \|Z_3 \p_{13} u_h \|_{L^2}^{\f12} \\
\lesssim &\; \|(\nabla_h u_3, \p_3 \nabla_h u_h) \|_{L^2}^{\f43}
 \|(\p_3 \nabla_h u_3,\p_{23} \nabla_h u_h ) \|_{L^2} \|\p_3  u_h \|_{L^2}^{\f38} \|(\p_3 u_h, Z_3^4 \p_3 u_h)\|_{L^2}^{\f18} \|(\p_{13} u_h, Z_3^3 \p_{13} u_h)\|_{L^2}^{\f16} \\
\lesssim &\; (\|\hs (\nabla_h u_3, \p_3 \nabla_h u_h) \|_{L^2}^{\f{1}{s+2}} \|\nabla_h^2 ( u_3, \p_3 u_h) \|_{L^2}^{\f{s+1}{s+2}})^{\f43}
 \|(\p_3 \nabla_h u_3,\p_{23} \nabla_h u_h ) \|_{L^2}  \\
&\; \times  (\|\hs \p_3  u_h \|_{L^2}^{\f{2}{s+2}} \|\nabla_h^2 \p_3  u_h \|_{L^2}^{\f{s}{s+2}})^{\f38} \|(\p_3 u_h, Z_3^4 \p_3 u_h)\|_{L^2}^{\f18} \|(\p_{13} u_h, Z_3^3 \p_{13} u_h)\|_{L^2}^{\f16} \\
\lesssim  &\;  \sqrt{\me(t)}\bmd_{tan}(t),
\dal \deq
 where we require the constants satisfy $m\ge 5$ and  $\f{4(s+1)}{3(s+2)} + \f{3s}{8(s+2)} \ge 1$, which yields $\f{16}{17} \le s < 1$.
Similarly, by rewriting $u_3 = \int_{0}^{x_3} \p_3 u_3 \ dz$, we can obtain
\begin{align*}
&\; \i |\nabla_h u_3 \, \p_3^2 u_h \cdot  \p_3 \nabla_h u_h| \chi^2  \ dx\\
\lesssim &\;  \|\p_3 \nabla_h  u_3 \|_{L^2}^{\f12} \|\p_3^2 \nabla_h u_3 \|_{L^2}^{\f12}
\|\p_3 \nabla_h u_h \|_{L^2}^{\f12}
\|\p_{13} \nabla_h u_h \|_{L^2}^{\f12}  \|Z_3 \p_3 u_h \|_{L^2}^{\f12} \|Z_3 \p_{23} u_h \|_{L^2}^{\f12}\\
\lesssim &\;  \|\p_3 u_h \|_{L^2}^{\f14} \| \p_3  \nabla_h u_h \|_{L^2}^{\f34}  \|(\p_3 \nabla_h  u_3,\p_3^2 \nabla_h u_3 , \p_{13} \nabla_h u_h )\|_{L^2}^{\f32}
 \|(\p_3 u_h, Z_3^2 \p_3 u_h )\|_{L^2}^{\f14}  \|(\p_{23} u_h, Z_3^2 \p_{23} u_h )\|_{L^2}^{\f14}\\
\lesssim &\;  (\|\hs \p_3  u_h \|_{L^2}^{\f{2}{s+2}} \|\nabla_h^2 \p_3  u_h \|_{L^2}^{\f{1}{s+2}})^{\f14}  (\|\hs \p_3 \nabla_h u_h \|_{L^2}^{\f{1}{s+2}} \|\nabla_h^2  \p_3 u_h \|_{L^2}^{\f{s+1}{s+2}})^{\f34}\\
&\; \times  \|(\p_3 \nabla_h  u_3,\p_3^2 \nabla_h u_3 , \p_{13} \nabla_h u_h )\|_{L^2}^{\f32}
 \|(\p_3 u_h, Z_3^2 \p_3 u_h )\|_{L^2}^{\f14}  \|(\p_{23} u_h, Z_3^2 \p_{23} u_h )\|_{L^2}^{\f14}\\
 \lesssim  &\;  \sqrt{\me(t)}\bmd_{tan}(t).
\end{align*}
Thus, the combination of the above estimates yields that
\beqq
II_{3,5} \lesssim  \sqrt{\me(t)}\bmd_{tan}(t),
\deqq
which, together with the estimates from $II_{3,1}$ to  $II_{3,4}$, yields that
\beqq
II_{3} \lesssim \sqrt{\me(t)}\bmd_{tan}(t).
\deqq
Substituting the estimates from $II_1$ to $II_6$ to \eqref{32-equ}, we can obtain that
\beq \label{3203}
    \f12 \frac{d}{dt} \|\nabla_h \p_3(\vro, u)(t)\|_{L^2}^2
				+\|\nabla_h \p_3(\nabla_h u, \du)(t)\|_{L^2}^2 +\var \|\nabla_h \p_3^2 u(t)\|_{L^2}^2
				\lesssim  \sqrt{\me(t)}\bmd_{tan}(t).
   \deq
    {\textbf{Now we deal with the case $|\ah|=m-2$.}}
   Integrating by parts, we have
   \begin{align*}
   II_1 = &\; - \i Z^{\al_h} ( u \cdot \nabla \p_3 \vro  + \p_3 u \cdot \nabla \vro + \vro \, \p_3 \du + \p_3  \vro \, \du) \cdot Z^{\al_h} \p_3 \vro \ dx\\
   = &\; \f12 \i \du \, | Z^{\al_h} \p_3 \vro|^2 \ dx -\sum_{0< \beta_h \le \alpha_h}C^{\beta_h}_{\alpha_h}
				\i Z^{\beta_h} u_h \cdot \nabla_h  Z^{\al_h-\beta_h}\p_3 \vro \cdot Z^{\al_h} \p_3 \vro\ dx\\
                &\; -\sum_{0< \beta_h \le \alpha_h}C^{\beta_h}_{\alpha_h}
				\i Z^{\beta_h} u_3 \cdot  Z^{\al_h-\beta_h} \p_3^2 \vro \cdot Z^{\al_h} \p_3 \vro\ dx\\
                &\; -\sum_{0 \le \beta_h \le \alpha_h}C^{\beta_h}_{\alpha_h}
				\i (Z^{\beta_h} \p_3 u_3 \cdot Z^{\al_h-\beta_h}\p_3  \vro + Z^{\beta_h} \p_3 u_h \cdot Z^{\al_h-\beta_h}\nabla_h  \vro) \cdot Z^{\al_h} \p_3 \vro\ dx\\
                &\; - \i Z^{\al_h} (\vro \, \p_3 \du + \p_3  \vro \, \du) \cdot Z^{\al_h} \p_3 \vro \ dx\\
                := &\; \sum_{i=1}^{5} II_{1,i}.
   \end{align*}
   Applying the anisotropic type inequality \eqref{ie:Sobolev}, it is easy to check that for $i=1,2$,
   \begin{align*}
   II_{1,2} \lesssim  &\; \sum_{e_i < \beta_h \le \alpha_h} \|Z^{\be_h} u_h\|_{H_{tan}^1}^{\f12}  \|\p_3 Z^{\be_h} u_h\|_{H_{tan}^1}^{\f12}
   \|Z^{\ah-\be_h} \nabla_h \p_3 \vro\|_{L^2}^{\f12}   \|\p_2 Z^{\ah-\be_h} \nabla_h \p_3 \vro\|_{L^2}^{\f12}
    \|Z^{\ah} \p_3 \vro\|_{L^2}\\
    + &\; \|Z^{e_i } u_h\|_{L^{\infty}} \|Z^{\ah-e_i} \nabla_h \p_3 \vro\|_{L^2}  \|Z^{\ah} \p_3 \vro\|_{L^2} \\
    \lesssim &\; \sqrt{\me(t)}\bmd_{tan}(t),\\
    II_{1,4} \lesssim &\;  (\|\p_3 u_3\|_{L^{\infty}} \|Z^{\ah} \p_3 \vro\|_{L^2} + \|\p_3 u_h\|_{L^{\infty}} \|Z^{\ah} \nabla_h \vro\|_{L^2} )\|Z^{\ah} \p_3 \vro\|_{L^2}\\
    &\; + \sum_{0 < \beta_h \le [\f{\alpha_h}{2}]} \Big(\|Z^{\be_h} \p_3 u_3\|_{H_{tan}^1}^{\f12}  \| Z^{\be_h} \p_3^2 u_3\|_{H_{tan}^1}^{\f12}
   \|Z^{\ah-\be_h} \p_3 \vro\|_{L^2}^{\f12}   \|\p_2 Z^{\ah-\be_h}  \p_3 \vro\|_{L^2}^{\f12} \\
   &\; +  \|Z^{\be_h} \p_3 u_h\|_{H_{tan}^2}
   \|Z^{\ah-\be_h} \nabla_h \vro\|_{L^2}^{\f12}   \|\p_3 Z^{\ah-\be_h} \nabla_h  \vro\|_{L^2}^{\f12} \Big) \|Z^{\ah} \p_3 \vro\|_{L^2}\\
   &\; + \sum_{[\f{\alpha_h}{2}] < \beta_h \le \ah} \Big(\|Z^{\be_h} \p_3 u_3\|_{L^2}^{\f12}  \| Z^{\be_h} \p_3^2 u_3\|_{L^2}^{\f12}
   \| Z^{\ah-\be_h} \p_3 \vro\|_{H_{tan}^2} \\
   &\; +  \|Z^{\be_h} \p_3 u_h\|_{L^2}^{\f12}  \|\p_1 Z^{\be_h} \p_3 u_h\|_{L^2}^{\f12}
   \|Z^{\ah-\be_h} \nabla_h \vro\|_{H_{tan}^1}^{\f12}   \|\p_3 Z^{\ah-\be_h} \nabla_h  \vro\|_{H_{tan}^1}^{\f12} \Big) \|Z^{\ah} \p_3 \vro\|_{L^2}\\
   \lesssim &\; (\me(t)^{\f13}+\me(t)^{\f12})\bmd_{tan}(t),
   \end{align*}
   where we have used the estimate \eqref{est-p32-u3-uh} and the fact $ \| \p_3 u_h \|_{L^{\infty}} \lesssim \| w_h \|_{L^{\infty}} + \| \nabla_h u_3 \|_{L^{\infty}} \lesssim \me(t)^{\f13}+\me(t)^{\f12}$. Similarly, we can obtain
   \beqq
   II_{1,1} + II_{1,5}  \lesssim \sqrt{\me(t)}\bmd_{tan}(t).
   \deqq
   It should be pointed out that the estimate of term $II_{1,3}$ is somewhat complicated.  We first decompose this term into two components similar to the estimate \eqref{ah1-II14}:	
   \beqq
   \bal
   II_{1,3} \lesssim &\; \sum_{0< \beta_h \le \alpha_h}C^{\beta_h}_{\alpha_h}
				\i |Z^{\beta_h} u_3 \cdot  Z^{\al_h-\beta_h} \p_3^2 \vro \cdot Z^{\al_h} \p_3 \vro| (1-\chi)^2 \ dx \\
               &\;  + \sum_{0< \beta_h \le \alpha_h}C^{\beta_h}_{\alpha_h} \i |Z^{\beta_h} u_3 \cdot  Z^{\al_h-\beta_h} \p_3^2 \vro \cdot Z^{\al_h} \p_3 \vro| \chi^2 \ dx.
   \dal \deqq
   If $\be_h=e_i$ for $i=1,2$, using the $L^{\infty}$ anisotropic inequality $\eqref{ie:Sobolev}_1$ and interpolation inequality \eqref{a6}, we can estimate
   \begin{align*}
   &\; \i |Z^{e_i} u_3 \cdot  Z^{\al_h-e_i} \p_3^2 \vro \cdot Z^{\al_h} \p_3 \vro| (1-\chi)^2 \ dx \\
   \lesssim &\; \| Z^{e_i} u_3\|_{L^{\infty}}  \|Z^{\al_h-e_i} Z_3 \p_3 \vro\|_{L^2}
   \|Z^{\al_h} \p_3 \vro\|_{L^2}\\
   \lesssim &\; \|\nabla_h u_3\|_{H_{tan}^2}^{\f12} \|\nabla_h \p_3 u_3\|_{H_{tan}^2}^{\f12} (\|Z^{\al_h-e_i} \p_3 \vro  \|_{L^2} +  \|Z^{\al_h-e_i} \p_3 \vro  \|_{L^2}^{\f12}  \|Z^{\al_h-e_i} Z_3^2 \p_3 \vro  \|_{L^2}^{\f12})  \|Z^{\al_h} \p_3 \vro\|_{L^2}\\
   \lesssim &\; \sqrt{\me(t)}\bmd_{tan}(t),\\
   \end{align*}
   and
 \begin{align*}
   &\;  \i |Z^{e_i} u_3 \cdot  Z^{\al_h-e_i} \p_3^2 \vro \cdot Z^{\al_h} \p_3 \vro| \chi^2 \ dx\\
   \lesssim &\; \| Z^{e_i} \p_3 u_3\|_{L^{\infty}}  \|Z^{\al_h-e_i} Z_3 \p_3 \vro\|_{L^2}
   \|Z^{\al_h} \p_3 \vro\|_{L^2}\\
    \lesssim &\; \|\nabla_h \p_3 u_3\|_{H_{tan}^2}^{\f12} \|\nabla_h \p_3^2 u_3\|_{H_{tan}^2}^{\f12}\|Z^{\al_h-e_i} Z_3 \p_3 \vro\|_{L^2}
   \|Z^{\al_h} \p_3 \vro\|_{L^2} \\
   \lesssim &\; \sqrt{\me(t)}\bmd_{tan}(t).
   \end{align*}
   If $2 \le |\be_h | \le [\f{|\ah|}{2}]$,  using the anisotropic type inequality $\eqref{ie:Sobolev}_2$, we can obtain
   \beqq \bal
    &\; \i |Z^{\be_h} u_3 \cdot  Z^{\al_h-\be_h} \p_3^2 \vro \cdot Z^{\al_h} \p_3 \vro| (1-\chi)^2 \ dx + \i |Z^{\beta_h} u_3 \cdot  Z^{\al_h-\beta_h} \p_3^2 \vro \cdot Z^{\al_h} \p_3 \vro| \chi^2 \ dx \\
    \lesssim &\; \| ( Z^{\be_h} u_3, Z^{\be_h} \p_3 u_3)\|_{L^2}^{\f14}  \| \p_3 ( Z^{\be_h} u_3, Z^{\be_h} \p_3 u_3)\|_{L^2}^{\f14} \| \p_1 ( Z^{\be_h} u_3, Z^{\be_h} \p_3 u_3)\|_{L^2}^{\f14} \| \p_{13} ( Z^{\be_h} u_3, Z^{\be_h} \p_3 u_3)\|_{L^2}^{\f14} \\
    &\; \times \|Z^{\al_h-\be_h} Z_3 \p_3 \vro\|_{L^2}^{\f12} \|\p_2 Z^{\al_h-\be_h} Z_3 \p_3 \vro\|_{L^2}^{\f12}
   \|Z^{\al_h} \p_3 \vro\|_{L^2} \\
     \lesssim &\; \sqrt{\me(t)}\bmd_{tan}(t).\\
   \dal \deqq
    If $[\f{|\ah|}{2}] < |\be_h | \le |\ah|$, we can obtain
   \beqq \bal
    &\; \i |Z^{\be_h} u_3 \cdot  Z^{\al_h-\be_h} \p_3^2 \vro \cdot Z^{\al_h} \p_3 \vro| (1-\chi)^2 \ dx + \i |Z^{\beta_h} u_3 \cdot  Z^{\al_h-\beta_h} \p_3^2 \vro \cdot Z^{\al_h} \p_3 \vro| \chi^2 \ dx \\
    \lesssim &\; \Big(\| Z^{\be_h} u_3\|_{L^2}^{\f14}  \| \p_3 Z^{\be_h} u_3\|_{L^2}^{\f14} \| \p_1 Z^{\be_h} u_3\|_{L^2}^{\f14} \| \p_{13}  Z^{\be_h} u_3\|_{L^2}^{\f14}  \|Z^{\al_h-\be_h} Z_3 \p_3 \vro\|_{L^2}^{\f12} \|\p_2 Z^{\al_h-\be_h} Z_3 \p_3 \vro\|_{L^2}^{\f12}\\
    &\; +   \|  Z^{\be_h}\p_3 u_3\|_{L^2}^{\f12} \|   Z^{\be_h} \p_{3}^2 u_3\|_{L^2}^{\f12}  \|Z^{\al_h-\be_h} Z_3 \p_3 \vro\|_{H_{tan}^2} \Big)
   \|Z^{\al_h} \p_3 \vro\|_{L^2} \\
     \lesssim &\; \sqrt{\me(t)}\bmd_{tan}(t).
   \dal \deqq
   Thus, the combination of the above estimates yields directly
   \beqq
II_{1,3} \lesssim \sqrt{\me(t)}\bmd_{tan}(t),
   \deqq
   which, together with the estimates of $II_{1,1}$, $II_{1,2}$, $II_{1,4}$ and $II_{1,5}$, we have
    \beqq
II_{1} \lesssim (\me(t)^{\f13}+\me(t)^{\f12})\bmd_{tan}(t).
   \deqq
   For $i=1,2$, integrating by parts, we can check that
\begin{align*}
II_{2} = &\; \i Z^{\ah} ( \vro\, \p_3 \vro ) \cdot  Z^{\ah} \p_3^2 u_3 \ dx +  \i Z^{\ah-e_i} \p_3 ( \vro\, \nabla_h \vro ) \cdot  Z^{\ah + e_i} \p_3 u_h \ dx\\
 =&\; \sum_{0 \le \beta_h \le \alpha_h }C^{\beta_h}_{\alpha_h} \i Z^{\be_h} \vro \, Z^{\ah-\be_h} \p_3 \vro  \cdot Z^{\al_h} \p_3^2 u_3 \ dx \\
 &\; +\sum_{0 < \beta_h \le \alpha_h-e_i}C^{\beta_h}_{\alpha_h-e_i} \i Z^{\be_h} \vro \, Z^{\ah- e_i -\be_h} \nabla_h \p_3 \vro \cdot Z^{\al_h+e_i} \p_3 u_h \ dx\\
 &\; +\sum_{0 < \beta_h \le \alpha_h-e_i}C^{\beta_h}_{\alpha_h-e_i} \i Z^{\be_h} \p_3  \vro \, Z^{\ah- e_i -\be_h} \nabla_h \vro \cdot Z^{\al_h+e_i} \p_3 u_h \ dx.
 \end{align*}
 By Lemma \ref{lemm:sobolev-ie}, one can arrive that
 \begin{align*}
 &\; \sum_{0 \le \beta_h \le \alpha_h }C^{\beta_h}_{\alpha_h} \i Z^{\be_h} \vro \, Z^{\ah-\be_h} \p_3 \vro  \cdot Z^{\al_h} \p_3^2 u_3 \ dx  \\
 \lesssim &\; \Big(\| \vro\|_{L^{\infty}}  \|Z^{\ah } \p_3 \vro \|_{L^2}  + \sum_{0 < \be_h \le [\f{\ah}{2}]}  \| Z^{\be_h} \vro \|_{H_{tan}^1}^{\f12} \| \p_3 Z^{\be_h} \vro \|_{H_{tan}^1}^{\f12}   \|Z^{\ah -\be_h} \p_3 \vro \|_{L^2}^{\f12}  \|\p_2 Z^{\ah -\be_h} \p_3 \vro \|_{L^2}^{\f12}\\
 &\; + \sum_{[\f{\ah}{2}] < \be_h \le \ah }  \| Z^{\be_h} \vro \|_{L^2}^{\f12} \| \p_3 Z^{\be_h} \vro \|_{L^2}^{\f12}   \|Z^{\ah -\be_h} \p_3 \vro \|_{H_{tan}^2} \Big) \| Z^{\al_h} \p_3^2 u_3 \|_{L^2}\\
\lesssim &\; \sqrt{\me(t)}\bmd_{tan}(t).
\end{align*}
Similarly, we can check that
\beqq \bal
&\; \sum_{0 < \beta_h \le \alpha_h-e_i}C^{\beta_h}_{\alpha_h-e_i} \i Z^{\be_h} \vro \, Z^{\ah- e_i -\be_h} \nabla_h \p_3 \vro \cdot Z^{\al_h+e_i} \p_3 u_h \ dx\\
 &\; +\sum_{0 < \beta_h \le \alpha_h-e_i}C^{\beta_h}_{\alpha_h-e_i} \i Z^{\be_h} \p_3  \vro \, Z^{\ah- e_i -\be_h} \nabla_h \vro \cdot Z^{\al_h+e_i} \p_3 u_h \ dx\\
 \lesssim &\; \sqrt{\me(t)}\bmd_{tan}(t),
\dal \deqq
which, together with above estimate, we have
\beqq
 II_2 \lesssim \sqrt{\me(t)}\bmd_{tan}(t).
\deqq
Now we estimate the term $II_{3}$. For $i=1,2$, integrating by parts, we decompose the term $II_3$ as follows:
\beqq \bal
II_3 = &\;  \i Z^{\ah}(u_h \cdot \nabla_h u_3 + u_3 \, \p_3 u_3) \cdot Z^{\al_h} \p_3^2 u_3  \ dx + \i Z^{\ah- e_i} \p_3 (u_h \cdot \nabla_h u_h)  \cdot Z^{\al_h+e_i} \p_3 u_h  \ dx \\
&\; + \i Z^{\ah- e_i}( \p_3  u_3 \, \p_3 u_h) \cdot Z^{\al_h+e_i} \p_3 u_h   \ dx
+ \i Z^{\ah- e_i}( u_3 \,  \p_3^2 u_h) \cdot Z^{\al_h+e_i} \p_3 u_h   \ dx \\
:= &\; \sum_{i=1}^{4} II_{3,i}.
\dal \deqq
 By Lemma \ref{lemm:sobolev-ie}, we can easily check that
 \beqq 
II_{3,1} + II_{3,2} + II_{3,3} \lesssim  \sqrt{\me(t)}\bmd_{tan}(t).
 \deqq
Similar to the estimate of $II_{1,3}$, we will decompose the term $II_{3,4}$ as follows:
\begin{align*}
II_{3,4} = &\; \i  u_3 \, Z^{\ah- e_i} \p_3^2 u_h \cdot Z^{\al_h+e_i} \p_3 u_h   \ dx  + \sum_{0 < \beta_h \le \alpha_h-e_i}C^{\beta_h}_{\alpha_h-e_i} \i Z^{\be_h} u_3 \, Z^{\ah- e_i -\be_h} \p_3^2 u_h  \cdot Z^{\al_h+e_i} \p_3 u_h   \ dx\\
= &\;  - \i Z^{\ah} \p_3 u_h \cdot \, ( u_3 \, Z^{\al_h}  \p_3^2 u_h  + Z^{e_i} u_3 \,  \cdot Z^{\al_h-e_i} \p_3^2 u_h  ) \ dx \\
&\; + \sum_{0 < \beta_h \le \alpha_h-e_i}C^{\beta_h}_{\alpha_h-e_i} \i Z^{\be_h} u_3 \, Z^{\ah- e_i -\be_h} \p_3^2 u_h  \cdot Z^{\al_h+e_i} \p_3 u_h  (\chi +1-\chi)^2 \ dx\\
\lesssim &\;  \f12 \i \p_3 u_3 |Z^{\ah} \p_3 u_h |^2 \ dx + \i |Z^{\ah} \p_3 u_h \cdot \,  Z^{e_i} u_3 \,  \cdot Z^{\al_h-e_i} \p_3^2 u_h |   \chi^2  \ dx\\
&\; + \sum_{0 < \beta_h \le \alpha_h-e_i}C^{\beta_h}_{\alpha_h-e_i} \i |Z^{\be_h} u_3 \, Z^{\ah- e_i -\be_h} \p_3^2 u_h  \cdot Z^{\al_h+e_i} \p_3 u_h|  \chi^2   \ dx\\
&\; + \i |Z^{\ah} \p_3 u_h \cdot \,  Z^{e_i} u_3 \,  \cdot Z^{\al_h-e_i} \p_3^2 u_h  | (1-\chi)^2 \ dx \\
&\; + \sum_{0 < \beta_h \le \alpha_h-e_i}C^{\beta_h}_{\alpha_h-e_i} \i |Z^{\be_h} u_3 \, Z^{\ah- e_i -\be_h} \p_3^2 u_h  \cdot Z^{\al_h+e_i} \p_3 u_h | (1-\chi)^2 \ dx.
\end{align*}
It is easy to check that
\beqq
  \i \p_3 u_3 |Z^{\ah} \p_3 u_h |^2 \ dx \lesssim \| \p_3 u_3 \|_{L^{\infty}} \|Z^{\ah} \p_3 u_h \|_{L^2}^2 \lesssim  \sqrt{\me(t)}\bmd_{tan}(t).
\deqq
On one hand, near the boundary, by writing $u_3 = \int_{0}^{x_3} \p_3 u_3 \ dz$,  we have
\beqq \bal
&\; \i |Z^{\ah} \p_3 u_h \cdot \,  Z^{e_i} u_3 \,  \cdot Z^{\al_h-e_i} \p_3^2 u_h |   \chi^2  \ dx \\
\lesssim &\; \|Z^{\ah} \p_3 u_h \|_{L^2}^{\f12}  \|\p_1 Z^{\ah} \p_3 u_h \|_{L^2}^{\f12}  \|Z^{e_i} \p_3 u_3 \|_{L^2}^{\f12} \|Z^{e_i} \p_3^2 u_3 \|_{L^2}^{\f12} \|Z^{\al_h-e_i} Z_3 \p_3 u_h\|_{L^2}^{\f12}  \|\p_2 Z^{\al_h-e_i} Z_3 \p_3 u_h\|_{L^2}^{\f12}, \\
\lesssim &\; \sqrt{\me(t)}\bmd_{tan}(t),
\dal \deqq
and
\beqq \bal
&\; \sum_{0 < \beta_h \le \alpha_h-e_i}C^{\beta_h}_{\alpha_h-e_i} \i |Z^{\be_h} u_3 \, Z^{\ah- e_i -\be_h} \p_3^2 u_h  \cdot Z^{\al_h+e_i} \p_3 u_h|  \chi^2   \ dx\\
\lesssim &\; \Big( \sum_{0 < \beta_h \le [\f{\alpha_h-e_i}{2}]} \|Z^{\be_h} \p_3 u_3 \|_{H_{tan}^1}^{\f12} \|Z^{\be_h} \p_3^2 u_3 \|_{H_{tan}^1}^{\f12} \|Z^{\al_h-e_i-\be_h} Z_3 \p_3 u_h\|_{L^2}^{\f12}  \|\p_2 Z^{\al_h-e_i-\be_h} Z_3 \p_3 u_h\|_{L^2}^{\f12} \\
&\; +  \sum_{[\f{\alpha_h-e_i}{2}] < \beta_h \le \alpha_h-e_i }  \|Z^{\be_h} \p_3 u_3 \|_{L^2}^{\f12} \|Z^{\be_h} \p_3^2 u_3 \|_{L^2}^{\f12} \|Z^{\al_h-e_i-\be_h} Z_3 \p_3 u_h\|_{H_{tan}^2}
\Big)\|Z^{\ah+e_i} \p_3 u_h \|_{L^2},\\
\lesssim &\; \sqrt{\me(t)}\bmd_{tan}(t).
\dal \deqq
On the other hand, away from the boundary, we will use the interpolation inequality \eqref{a6} to obtain that
\beqq \bal
&\; \i |Z^{\ah} \p_3 u_h \cdot \,  Z^{e_i} u_3 \,  \cdot Z^{\al_h-e_i} \p_3^2 u_h |   (1-\chi)^2  \ dx \\
\lesssim &\; \|Z^{\ah} \p_3 u_h \|_{L^2}^{\f12}  \|\p_1 Z^{\ah} \p_3 u_h \|_{L^2}^{\f12}  \|Z^{e_i} u_3 \|_{L^2}^{\f14} \|\p_2 Z^{e_i} u_3 \|_{L^2}^{\f14}  \|\p_3 Z^{e_i} u_3 \|_{L^2}^{\f14}  \|\p_{23} Z^{e_i} u_3 \|_{L^2}^{\f14}  \|Z^{\al_h-e_i} Z_3 \p_3 u_h\|_{L^2} \\
\lesssim &\; \|Z^{\ah} \p_3 u_h \|_{L^2}^{\f12}  \|\p_1 Z^{\ah} \p_3 u_h \|_{L^2}^{\f12}  \|Z^{e_i} u_3 \|_{L^2}^{\f14} \|\p_2 Z^{e_i} u_3 \|_{L^2}^{\f14}  \|\p_3 Z^{e_i} u_3 \|_{L^2}^{\f14}  \|\p_{23} Z^{e_i} u_3 \|_{L^2}^{\f14}  \\
&\; \times (\|Z^{\al_h-e_i} \p_3 u_h\|_{L^2} + \|Z^{\al_h-e_i}\p_3 u_h\|_{L^2}^{\f12} \|Z^{\al_h-e_i} Z_3^2 \p_3 u_h\|_{L^2}^{\f12} )  \\
\lesssim &\; \sqrt{\me(t)}\bmd_{tan}(t),
\dal \deqq
and for $i=1,2$,
\begin{align*}
&\; \sum_{0 < \beta_h \le \alpha_h-e_i}C^{\beta_h}_{\alpha_h-e_i} \i |Z^{\be_h} u_3 \, Z^{\ah- e_i -\be_h} \p_3^2 u_h  \cdot Z^{\al_h+e_i} \p_3 u_h|  (1-\chi)^2   \ dx\\
\lesssim &\; \Big(\|Z^{e_j}  u_3 \|_{L^2}^{\f14} \| \p_3 Z^{e_j} u_3 \|_{L^2}^{\f14}  \|\p_1 Z^{e_j}  u_3 \|_{L^2}^{\f14} \| \p_{13} Z^{e_j} u_3 \|_{L^2}^{\f14} \|Z^{\al_h-e_i-e_j} Z_3 \p_3 u_h\|_{L^2}^{\f12}  \|\p_2 Z^{\al_h-e_i-e_j} Z_3 \p_3 u_h\|_{L^2}^{\f12}
\\
&\; + \sum_{e_j < \beta_h \le [\f{\alpha_h-e_i}{2}]} \|Z^{\be_h}  u_3 \|_{H_{tan}^1}^{\f12} \|Z^{\be_h} \p_3 u_3 \|_{H_{tan}^1}^{\f12} \|Z^{\al_h-e_i-\be_h} Z_3 \p_3 u_h\|_{L^2}^{\f12}  \|\p_2 Z^{\al_h-e_i-\be_h} Z_3 \p_3 u_h\|_{L^2}^{\f12} \\
&\; +  \sum_{[\f{\alpha_h-e_i}{2}] < \beta_h \le \alpha_h-e_i }  \|Z^{\be_h}  u_3 \|_{L^2}^{\f12} \|Z^{\be_h} \p_3 u_3 \|_{L^2}^{\f12} \|Z^{\al_h-e_i-\be_h} Z_3 \p_3 u_h\|_{H_{tan}^2}
\Big)\|Z^{\ah+e_i} \p_3 u_h \|_{L^2} \\
\lesssim &\; \sqrt{\me(t)}\bmd_{tan}(t),
\end{align*}
where we have used the following estimate
\begin{align*}
\|\p_2 Z^{\al_h-e_i-e_j} Z_3 \p_3 u_h\|_{L^2} \lesssim  \|\p_2 Z^{\al_h-e_i-e_j} \p_3 u_h\|_{L^2} +  \|\p_2 Z^{\al_h-e_i-e_j} \p_3 u_h\|_{L^2}^{\f12} \|\p_2 Z^{\al_h-e_i-e_j} Z_3^2 \p_3 u_h\|_{L^2}^{\f12}.
\end{align*}
Thus, the combination of the above estimates yields that
\beqq
II_{3,4} \lesssim  \sqrt{\me(t)}\bmd_{tan}(t).
\deqq
Combining the estimates from $II_{3,1}$ to $II_{3,3}$, we can obtain that
\beqq
II_{3} \lesssim \sqrt{\me(t)}\bmd_{tan}(t).
\deqq
Similarly, for $i=1,2$, integrating by parts, we can check that
\begin{align*}
II_4 = &\; \i  Z^{\ah - e_i } \p_3 ( \f{\vro}{1+\vro}\, \Delta_h u ) \cdot  Z^{\ah+e_i} \p_3 u \ dx \\
\lesssim &\; \sqrt{\me(t)}\bmd_{tan}(t),\\
II_5 = &\; \i Z^{\ah} ( \f{\vro}{1+\vro} \, \p_3 \du ) \cdot  Z^{\ah} \p_3^2 u_3 \ dx +  \i Z^{\ah-e_i} \p_3 ( \f{\vro}{1+\vro}\, \nabla_h \du ) \cdot  Z^{\ah + e_i} \p_3 u_h \ dx\\
\lesssim &\; \sqrt{\me(t)}\bmd_{tan}(t),
\end{align*}
and
\beqq \bal
II_6 = &\; \var \sum_{0 < \beta_h \le \alpha_h}C^{\beta_h}_{\alpha_h}  Z^{\be_h} (\f{\vro}{1+\vro}) Z^{\al_h - \be_h} \p_3^2 u \cdot  Z^{\al_h } \p_3^2 u \ dx
\lesssim  \sqrt{\me(t)}\bmd_{tan}(t).
\dal \deqq
Substituting the estimates from $II_{1}$ to $II_{6}$ into \eqref{32-equ}, we have, for $|\ah| = m-2$,
\beq\label{3204}  \bal
    \f12 \frac{d}{dt} \|Z^{\ah} \p_3(\vro, u)(t)\|_{L^2}^2
				+\|Z^{\ah} \p_3(\nabla_h u, \du)(t)\|_{L^2}^2 +\var \|Z^{\ah} \p_3^2 u(t)\|_{L^2}^2
				\lesssim  (\me(t)^{\f13}+\me(t)^{\f12})\bmd_{tan}(t).
\dal \deq
{\bf{Finally, we deal with the case  $|\ah| = m-1$.}} Similar to the estimates of the case $|\ah|=m-2$, we can check that if $|\ah| = m-1$,
\beq\label{3205}  \bal
    \f12 \frac{d}{dt} \|Z^{\ah} \p_3(\vro, u)(t)\|_{L^2}^2
				+\|Z^{\ah} \p_3(\nabla_h u, \du)(t)\|_{L^2}^2 +\var \|Z^{\ah} \p_3^2 u(t)\|_{L^2}^2
				\lesssim  (\me(t)^{\f13}+\me(t)^{\f12})\md_{tan}(t),
\dal \deq
without using any interpolation inequalities for the terms $II_{1,3}$ and $II_{3,4}$.
Combining the estimates \eqref{3202}, \eqref{3203}-\eqref{3205}, we finish the proof of this lemma.
   \end{proof}

   Finally, in order to close the energy estimate, we will estimate the conormal energy norm for the one-order vertical derivative for the density and velocity field.
   \begin{lemm}\label{lemma34}
				For any smooth solution $(\vro,u)$ of equation \eqref{eqr}, under the assumption \eqref{assumption},
				it holds for any positive integer $\ 1 \le k \le m-1$,
				\beq \label{3301} \bal
				&\;  \f{d}{dt}\|\p_3(\vro, u)\|_{H_{co}^{k}}^2 + \sum_{0 \le |\al| \le k-1} \f{d}{dt}\i Z^{\al} u \cdot Z^{\al} \nabla  \vro \ dx +\f32\sum_{|\al|=k} \|\p_3 Z^{\al} \nabla_h  u\|_{L^2}^2   \\
                &\; + \f32 \sum_{|\al|=k}  (\| \p_3  Z^{\al} \du\|_{L^2}^2 +\var   \|\p_3^2 Z^{\al}u\|_{L^2}^2 ) \\
				\lesssim &\;  \|\p_3(\nabla_h u, \du)\|_{H_{co}^{k-1}}^2  +\var \|\p_3^2 u\|_{H_{co}^{k-1}}^2  +\|\du\|_{H^{k-1}_{co}}^2 + \| (\nabla_h u, \Delta_h u)\|_{H_{co}^{k-1}}^2  +  (\me(t)^{\f13}+\me(t)^{\f12})\md(t).
            \dal \deq
            \begin{proof}
            For any $1\le |\al|=k \le m-1$, we only deal with the case $\al_3 \ge 1 $, since we have already dealt with the case $\al_3=0$, see Lemma \ref{lemma32}. Applying $\p_3 Z^{\al} $-operator to $\eqref{eqr}_1$ and multiplying by $\p_3 Z^{\al} \vro$ and integrating over $\mathbb{R}_+^3$, we have
                \begin{equation} \label{Zal-vro}
					\begin{aligned}
						\frac{d}{dt}\frac{1}{2}\i| \p_3 Z^{\al} \vro|^2 \ dx+ \i  \p_3 Z^{\al} \du\cdot \p_3 Z^{\al} \vro \ dx
						=\i \p_3 Z^{\al}  (-u\cdot \nabla \vro -\vro \, \du )\cdot \p_3 Z^{\al} \vro \ dx.
					\end{aligned}
				\end{equation}
                Next, applying $Z^{\al}$-operator to $\eqref{eqr}_2$ and multiplying by $-\p_3^2 Z^{\al} u$ and integrating over $\mathbb{R}_+^3$, we have
\begin{equation} \label{Zal-u}
					\begin{aligned}
						&\; \frac{d}{dt}\frac{1}{2}\i |\p_3 Z^{\al} u|^2 dx
						+\|\p_3 Z^{\al} \nabla_h  u\|_{L^2}^2 +\| \p_3  Z^{\al} \du\|_{L^2}^2 + \var \|\p_3^2 Z^{\al}u\|_{L^2}^2 - \i Z^{\al} \nabla \vro \cdot \p_3^2 Z^{\al} u \ dx \\
						=&\;\i Z^{\al} (\vro\, \nabla \vro+u\cdot \nabla u+\f{\vro}{1+\vro} (\Delta_h u+ \nabla \du +\var  \p_3^2 u))\cdot \p_3^2 Z^{\al}  u\ dx\\
                        &\; - \i \p_3 Z^{\al} \du \cdot (\p_3^2 Z^{\al} u_3 - \p_3 Z^{\al} \p_3 u_3)\ dx
                        - \i ( Z^{\al} \p_3 \du - \p_3 Z^{\al} \du ) \cdot \p_3^2 Z^{\al} u_3\ dx  \\
                       &\; - \var \i (Z^{\al} \p_3^2  u - \p_3^2 Z^{\al} u) \cdot \p_3^2 Z^{\al} u\ dx,
					\end{aligned}
				\end{equation}
                where we have used that
                \begin{align*}
                &\; \i Z^{\al} \nabla \du \cdot \p_3^2 Z^{\al} u \ dx
               =  \i Z^{\al} \p_3 \du \cdot \p_3^2 Z^{\al} u_3 \ dx + \i Z^{\al} \nabla_h \du \cdot \p_3^2 Z^{\al} u_h \ dx  \\
               = &\; \i (Z^{\al}  \p_3  \du -\p_3 Z^{\al}  \du ) \cdot  \p_3^2 Z^{\al} u_3 \ dx  + \i \p_3 Z^{\al}  \du \cdot (\p_3^2 Z^{\al} u_3 -\p_3 Z^{\al} \p_3 u_3 + \p_3 Z^{\al} \p_3 u_3  ) \ dx \\
               &\; + \i \p_3 Z^{\al} \du \cdot \p_3 Z^{\al} \nabla_h  \cdot u_h \ dx  \\
                = &\; \i ( Z^{\al} \p_3 \du - \p_3 Z^{\al} \du ) \cdot \p_3^2 Z^{\al} u_3\ dx + \i \p_3 Z^{\al} \du \cdot (\p_3^2 Z^{\al} u_3 - \p_3 Z^{\al} \p_3 u_3)\ dx +\| \p_3  Z^{\al} \du\|_{L^2}^2,
                \end{align*}
                and
                 \beqq \bal
                  \var \i Z^{\al} \p_3^2 u \cdot \p_3^2 Z^{\al} u \ dx
                  =  \var \|\p_3^2 Z^{\al} u\|_{L^2}^2 + \var \i (Z^{\al} \p_3^2  u - \p_3^2 Z^{\al} u) \cdot \p_3^2 Z^{\al} u\ dx.
                 \dal \deqq
                Notice that
                \begin{align*}
              &\; - \i Z^{\al} \nabla \vro \cdot \p_3^2 Z^{\al} u \ dx = - \i Z^{\al} \p_3 \vro \cdot \p_3^2 Z^{\al} u_3 \ dx - \i Z^{\al} \nabla_h \vro \cdot \p_3^2 Z^{\al} u_h \ dx  \\
              = &\;- \i  Z^{\al} \p_3 \vro \cdot (\p_3^2 Z^{\al} u_3 -\p_3 Z^{\al} \p_3 u_3) \ dx - \i  Z^{\al} \p_3 \vro \cdot \p_3 Z^{\al} \p_3 u_3  \ dx  \\
              &\; - \i \p_3 Z^{\al}  \vro \cdot \p_3 Z^{\al} \nabla_h \cdot u_h \ dx \\
              = &\;- \i  Z^{\al} \p_3 \vro \cdot (\p_3^2 Z^{\al} u_3 -\p_3 Z^{\al} \p_3  u_3) \ dx - \i (Z^{\al} \p_3 \vro- \p_3 Z^{\al}  \vro) \cdot \p_3 Z^{\al} \p_3 u_3 \ dx \\
              &\;- \i \p_3 Z^{\al}  \vro \cdot \p_3 Z^{\al} \du \ dx.
               \end{align*}
            Thus, combining the above equality  and the equalities \eqref{Zal-vro}-\eqref{Zal-u}, we can obtain that
					\begin{align*}
						&\; \frac{d}{dt}\frac{1}{2}\i(| \p_3 Z^{\al} \vro|^2+|\p_3 Z^{\al}  u|^2)dx
						+\|\p_3 Z^{\al} \nabla_h u\|_{L^2}^2 +\| \p_3Z^{\al}  \du\|_{L^2}^2 + \var \|\p_3^2 Z^{\al}u\|_{L^2}^2 \\
						=&\;\i  \p_3 Z^{\al} (-u\cdot \nabla \vro -\vro \, \du )\cdot  \p_3 Z^{\al} \vro \ dx
						+\i Z^{\al} (\vro\, \nabla \vro)\cdot \p_3^2 Z^{\al} u\ dx \notag\\
						& \;+\i Z^{\al} (u\cdot \nabla u)\cdot \p_3^2 Z^{\al} u\ dx+\i Z^{\al} (\f{\vro}{1+\vro} \Delta_h u)\cdot \p_3^2 Z^{\al} u\ dx  \\
                        &\;+\i Z^{\al} (\f{\vro}{1+\vro} \nabla \du)\cdot \p_3^2 Z^{\al} u\ dx +\var \i Z^{\al} (\f{\vro}{1+\vro}  \p_3^2 u)\cdot \p_3^2 Z^{\al} u\ dx\\
                        &\; - \i \p_3 Z^{\al} \du \cdot (\p_3^2 Z^{\al} u_3 - \p_3 Z^{\al} \p_3 u_3)\ dx
                        - \i ( Z^{\al} \p_3 \du - \p_3 Z^{\al} \du ) \cdot \p_3^2 Z^{\al} u_3\ dx  \notag\\
                       &\; - \var \i (Z^{\al} \p_3^2  u - \p_3^2 Z^{\al} u) \cdot \p_3^2 Z^{\al} u\ dx + \i  Z^{\al} \p_3 \vro \cdot (\p_3^2 Z^{\al} u_3 -\p_3 Z^{\al} \p_3 u_3) \ dx \\
                       &\; + \i (Z^{\al} \p_3 \vro- \p_3 Z^{\al}  \vro) \cdot \p_3 Z^{\al} \p_3 u_3 \ dx \\
				:= &\; \sum_{i=1}^{11} III_i.
					\end{align*}
                First we estimate the term $III_1$.
                \begin{align*}
          III_1 = &\; - \sum_{0 \le \beta \le \alpha}C^{\beta}_{\alpha} \i \p_3 Z^{\be} u \cdot Z^{\al-\be} \nabla \vro  \cdot \p_3 Z^{\al} \vro \ dx  - \sum_{0 \le \beta \le \alpha}C^{\beta}_{\alpha} \i  Z^{\be} u_h \cdot \p_3 Z^{\al-\be} \nabla_h \vro  \cdot \p_3 Z^{\al} \vro \ dx \\
          &\; - \sum_{0 \le \beta \le \alpha}C^{\beta}_{\alpha} \i  Z^{\be} u_3 \, \p_3 Z^{\al-\be} \p_3 \vro  \cdot \p_3 Z^{\al} \vro \ dx - \sum_{0 \le \beta \le \alpha}C^{\beta}_{\alpha} \i  \p_3 Z^{\be}  \vro \cdot  Z^{\al-\be} \du  \cdot \p_3 Z^{\al} \vro \ dx  \\
          &\; - \sum_{0 \le \beta \le \alpha}C^{\beta}_{\alpha} \i  Z^{\be} \vro \cdot \p_3 Z^{\al-\be} \du  \cdot \p_3 Z^{\al} \vro \ dx \\
          := &\; \sum_{i=1}^{5} III_{1,i}.
                 \end{align*}
                 By Lemma \ref{lemm:sobolev-ie}, we have
                \begin{align*}
                III_{1,1} \lesssim &\; (\| \p_3 u_h\|_{L^{\infty}} \| Z^{\al} \nabla_h \vro \|_{L^2} +  \| \p_3 u_3 \|_{H_{tan}^1}^{\f12} \| \p_3^2 u_3\|_{H_{tan}^1}^{\f12} \| Z^{\al} \p_3 \vro \|_{L^2} )\| \p_3 Z^{\al}  \vro \|_{L^2} \\
                &\; + \sum_{0 < \beta \le [\f{\alpha}{2}]} (\| \p_3 Z^{\be} u_h\|_{H_{tan}^2} \|  Z^{\al-\be} \nabla_h\vro\|_{L^2}^{\f12}  \| \p_3  Z^{\al-\be} \nabla_h \vro\|_{L^2}^{\f12} \\
                &\; + \| \p_3 Z^{\be} u_3\|_{H_{tan}^1}^{\f12} \| \p_3 Z^{\be} u_3\|_{H_{tan}^1}^{\f14} \| Z^{\al-\be} \p_3 \vro\|_{L^2}^{\f12}  \| \p_2  Z^{\al-\be} \p_3 \vro\|_{L^2}^{\f12}
                ) \| \p_3 Z^{\al}  \vro \|_{L^2} \\
                &\; + \sum_{[\f{\alpha}{2}] < \beta \le \al} (\| \p_3 Z^{\be} u_h\|_{L^2}^{\f12} \| \p_1 \p_3 Z^{\be} u_h\|_{L^2}^{\f12} \|  Z^{\al-\be} \nabla_h\vro\|_{H_{tan}^1}^{\f12}  \| \p_3  Z^{\al-\be} \nabla_h \vro\|_{H_{tan}^1}^{\f12} \\
                &\; + \| \p_3 Z^{\be} u_3\|_{L^2}^{\f12} \| \p_3^2 Z^{\be} u_3\|_{L^2}^{\f12} \| Z^{\al-\be} \p_3 \vro\|_{H_{tan}^2}
                ) \| \p_3 Z^{\al}  \vro \|_{L^2}\\
                \lesssim &\;(\me(t)^{\f13}+\me(t)^{\f12}) \md(t),
                \end{align*}
                and integrating by parts, we have
                \begin{align*}
                III_{1,2} = &\; - \f12 \i u_h \cdot \nabla_h  (|\p_3 Z^{\al} \vro|^2) \ dx - \sum_{0 < \beta \le \alpha}C^{\beta}_{\alpha} \i  Z^{\be} u_h \cdot \p_3 Z^{\al-\be} \nabla_h \vro  \cdot \p_3 Z^{\al} \vro \ dx \\
               = &\;   \f12 \i \nabla_h \cdot  u_h  \, |\p_3 Z^{\al} \vro|^2 \ dx - \sum_{0 < \beta \le \alpha}C^{\beta}_{\alpha} \i  Z^{\be} u_h \cdot \p_3 Z^{\al-\be} \nabla_h \vro  \cdot \p_3 Z^{\al} \vro \ dx \\
               \lesssim &\; \| \nabla_h u\|_{L^{\infty}} \|\p_3 Z^{\al} \vro \|_{L^2}^2 + \|Z^{e_i} u_h \|_{L^{\infty}} \| \p_3 Z^{\al-e_i} \nabla_h \vro\|_{L^2}  \| \p_3 Z^{\al} \vro\|_{L^2} \\
               &\; + \sum_{e_i < \beta \le \al} \| Z^{\be} u_h\|_{H_{tan}^1}^{\f12}\|\p_3 Z^{\be} u_h\|_{H_{tan}^1}^{\f12} \| \p_3 Z^{\al-\be} \nabla_h\vro\|_{L^2}^{\f12}  \| \p_{23}  Z^{\al-\be} \nabla_h \vro\|_{L^2}^{\f12} \| \p_3 Z^{\al} \vro\|_{L^2} \\
               \lesssim &\; \sqrt{\me(t)} \md(t),
                \end{align*}
                where $i=1,2,3$.
                Now we deal with the difficult term $III_{1,3}$. If $|\beta|=0$, integrating by parts, we can obtain
                \begin{align}\label{est-communtator}
                &\; - \i u_3 \,  \p_3 Z^{\al} \p_3 \vro \cdot \p_3 Z^{\al} \vro \, dx \notag\\
                = &\; - \f12 \i u_3 \,  \p_3(|\p_3 Z^{\al} \vro |^2) \, dx -  \i u_3 \,  \p_3 ( Z^{\al} \p_3 \vro-\p_3 Z^{\al} \vro) \cdot \p_3 Z^{\al} \vro \, dx \notag\\
                = &\; \f12 \i \p_3 u_3 \,|\p_3 Z^{\al} \vro |^2\, dx - \sum_{0< \be_3 \le \al_3} C_{\al_3}^{\be_3}   \i u_3 \,  \p_3 Z_3^{\be_3}(\f{1}{\vf})   \,  Z^{\al_{h}} Z_3^{\al_3-\be_3 +1 } \vro \cdot \p_3 Z^{\al} \vro \, dx \notag\\
                &\; - \sum_{0< \be_3 \le \al_3} C_{\al_3}^{\be_3}   \i u_3 \, Z_3^{\be_3}(\f{1}{\vf})  \, \p_3 Z^{\al_{h}}  Z_3^{\al_3-\be_3+1 } \vro \cdot \p_3 Z^{\al} \vro \, dx, 
               \end{align}
                where we have used that
                \beq \label{properity-p3}
		Z^{\al} \p_3 - \p_3 Z^{\al}= Z^{\al_{h}} Z_3^{\al_3} (\f{1}{\vf} Z_3) - \p_3 Z^{\al_{h}} Z_3^{\al_3}
		= \sum_{0< \be_3 \le \al_3} C_{\al_3}^{\be_3}     Z_3^{\be_3}(\f{1}{\vf})  \,  Z^{\al_{h}} Z_3^{\al_3-\be_3+1 }.
		\deq
        Furthermore, by induction, we can check that
\beqq
Z_3^{\be_3}(\f{1}{\vf}) \lesssim \f{1}{x_3}.
\deqq
         Using the above estimate and writing $u_3 =\int_{0}^{x_3} \p_3 u_3 \ dz$,  we can obtain that
        \beqq  \bal
         &\; - \sum_{0< \be_3 \le \al_3} C_{\al_3}^{\be_3}   \i u_3 \, Z_3^{\be_3}(\f{1}{\vf})  \,  \p_3 Z^{\al_{h}}  Z_3^{\al_3-\be_3+1 } \vro \cdot \p_3 Z^{\al} \vro \, dx\\
                \lesssim &\; \sum_{0< \be_3 \le \al_3} \i \f{1}{x_3} \cdot  x_3 \|\p_3 u_3 \|_{L_{x_3}^{\infty}} | \p_3 Z^{\al_{h}}  Z_3^{\al_3-\be_3+1 } \vro| \,|  \p_3 Z^{\al} \vro|\  dx \\
                \lesssim &\; \sum_{0< \be_3 \le \al_3} \|\p_3 u_3 \|_{L^{\infty}} \| \p_3 Z^{\al_{h}}  Z_3^{\al_3-\be_3+1 } \vro\|_{L^2}  \|  \p_3 Z^{\al} \vro\|_{L^2} \\
                \lesssim &\;\sqrt{\me(t)} \md(t),
        \dal \deqq
        and similarly, we can check that
         \beqq \bal
       \f12 \i \p_3 u_3 \,|\p_3 Z^{\al} \vro |^2\, dx - \sum_{0< \be_3 \le \al_3} C_{\al_3}^{\be_3}   \i u_3 \,  \p_3 Z_3^{\be_3}(\f{1}{\vf})   \,  Z^{\al_{h}} Z_3^{\al_3-\be_3 +1 } \vro \cdot \p_3 Z^{\al} \vro \, dx
        \lesssim \sqrt{\me(t)} \md(t).
         \dal \deqq
         Thus, we have
         \beqq \bal
       - \i u_3 \,  \p_3 Z^{\al} \p_3 \vro \cdot \p_3 Z^{\al} \vro \, dx
        \lesssim \sqrt{\me(t)} \md(t).
         \dal \deqq
         If $\be =e_i$, $i=1,2$, similar to the estimate \eqref{ah1-II14} of $II_{1,4}$, we have
         \beqq \bal
        &\; - \i Z^{e_i} u_3 \,  \p_3 Z^{\al-e_i} \p_3 \vro \cdot \p_3 Z^{\al} \vro \, dx \\
        \lesssim &\; \i | Z^{e_i} u_3 \,  \p_3 Z^{\al-e_i} \p_3 \vro \,\p_3 Z^{\al} \vro | \chi^2 \, dx + \i | Z^{e_i} u_3 \,  \p_3 Z^{\al-e_i} \p_3 \vro \,\p_3 Z^{\al} \vro | (1-\chi)^2 \, dx \\
        \lesssim &\; (\| Z^{e_i} \p_3 u_3\|_{L^{\infty}} +\| Z^{e_i} u_3\|_{L^{\infty}} )  \| Z_3 Z^{\al-e_i} \p_3 \vro\|_{L^2} \| \p_3 Z^{\al} \vro\|_{L^2} \\
         \lesssim &\;   \sqrt{\me(t)} \md(t),
         \dal \deqq
         and if $\be =e_3$, it is easy to check that
         \beqq \bal
         &\; - \i Z^{e_3} u_3 \,  \p_3 Z^{\al-e_3} \p_3 \vro \cdot \p_3 Z^{\al} \vro \, dx \\
         \lesssim &\;\|\p_3 u_3\|_{L^{\infty}}\| Z_3 Z^{\al-e_i} \p_3 \vro\|_{L^2}  \| \p_3 Z^{\al} \vro\|_{L^2} \\
         \lesssim &\;   \sqrt{\me(t)} \md(t).
         \dal \deqq
         If $2 \le |\be| \le  [\f{|\al|}{2}]$, we first consider the case $\be_3=0$.
         \begin{align*}
         &\; - \i Z^{\be_h} u_3 \,  \p_3 Z^{\al-\be_h} \p_3 \vro \cdot \p_3 Z^{\al} \vro \, dx \\
         \lesssim &\; \i | Z^{\be_h} u_3 \,  \p_3 Z^{\al-\be_h} \p_3 \vro \,\p_3 Z^{\al} \vro | \chi^2 \, dx + \i | Z^{\be_h} u_3 \,  \p_3 Z^{\al-\be_h} \p_3 \vro \,\p_3 Z^{\al} \vro | (1-\chi)^2 \, dx \\
         \lesssim &\; (\|Z^{\be_h} \p_3 u_3\|_{H_{tan}^1}^{\f12} \| Z^{\be_h} \p_{13} u_3\|_{H_{tan}^1}^{\f12} + \|Z^{\be_h}  u_3\|_{H_{tan}^1}^{\f12} \| Z^{\be_h} \p_{1} u_3\|_{H_{tan}^1}^{\f12} ) \\
         &\; \times \| Z_3 Z^{\al-\be_h} \p_3 \vro\|_{L^2}^{\f12} \|\p_2 Z_3 Z^{\al-\be_h} \p_3 \vro\|_{L^2}^{\f12}\| \p_3 Z^{\al} \vro\|_{L^2} \\
         \lesssim &\;   \sqrt{\me(t)} \md(t),
         \end{align*}
         and if $\be_3 \ge 1$,
         \beqq \bal
         &\; - \i Z^{\be} u_3 \,  \p_3 Z^{\al-\be} \p_3 \vro \cdot \p_3 Z^{\al} \vro \, dx \\
         \lesssim &\; \|\p_3 Z^{\be-e_3} u_3\|_{H_{tan}^1}^{\f12} \| \p_3^2 Z^{\be-e_3}  u_3\|_{H_{tan}^1}^{\f12} \| Z_3 Z^{\al-\be} \p_3 \vro\|_{L^2}^{\f12} \|\p_2 Z_3 Z^{\al-\be} \p_3 \vro\|_{L^2}^{\f12}\| \p_3 Z^{\al} \vro\|_{L^2} \\
         \lesssim &\;   \sqrt{\me(t)} \md(t).
         \dal \deqq
         Similarly, we can check that if $[\f{|\al|}{2}] < |\be| \le |\al|$,
         \beqq
         - \i Z^{\be_h} u_3 \,  \p_3 Z^{\al-\be_h} \p_3 \vro \,\p_3 Z^{\al} \vro \, dx
         \lesssim \sqrt{\me(t)} \md(t).
         \deqq
         Combining above estimates of the term $III_{1,3}$, we have
         \beqq
         III_{1,3}   \lesssim \sqrt{\me(t)} \md(t).
         \deqq
                Similar to the estimate of $III_{1,1}$, we can check that
                \beqq
                III_{1,4} +III_{1,5} \lesssim  \sqrt{\me(t)} \md(t),
                \deqq
                which, together with the estimates from $III_{1,1}$ to $III_{1,3}$, yields that
                \beqq
                III_{1} \lesssim  (\me(t)^{\f13}+\me(t)^{\f12})\md(t).
                \deqq
                Next we estimate the term $III_3$. Integrating by parts, we have
                \begin{align*}
III_{3} = &\; -\i \p_3Z^{\al}  (u_h \cdot \nabla_h u+ u_3 \, \p_3 u) \cdot \p_3 Z^{\al} u \ dx \\
= &\; -\sum_{0 \le \be \le \al} C_{\al}^{\be}  \i  (\p_3 Z^{\be} u_h \cdot Z^{\al-\be} \nabla_h u +  Z^{\be} u_h \cdot \p_3 Z^{\al-\be} \nabla_h u)\cdot \p_3 Z^{\al} u \ dx \\
&\;  -\sum_{0 \le \be \le \al} C_{\al}^{\be}  \i  \p_3 Z^{\be} u_3 \, Z^{\al-\be} \p_3 u\cdot \p_3 Z^{\al} u \ dx +  \i  u_3 \, \p_3 Z^{\al} \p_3 u \cdot \p_3 Z^{\al} u \ dx \\
&\; -\sum_{0 < \be \le \al} C_{\al}^{\be}  \i   Z^{\be} u_3 \, \p_3 Z^{\al-\be} \p_3 u \cdot \p_3 Z^{\al} u \ dx \\
:= &\; \sum_{i=1}^{4} III_{3,i}.
\end{align*}
By Lemma \ref{lemm:sobolev-ie}, we can check
\begin{align*}
III_{3,1} \lesssim &\;  \| \p_3  Z^{\al}  u\|_{L^2}^{\f12} \| \p_{13}  Z^{\al} u\|_{L^2}^{\f12} \sum_{0 \le \be \le \al}  \Big( \|\p_3 Z^{\be} u_h\|_{L^2}^{\f12} \|\p_{23} Z^{\be} u_h\|_{L^2}^{\f12}
                \|Z^{\al-\be} \nabla_h u\|_{L^2}^{\f12} \|\p_3 Z^{\al-\be} \nabla_h u\|_{L^2}^{\f12} \\
                &\; +  \|Z^{\be} u_h\|_{L^2}^{\f14} \|\p_2 Z^{\be} u_h\|_{L^2}^{\f14} \|\p_3 Z^{\be} u_h\|_{L^2}^{\f14}  \|\p_{23} Z^{\be} u_h\|_{L^2}^{\f14}
                \|\p_3 Z^{\al-\be} \nabla_h u\|_{L^2}\Big) \\
                \lesssim &\; \sqrt{\me(t)} \md(t),\\
                III_{3,2} \lesssim &\;\sum_{0 \le \be \le \al}  \| \p_3  Z^{\al}  u\|_{L^2}^{\f12} \| \p_{13}  Z^{\al} u \|_{L^2}^{\f12} \|\p_3 Z^{\be} u_3\|_{L^2}^{\f12} \|\p_3^2 Z^{\be} u_3\|_{L^2}^{\f12} \
                \|Z^{\al-\be} \p_3 u\|_{L^2}^{\f12} \|\p_2 Z^{\al-\be} \p_3 u\|_{L^2}^{\f12}\\
                \lesssim &\; \sqrt{\me(t)} \md(t).
\end{align*}
               Integrating by parts and using anisotropic inequalities \eqref{ie:Sobolev}, similar to estimate \eqref{est-communtator} of $III_{1,3}$, we have
                \beqq \bal
III_{3,3} = &\; - \i u_3 \p_3 (\p_3 Z^{\al} u +  Z^{\al}  \p_3 u -\p_3 Z^{\al} u) \cdot \p_3 Z^{\al} u \ dx \\
= &\; \f12 \i \p_3 u_3  |\p_3 Z^{\al} u|^2 \ dx - \i u_3  \p_3(Z^{\al}  \p_3 u -\p_3 Z^{\al} u) \cdot \p_3 Z^{\al} u \,\ dx\\
\lesssim &\; \sqrt{\me(t)} \md(t).
\dal \deqq
               It remains to deal with the term $III_{3,4}$. If $\be_3=0$, similar to the estimate \eqref{ah1-II14} of $II_{1,4}$, we can check
                \beqq \bal
                &\; -  \i   Z^{\be} u_3 \, \p_3 Z^{\al-\be} \p_3 u \cdot \p_3 Z^{\al} u \ dx \\
                \lesssim &\;  \i   |Z^{\be} u_3 \, \p_3 Z^{\al-\be} \p_3 u \cdot \p_3 Z^{\al} u| \chi^2 \ dx  +  \i   |Z^{\be} u_3 \, \p_3 Z^{\al-\be} \p_3 u \cdot \p_3 Z^{\al} u| (1- \chi)^2 \ dx \\
                \lesssim &\; \Big( \| Z^{\be} \p_3 u_3\|_{L^2}^{\f12} \|\p_3 Z^{\be} \p_3 u_3\|_{L^2}^{\f12} +\| Z^{\be}  u_3\|_{L^2}^{\f12} \|\p_3 Z^{\be} u_3\|_{L^2}^{\f12}
                ) \\ &\; \times \| Z_3  Z^{\al-\be} \p_3 u\|_{L^2}^{\f12} \| \p_1 Z_3  Z^{\al-\be} \p_3 u\|_{L^2}^{\f12}
                 \| \p_3  Z^{\al}  u\|_{L^2}^{\f12} \| \p_{13}  Z^{\al} u\|_{L^2}^{\f12} \\
              \lesssim &\; \sqrt{\me(t)} \md(t).
\dal \deqq
If $\be_3 \ge 1$, we have
                \beqq \bal
                &\; -  \i   Z^{\be} u_3 \, \p_3 Z^{\al-\be} \p_3 u \cdot \p_3 Z^{\al} u \ dx \\
                \lesssim &\; \| \p_3 Z^{\be -e_3}  u_3\|_{L^2}^{\f12} \|\p_3 Z^{\be -e_3} u_3\|_{L^2}^{\f12}  \| Z_3  Z^{\al-\be} \p_3 u\|_{L^2}^{\f12} \| \p_1 Z_3  Z^{\al-\be} \p_3 u\|_{L^2}^{\f12}
                 \| \p_3  Z^{\al}  u\|_{L^2}^{\f12} \| \p_{13}  Z^{\al} u\|_{L^2}^{\f12} \\
              \lesssim &\; \sqrt{\me(t)} \md(t).
\dal \deqq
Thus, we can obtain that
\beqq
III_{3,4} \lesssim  \sqrt{\me(t)} \md(t),
\deqq
which, together with the estimates from $III_{3,1}$ to $III_{3,3}$, yields that
\beqq
III_{3} \lesssim  \sqrt{\me(t)} \md(t).
\deqq
                Similarly, integrating by part, we can check that
                \begin{align*}
                 III_2= &\;   -\i \p_3 Z^{\al} ( \vro \,   \nabla_h \vro)\cdot \p_3 Z^{\al} u_h \ dx   + \i Z^{\al} ( \vro \,   \p_3 \vro)\cdot \p_3^2 Z^{\al} u_3 \ dx \\
                 \lesssim &\; \sqrt{\me(t)} \md(t),\\
III_{4} = &\; \sum_{0 \le \be \le \al} C_{\al}^{\be}  \i (\p_3 Z^{\be} (\f{\vro}{1+\vro}) \, Z^{\al-\be} \Delta_h u + Z^{\be} (\f{\vro}{1+\vro})\, \p_3 Z^{\al-\be} \Delta_h u )\cdot \p_3 Z^{\al} u \ dx \\
= &\; - \i Z^{\al} \nabla_h u \cdot \nabla_h \,\Big( \p_3 (\f{\vro}{1+\vro}) \, \p_3 Z^{\al} u \Big)\ dx
+\sum_{0 < \be \le \al} C_{\al}^{\be}  \i \p_3 Z^{\be} (\f{\vro}{1+\vro}) \, Z^{\al-\be} \Delta_h u  \cdot \p_3 Z^{\al} u \ dx\\
&\; -\sum_{0 \le \be < \al} C_{\al}^{\be}  \i \p_3 Z^{\al-\be} \nabla_h u \cdot \nabla_h \, \Big( Z^{\be} (\f{\vro}{1+\vro}) \, \p_3 Z^{\al} u \Big)\ dx  + \i  Z^{\al} (\f{\vro}{1+\vro}) \, \p_3 \Delta_h u  \cdot \p_3 Z^{\al} u \ dx \\
\lesssim &\; \sqrt{\me(t)} \md(t),\\
III_5 = &\;  \i Z^{\al} (\f{\vro}{1+\vro} \p_3 \du) \cdot \p_3^2 Z^{\al} u_3 \ dx -  \i \p_3 Z^{\al} (\f{\vro}{1+\vro} \nabla_h \du) \cdot \p_3 Z^{\al} u_h \ dx \\
= &\; \sum_{0 \le \be \le \al} C_{\al}^{\be}  \i  Z^{\be} (\f{\vro}{1+\vro}) \, Z^{\al-\be} \p_3 \du \cdot \p_3^2 Z^{\al} u_3 \ dx \\
&\; -  \sum_{0 \le \be \le \al} C_{\al}^{\be}  \i  \p_3 Z^{\be} (\f{\vro}{1+\vro}) \, Z^{\al-\be} \nabla_h \du \cdot \p_3 Z^{\al} u_h \ dx  \\
&\; -  \sum_{0 \le \be \le \al} C_{\al}^{\be}  \i  Z^{\be} (\f{\vro}{1+\vro}) \, \p_3  Z^{\al-\be} \nabla_h \du \cdot \p_3 Z^{\al} u_h \ dx  \\
= &\; \sum_{0 \le \be \le \al} C_{\al}^{\be}  \i  Z^{\be} (\f{\vro}{1+\vro}) \, Z^{\al-\be} \p_3 \du \cdot \p_3^2 Z^{\al} u_3 \ dx   + \i Z^{\al}  \du \, \nabla_h \cdot \Big(\p_3 (\f{\vro}{1+\vro}) \, \p_3 Z^{\al} u_h \Big) \ dx \\
&\; -  \sum_{0 < \be \le \al} C_{\al}^{\be}  \i  \p_3 Z^{\be} (\f{\vro}{1+\vro}) \, Z^{\al-\be} \nabla_h \du \cdot \p_3 Z^{\al} u_h \ dx  +  \i \p_3 Z^{\al}  \du \, \nabla_h \cdot \Big(\f{\vro}{1+\vro} \, \p_3 Z^{\al} u_h \Big) \ dx\\
&\; -  \sum_{0 < \be \le \al} C_{\al}^{\be}  \i  Z^{\be} (\f{\vro}{1+\vro}) \, \p_3  Z^{\al-\be} \nabla_h \du \cdot \p_3 Z^{\al} u_h \ dx  \\
\lesssim &\; \sqrt{\me(t)} \md(t),
                \end{align*}
                and
                \beqq \bal
                III_6 = \var \sum_{0 \le \be \le \al} C_{\al}^{\be}  \i  Z^{\be} (\f{\vro}{1+\vro}) \,   Z^{\al-\be} \p_3^2 u \cdot \p_3^2 Z^{\al} u \ dx \lesssim  \sqrt{\me(t)} \md(t).
                \dal \deqq
                Using the property of the commutator \eqref{properity-p3}, we can obtain
        \begin{align*}
        III_7 = &\; - \sum_{0< \be_3 \le \al_3} C_{\al_3}^{\be_3}  \i \p_3 Z^{\al} \du   \cdot \p_3( Z_3^{\be_3}(\f{1}{\vf}) \,  Z^{\al_{h}} Z_3^{\al_3-\be_3+1 }  u_3 ) \ dx \\
        \lesssim &\;
        \| \p_3 Z^{\al} \du\|_{L^2} \| \p_3^2 u_3\|_{H_{co}^{k-1}} \lesssim \nu  \| \p_3 Z^{\al} \du\|_{L^2}^2 + C_{\nu}\| \p_3^2 u_3\|_{H_{co}^{k-1}}^2,\\
         III_{8} = &\; - \i ( Z^{\al} \p_3 \du - \p_3 Z^{\al} \du ) \cdot \p_3 Z^{\al} \p_3 u_3\ dx  \\
         &\; - \i ( Z^{\al} \p_3 \du - \p_3 Z^{\al} \du ) \cdot (\p_3^2  Z^{\al} u_3-\p_3 Z^{\al} \p_3 u_3)\ dx \\
         \lesssim  &\; \| \p_3 \du\|_{H_{co}^{k-1}} (\|\p_3 Z^{\al} \p_3 u_3\|_{L^2} + \| \p_3^2 u_3\|_{H_{co}^{k-1}} )\\
         \lesssim &\; \nu  \| \p_3 Z^{\al} \p_3 u_3 \|_{L^2}^2 + C_{\nu} \|( \p_3^2 u_3 ,\p_3 \du)\|_{H_{co}^{k-1}}^2,\\
         III_{9} \lesssim &\; \nu \var \| \p_3^2 Z^{\al} u\|_{L^2}^2 + C_{\nu} \var \| \p_3^2 u\|_{H_{co}^{k-1}}^2,
        \end{align*}
        and similarly, for $i=1,2,3$, we have
        \beqq \bal
        III_{10}  = &\; \i Z^{\al-e_i} \p_3 \vro \cdot Z^{e_i} (\p_3^2 Z^{\al} u_3 -\p_3 Z^{\al} \p_3 u_3) \ dx
        \lesssim \nu \| \p_3^2 u_3\|_{H_{co}^{k}}^2+ C_{\nu} \| \p_3 \vro\|_{H_{co}^{k-1}}^2, \\
        III_{11}  \lesssim &\;  \nu  \| \p_3 Z^{\al} \p_3 u_3\|_{L^2}^2 + C_{\nu} \| \p_3 \vro\|_{H_{co}^{k-1}}^2. \\
        \dal \deqq
        Combining the estimates from $III_{1}$ to  $III_{11}$, summing over $|\al| = k$ and using the smallness of $\nu$, we have
        \beq \label{Z3-all} \bal
&\; \frac{d}{dt}\i\sum_{|\al|=k}(| \p_3 Z^{\al} \vro|^2+|\p_3 Z^{\al}  u|^2)dx
						+\f32\sum_{|\al|=k}(\|\p_3 Z^{\al} \nabla_h  u\|_{L^2}^2 +\| \p_3  Z^{\al} \du\|_{L^2}^2 + \var \|\p_3^2 Z^{\al}u\|_{L^2}^2)  \\
 \lesssim &\;  \|\p_3(\nabla_h u, \du)(t)\|_{H_{co}^{k-1}}^2  +\var \|\p_3^2 u(t)\|_{H_{co}^{k-1}}^2 + \|\p_3 \vro(t)\|_{H_{co}^{k-1}}^2+ (\me(t)^{\f13}+\me(t)^{\f12})\md(t) .
       \dal \deq
       Thanks to the Lemma \ref{lemma36}, which will be proved later, we have
       \begin{align*}
			 \| \nabla \vro\|_{H_{co}^{k-1}}^2
		\lesssim &\; -\sum_{0 \le |\al| \le k-1}\frac{d}{dt} \i Z^{\al} u \cdot Z^{\al} \nabla  \vro \ dx +  \|(\du, \nabla \du)\|_{H^{k-1}_{co}}^2+ \| (\nabla_h u,\Delta_h u)\|_{H_{co}^{k-1}}^2 \\
                    &\; +
					\var \|\p_3^2 u \|_{H^{k-1}_{co}}^2
					+\sqrt{\me(t)}\md(t)\\
                    \lesssim &\; -\sum_{0 \le |\al| \le k-1}\frac{d}{dt} \i Z^{\al} u \cdot Z^{\al} \nabla  \vro \ dx +\|\p_3(\nabla_h u, \du)\|_{H_{co}^{k-1}}^2 +  \|\du\|_{H^{k-1}_{co}}^2\\
                    &\;+ \| (\nabla_h u, \Delta_h u)\|_{H_{co}^{k-1}}^2  +
					\var \|\p_3^2 u \|_{H^{k-1}_{co}}^2
					+\sqrt{\me(t)}\md(t).
       \end{align*}
         Thus, substituting above estimate into \eqref{Z3-all},  we can obtain
       \begin{align*}
				&\; \f{d}{dt}\|\p_3(\vro, u)\|_{H_{co}^{k}}^2 + \sum_{0 \le |\al| \le k-1} \f{d}{dt}\i Z^{\al} u \cdot Z^{\al} \nabla  \vro \ dx +\f32\sum_{|\al|=k} \|\p_3 Z^{\al} \nabla_h  u\|_{L^2}^2   \\
                &\; + \f32 \sum_{|\al|=k}  (\| \p_3  Z^{\al} \du\|_{L^2}^2 +\var   \|\p_3^2 Z^{\al}u\|_{L^2}^2 ) \\
				\lesssim &\;  \|\p_3(\nabla_h u, \du)\|_{H_{co}^{k-1}}^2  +\var \|\p_3^2 u\|_{H_{co}^{k-1}}^2  +\|\du\|_{H^{k-1}_{co}}^2 + \| (\nabla_h u, \Delta_h u)\|_{H_{co}^{k-1}}^2  +  (\me(t)^{\f13}+\me(t)^{\f12})\md(t),
            \end{align*}
             thus we complete the proof of this lemma.
            \end{proof}
			\end{lemm}

			\subsection{Dissipation estimate of density}\label{sec-density}
            In this subsection, we will establish the dissipative
			estimate of density derived from the velocity equation $\eqref{eqr}_2$. First of all, we will establish the tangential dissipative
			estimate.
			\begin{lemm}\label{lemma35}
			For any smooth solution $(\vro,u)$ of equation \eqref{eqr}, it holds for all $0 \le |\ah| = k \le m-1$\
			\beq \label{3401}
				\begin{aligned}
					&\frac{d}{dt}
					\i Z^{\ah} u \cdot Z^{\ah} \nabla \vro \ dx
					+ \f12\|   Z^{\ah}  \nabla \vro\|_{L^2}^2 \\
					\lesssim &
					\left\{\begin{array}{*{2}{ll}}
				\|(\du, \nabla \du, \Delta_h u)\|_{H^{k}_{tan}}^2+
					\var \|\p_3^2 u \|_{H^{k}_{tan}}^2+ \sqrt{\me(t)}\md_{tan}(t) & ~~ k=0 ~{\rm or}~ m-1,\\
				 \|\nabla_h(\du, \nabla \du, \Delta_h u)\|_{H^{k-1}_{tan}}^2+
					\var \|\nabla_h \p_3^2 u \|_{H^{k-1}_{tan}}^2+\sqrt{\me(t)}\bmd_{tan}(t) & ~~ k=1 ~{\rm or}~ m-2.
		\end{array}\right.
				\end{aligned}
				\deq
			\end{lemm}
            \begin{proof}
            For any $|\ah|=k \le m-1$, the equation $\eqref{eqr}_2$ yields that
            \beqq \bal
            \|Z^{\ah} \nabla \vro\|_{L^2}^2  =  &\;  -\i \p_t Z^{\ah} u \cdot Z^{\ah} \nabla \vro \ dx + \i Z^{\ah} (\Delta_h u + \var \p_3^2 u + \nabla \du) \cdot Z^{\ah} \nabla \vro \ dx\\
            &\; - \i Z^{\ah}\Big( u \cdot \nabla u+ \f{\vro}{1+\vro} (\Delta_h u +\var \p_3^2 u+\nabla \du) +\vro\, \nabla \vro \Big) \cdot Z^{\ah} \nabla \vro \ dx.
            \dal \deqq
            On the other hand, using the density equation $\eqref{eqr}_1$ and integrating by part, we conclude
            \beqq \bal
        &\; -\i \p_t Z^{\ah} u \cdot Z^{\ah} \nabla \vro \ dx \\ = &\;  -\p_t \i Z^{\ah} u \cdot Z^{\ah} \nabla \vro \ dx  +\i  Z^{\ah} u \cdot  Z^{\ah} \nabla \p_t \vro \ dx \\
        = &\;  -  \p_t \i Z^{\ah} u \cdot Z^{\ah} \nabla \vro \ dx   -\i  Z^{\ah} u \cdot \nabla Z^{\ah} (\du + u \cdot \nabla \vro + \vro \, \du)\ dx \\
        = &\;  -  \p_t \i Z^{\ah} u \cdot Z^{\ah} \nabla \vro \ dx   +\i  Z^{\ah} \du \cdot Z^{\ah} (\du + u \cdot \nabla \vro + \vro \, \du)\ dx.
            \dal \deqq
            Thus, we have
            \begin{align*}
           &\; \f{d}{dt} \i Z^{\ah} u \cdot Z^{\ah} \nabla \vro \ dx  + \|Z^{\ah} \nabla \vro\|_{L^2}^2 \\
           = &\; \| Z^{\ah} \du\|_{L^2}^2 +  \i Z^{\ah} \du \cdot Z^{\ah} ( u \cdot \nabla \vro + \vro \, \du)\ dx  + \i Z^{\ah} (\Delta_h u + \var \p_3^2 u + \nabla \du) \cdot Z^{\ah} \nabla \vro \ dx\\
           &\; - \i Z^{\ah} (\vro\, \nabla \vro)\cdot Z^{\ah} \nabla \vro \ dx- \i Z^{\ah}( u \cdot \nabla u) \cdot Z^{\ah} \nabla \vro \ dx \\
           &\; - \i Z^{\ah} \Big(\f{\vro}{1+\vro} (\Delta_h u+\var\p_3^2 u+\nabla \du  ) \Big) \cdot Z^{\ah} \nabla \vro \ dx \\
          := &\;  \| Z^{\ah} \du\|_{L^2}^2 + \sum_{i=1}^{5} IV_{i}.  
            \end{align*}
            We can directly obtain that
            \beqq
            IV_{2} \lesssim \nu \|  Z^{\ah} \nabla \vro\|_{L^2}^2 +  C_{\nu} \| Z^{\ah} (\Delta_h u, \nabla \du)\|_{L^2}^2 +  C_{\nu} \var \| Z^{\ah} \p_3^2 u\|_{L^2}^2.
            \deqq
             Next we estimate the other terms by the sequence.
           {\bf We first deal with the case $|\ah|=0$.}
            \begin{align*}
          IV_1 \lesssim &\; \| \vro\|_{L^{\infty}} \| \du\|_{L^2}^2 + \| u \|_{L^{\infty}} \| \du\|_{L^2} \| \nabla \vro\|_{L^2},\\
          IV_3 + IV_4  \lesssim &\; \| \vro\|_{L^{\infty}} \| \nabla \vro\|_{L^2}^2 + \| \nabla \vro\|_{L^2}(\| u_h\|_{L^{\infty}} \| \nabla_h u\|_{L^2} + \| u_3\|_{H_{tan}^1}^{\f12}  \| \p_3 u_3\|_{H_{tan}^1}^{\f12}  \| \p_3 u\|_{L^2}^{\f12} \| \p_{13} u\|_{L^2}^{\f12}),\\
          IV_5  \lesssim &\; \| \vro\|_{L^{\infty}} \| \nabla \vro\|_{L^2} ( \| \Delta_h u\|_{L^2} + \var \| \p_3^2  u\|_{L^2} + \| \nabla \du\|_{L^2} ),
            \end{align*}
            thus we have
            \beq \label{3402}
            \f{d}{dt} \i u \cdot \nabla  \vro \ dx  + \|\nabla \vro\|_{L^2}^2 \lesssim \|(\du, \nabla \du, \Delta_h u)\|_{L^2}^2+
					\var \|\p_3^2 u \|_{L^2}^2
					+
				\sqrt{\me(t)}\md_{tan}(t).
            \deq
            {\bf Now we deal with the case $|\ah|=1$.}
            \begin{align*}
            IV_1 \lesssim &\; \| u\|_{L^{\infty}} \| \nabla_h  \nabla \vro\|_{L^2} \|\nabla_h \du\|_{L^2} +\|\nabla_h \du\|_{L^2}^{\f12 } \|\p_1 \nabla_h \du\|_{L^2}^{\f12 }\Big(\| \nabla_h u_h \|_{L^2}^{\f12} \| \p_2 \nabla_h u_h \|_{L^2}^{\f12} \|\nabla_h \vro\|_{L^2}^{\f12} \|\p_3 \nabla_h \vro\|_{L^2}^{\f12} \\
            &\; + \| \nabla_h u_3 \|_{L^2}^{\f12} \| \p_3 \nabla_h u_3 \|_{L^2}^{\f12} \|\p_3 \vro\|_{L^2}^{\f12} \|\p_{23} \vro\|_{L^2}^{\f12} + \|\nabla_h \vro\|_{L^2}^{\f12} \|\p_3 \nabla_h \vro\|_{L^2}^{\f12} \| \du \|_{L^2}^{\f12} \| \p_2 \du\|_{L^2}^{\f12} \Big)\\
            &\;+ \| \vro\|_{L^{\infty}} \| \nabla_h \du\|_{L^2}^2,\\
            \lesssim &\; \sqrt{\me(t)}\bmd_{tan}(t),\\
            IV_{3} \lesssim &\; \| \nabla_h  \nabla \vro\|_{L^2} \| \nabla_h \vro\|_{H_{tan}^1}^{\f12} \| \p_3 \nabla_h \vro\|_{H_{tan}^1}^{\f12} \| \nabla \vro \|_{L^2}^{\f12} \| \p_2 \nabla \vro \|_{L^2}^{\f12} + \| \vro\|_{L^{\infty}}  \| \nabla_h  \nabla \vro\|_{L^2}^2\\
            \lesssim &\; \sqrt{\me(t)}\bmd_{tan}(t),
            \end{align*}
            and using the interpolation inequality \eqref{inter-1}, we have
            \begin{align*}
            IV_{4} \lesssim &\; \Big(\| u_h \|_{H_{tan}^2} \| \nabla_h^2 u\|_{L^2}^{\f12} \| \p_3 \nabla_h^2 u\|_{L^2}^{\f12} +  \|\nabla_h u_h \|_{L^2}^{\f14} \|\p_3 \nabla_h u_h  \|_{L^2}^{\f14} \|\p_1 \nabla_h u_h  \|_{L^2}^{\f14} \|\p_{13} \nabla_h u_h \|_{L^2}^{\f14} \| \nabla_h u\|_{L^2}^{\f12} \| \p_{2} \nabla_h u\|_{L^2}^{\f12} \\
            &\; + \|u_3 \|_{L^2}^{\f14} \|\p_3 u_3 \|_{L^2}^{\f14} \|\p_1 u_3 \|_{L^2}^{\f14} \|\p_{13} u_3 \|_{L^2}^{\f14} \| \p_3 \nabla_h u\|_{L^2}^{\f12} \| \p_{23} \nabla_h u\|_{L^2}^{\f12} \\
            &\; +  \|\nabla_h u_3 \|_{L^2}^{\f14} \|\p_3 \nabla_h u_3  \|_{L^2}^{\f14} \|\p_1 \nabla_h u_3  \|_{L^2}^{\f14} \|\p_{13} \nabla_h u_3 \|_{L^2}^{\f14} \| \p_3 u\|_{L^2}^{\f12} \| \p_{23} u\|_{L^2}^{\f12}
            \Big)
            \| \nabla_h  \nabla \vro\|_{L^2} \\
            \lesssim &\; \sqrt{\me(t)}\bmd_{tan}(t)  + \Big( \|u_3 \|_{L^2}^{\f14} \|\p_3 u_3 \|_{L^2}^{\f14} \|\p_1 u_3 \|_{L^2}^{\f14} \|\p_{13} u_3 \|_{L^2}^{\f14} (\| \p_3  u\|_{L^2}^{\f12} \| \p_3 \nabla_h^2 u\|_{L^2}^{\f12} )^{\f12} \| \p_{23} \nabla_h u\|_{L^2}^{\f12} \\
            &\; + \|\nabla_h u_3 \|_{L^2}^{\f14} \|\p_3 \nabla_h u_3  \|_{L^2}^{\f14} \|\p_1 \nabla_h u_3  \|_{L^2}^{\f14} \|\p_{13} \nabla_h u_3 \|_{L^2}^{\f14} \| \p_3 u\|_{L^2}^{\f12} (\| \p_3  u\|_{L^2}^{\f12} \| \p_3 \p_2^2 u\|_{L^2}^{\f12} )^{\f12} \Big)
            \| \nabla_h  \nabla \vro\|_{L^2} \\
            \lesssim &\; \sqrt{\me(t)}\bmd_{tan}(t).
            \end{align*}
            Similarly, using the estimate \eqref{est-p32-u3-uh}, we can check that
            \beqq \bal
             IV_5 \lesssim &\; \sqrt{\me(t)}\bmd_{tan}(t).
            \dal \deqq
            Thus, we have
            \beq \label{3403}
            \f{d}{dt} \i \nabla_h u \cdot \nabla_h \nabla  \vro \ dx  + \|\nabla_h \nabla \vro\|_{L^2}^2 \lesssim \|\nabla_h (\du, \nabla \du, \Delta_h u)\|_{L^2}^2+
					\var \|\nabla_h \p_3^2 u \|_{L^2}^2
					+
				\sqrt{\me(t)}\bmd_{tan}(t).
            \deq
            {\bf  Next, we deal with the case $|\ah|= m-2$.} Using the anisotropic inequalities \eqref{ie:Sobolev}, we can check that
            \begin{align*} 
           IV_1 = &\; \sum_{0 \le \beta_h \le \alpha_h}C^{\beta_h}_{\alpha_h} \i \Big(Z^{\be_h}  u_h \cdot Z^{\ah- \be_h} \nabla_h \vro +  Z^{\be_h}  u_3 \, Z^{\ah- \be_h} \p_3 \vro + Z^{\be_h}  \vro \cdot Z^{\ah- \be_h} \du\Big)\, Z^{\ah} \du \, dx \\
           \lesssim  &\; \|u\|_{L^{\infty}} \|Z^{\ah} \du\|_{L^2} (\|Z^{\ah} \nabla \vro\|_{L^2}  + \|Z^{\ah} \du\|_{L^2}) +
           \|Z^{\ah} \du\|_{L^2}\\
           &\; \times \sum_{0 < \beta_h \le \alpha_h} \Big(
           \| Z^{\be_h}  u_h\|_{L^2}^{\f14}\|\p_3 Z^{\be_h}  u_h\|_{L^2}^{\f14} \|\p_1 Z^{\be_h}  u_h\|_{L^2}^{\f14}\|\p_{13} Z^{\be_h}  u_h\|_{L^2}^{\f14} \|Z^{\ah- \be_h} \nabla_h \vro \|_{L^2}^{\f12} \|\p_2 Z^{\ah- \be_h} \nabla_h \vro \|_{L^2}^{\f12}\\
           &\;  +\| Z^{\be_h}  u_3\|_{L^2}^{\f14}\|\p_3 Z^{\be_h}  u_3\|_{L^2}^{\f14} \|\p_1 Z^{\be_h}  u_3\|_{L^2}^{\f14}\|\p_{13} Z^{\be_h}  u_3\|_{L^2}^{\f14} \|Z^{\ah- \be_h} \p_3 \vro \|_{L^2}^{\f12} \|\p_{23} Z^{\ah- \be_h} \vro \|_{L^2}^{\f12}\\
           &\; + \| Z^{\be_h}  \vro\|_{L^2}^{\f12}\|\p_3 Z^{\be_h}  \vro\|_{L^2}^{\f12}   \|Z^{\ah- \be_h} \du \|_{L^2}^{\f14} \|\p_1 Z^{\ah- \be_h} \du \|_{L^2}^{\f14}  \|\p_2 Z^{\ah- \be_h} \du \|_{L^2}^{\f14}  \|\p_{12} Z^{\ah- \be_h} \du\|_{L^2}^{\f14}
           \Big)\\
           \lesssim &\; \sqrt{\me(t)}\bmd_{tan}(t).
           \end{align*}
Similarly, we can check that
            \beqq \bal
             IV_3 + IV_4+ IV_5  \lesssim &\; \sqrt{\me(t)}\bmd_{tan}(t).
            \dal \deqq
            Thus, we have
            \beq \label{3404}
            \f{d}{dt} \i Z^{\ah} u \cdot Z^{\ah} \nabla  \vro \ dx  + \|Z^{\ah} \nabla \vro\|_{L^2}^2 \lesssim \|Z^{\ah} (\du, \nabla \du, \Delta_h u)\|_{L^2}^2+
					\var \|Z^{\ah} \p_3^2 u \|_{L^2}^2
					+
				\sqrt{\me(t)}\bmd_{tan}(t).
            \deq
          {\bf{Finally, we deal with the case  $|\ah| = m-1$.}} Following essentially the same strategy as before estimates, we can check that if $|\ah| = m-1$,
            \beq \label{3405}
            \f{d}{dt} \i Z^{\ah} u \cdot Z^{\ah} \nabla  \vro \ dx  + \|Z^{\ah} \nabla \vro\|_{L^2}^2 \lesssim \|Z^{\ah} (\du, \nabla \du, \Delta_h u)\|_{L^2}^2+
					\var \|Z^{\ah} \p_3^2 u \|_{L^2}^2
					+
				\sqrt{\me(t)}\md_{tan}(t),
            \deq
            and thus omit the proof.
            Combining the estimates from \eqref{3402} to \eqref{3405}, we finish the proof of this lemma.
            \end{proof}			
        	
			Finally, we will establish the dissipative estimate under the higher conormal Sobolev space  as follows.
            \begin{lemm} \label{lemma36}
            For any smooth solution $(\vro,u)$ of equation \eqref{eqr}, it holds for $ |\al| = k\le m-1$,
			\beq  \label{3501}
				\begin{aligned}
				&\; \frac{d}{dt}
					 \i Z^{\al} u \cdot Z^{\al} \nabla  \vro \ dx
					+ \f12 \|Z^{\al} \nabla \vro\|_{L^2}^2
					\\
                    \lesssim &\;\|(\du, \nabla \du)\|_{H^{k}_{co}}^2+ \| (\nabla_h u, \Delta_h u)\|_{H_{co}^{k}}^2 +
					\var \|\p_3^2 u \|_{H^{k}_{co}}^2
					+\sqrt{\me(t)}\md(t).
					\end{aligned}
				\deq
			\end{lemm}
            \begin{proof}
                 For any $ |\al|=k \le m-1$, the equation $\eqref{eqr}_2$ yields that
            \beqq \bal
            \|Z^{\al} \nabla \vro\|_{L^2}^2  =  &\;  -\i \p_t Z^{\al} u \cdot Z^{\al} \nabla \vro \ dx + \i Z^{\al} (\Delta_h u + \var \p_3^2 u + \nabla \du) \cdot Z^{\al} \nabla \vro \ dx\\
            &\; + \i Z^{\al}\Big( u \cdot \nabla u+ \f{\vro}{1+\vro} (\Delta_h u +\var \p_3^2 u+\nabla \du) +\vro\, \nabla \vro \Big) \cdot Z^{\al} \nabla \vro \ dx.
            \dal \deqq
            On the other hand, using the density equation $\eqref{eqr}_1$ and integrating by part, we conclude
            \beqq \bal
        &\; -\i \p_t Z^{\al} u \cdot Z^{\al} \nabla \vro \ dx \\
        = &\; - \p_t \i Z^{\al} u \cdot Z^{\al} \nabla \vro \ dx  +\i  Z^{\al} u \cdot  Z^{\al} \nabla \p_t \vro \ dx \\
        = &\;  -  \p_t \i Z^{\al} u \cdot Z^{\al} \nabla \vro \ dx   +\i  Z^{\al} u \cdot (Z^{\al} \nabla  - \nabla Z^{\al}) \p_t \vro \ dx  + \i  Z^{\al} u \cdot  \nabla Z^{\al} \p_t \vro\ dx \\
        = &\;  -  \p_t \i Z^{\al} u \cdot Z^{\al} \nabla \vro \ dx  -\i  Z^{\al} u_3 \cdot (Z^{\al} \p_3-\p_3 Z^{\al})( \du +u \cdot \nabla \vro + \vro \, \du) \ dx \\
        &\; +\i  {\rm div} Z^{\al} u \cdot Z^{\al} (\du + u \cdot \nabla \vro + \vro \, \du)\ dx.
            \dal \deqq
            Thus, we have
              \begin{align*}
           &\; \f{d}{dt} \i Z^{\al} u \cdot  Z^{\al} \nabla \vro \ dx  + \|Z^{\al} \nabla \vro\|_{L^2}^2 \\
           = &\; -\i  Z^{\al} u_3 \cdot (Z^{\al} \p_3-\p_3 Z^{\al})( \du +u \cdot \nabla \vro + \vro \, \du) \ dx   + \i  {\rm div} Z^{\al} u \cdot Z^{\al} \du\ dx \\
           &\; +  \i {\rm div} Z^{\al} u  \cdot Z^{\al} ( u \cdot \nabla \vro + \vro \, \du)\ dx  + \i Z^{\al} (\Delta_h u + \var \p_3^2 u + \nabla \du) \cdot Z^{\al} \nabla \vro \ dx\\
           &\; - \i Z^{\al} (\vro\, \nabla \vro)\cdot Z^{\al} \nabla \vro \ dx- \i Z^{\al}( u \cdot \nabla u) \cdot Z^{\al} \nabla \vro \ dx \\
           &\; - \i Z^{\al} \Big(\f{\vro}{1+\vro} (\Delta_h u +\var \p_3^2 u+ \nabla \du)  \Big) \cdot Z^{\al} \nabla \vro \ dx \\
          := &\; \sum_{i=1}^{7} V_i.
            \end{align*}
            If $|\al|=0$, we can obtain the estimate \eqref{3501} directly from the estimate \eqref{3401}.
            Next, we deal with the case $1\le |\al| \le m-1$. If $\al_3 \ge 1$, using the property of the commutator \eqref{properity-p3}, integrating by parts and applying the anisotropic type inequality $\eqref{ie:Sobolev}_2$, we have
            \begin{align*}
            V_1
            = &\;\sum_{0< \be_3 \le \al_3} C_{\al_3}^{\be_3}  \i Z^{\al} u_3  \,   Z_3^{\be_3}(\f{1}{\vf})  \cdot  Z^{\al_{h}} Z_3^{\al_3-\be_3+1 }  ( \du +u \cdot \nabla \vro + \vro \, \du) \ dx \\
            = &\;\sum_{0< \be_3 \le \al_3} C_{\al_3}^{\be_3}  \i Z^{\al_{h}} Z_3^{\al_3-\be_3}  ( \du +u \cdot \nabla \vro + \vro \, \du)  \, \p_3 \Big( Z^{\al} u_3  \,  \vf Z_3^{\be_3}(\f{1}{\vf})\Big)\ dx \\
           \lesssim &\;\| \p_3 u_3\|_{H_{co}^{k}} \| \du\|_{H_{co}^{k-1}} +  \sqrt{\me(t)} \md(t)\\
            \lesssim &\;  \| \du \|_{H_{co}^{k}}^2+  \| \nabla_h u\|_{H_{co}^{k}}^2 + \sqrt{\me(t)} \md(t).
            \end{align*}
            If $\al_3=0$, we have $V_1=0$. 
            By Lemma \ref{lemm:sobolev-ie} and H\"older inequality, it is easy to check that
            \begin{align*}
          V_2 \lesssim &\; \| \du \|_{H_{co}^k}^2,\\
          V_3 = &\; \sum_{0\le \be \le \al} C_{\al}^{\be}  \i (Z^{\be} u \cdot  Z^{\al-\be}  \nabla \vro + Z^{\be} \vro \cdot  Z^{\al-\be}  \du )  \cdot  {\rm  div } Z^{\al} u \ dx \\
          \lesssim &\;  \|{\rm  div } Z^{\al} u \|_{L^2}^{\f12}  \|\p_1{\rm  div } Z^{\al} u \|_{L^2}^{\f12}  \sum_{0\le \be \le \al} \Big( \| Z^{\be} u\|_{H_{tan}^2} \|Z^{\al-\be}  \nabla \vro\|_{L^2}^{\f12}\|\p_3 Z^{\al-\be}  \nabla \vro\|_{L^2}^{\f12} \\
          &\; + \| Z^{\be} \vro\|_{L^2}^{\f12} \| \p_3 Z^{\be} \vro\|_{L^2}^{\f12}   \|Z^{\al-\be}  \du\|_{L^2}^{\f12}\|\p_2 Z^{\al-\be}  \du\|_{L^2}^{\f12} \Big)\\
          \lesssim &\; \sqrt{\me(t)} \md(t),\\
          V_4 \le &\; \nu \|Z^{\al}  \nabla \vro\|_{L^2}^2 +C_{\nu}\|Z^{\al} (\Delta_h u + \var \p_3^2 u + \nabla \du)\|_{L^2}^2,\\
            \end{align*}
and
\begin{align*}
V_5 = &\; \sum_{0\le \be \le \al} C_{\al}^{\be}  \i Z^{\be} \vro \cdot  Z^{\al-\be}  \nabla \vro \cdot   Z^{\al} \nabla \vro \ dx \\
\lesssim &\; \|\vro\|_{L^{\infty}} \| Z^{\al}  \nabla \vro \|_{L^2}^2 + \Big(\sum_{0 <  \be < \al} \| Z^{\be} \vro\|_{L^2}^{\f14} \| \p_3 Z^{\be} \vro\|_{L^2}^{\f14} \| \p_1 Z^{\be} \vro\|_{L^2}^{\f14}  \| \p_{13} Z^{\be} \vro\|_{L^2}^{\f14}   \|Z^{\al-\be}  \nabla \vro\|_{L^2}^{\f12}\|\p_2 Z^{\al-\be}  \nabla \vro\|_{L^2}^{\f12} \\
          &\; + \| Z^{\al} \vro\|_{L^2}^{\f12} \| \p_3 Z^{\al} \vro\|_{L^2}^{\f12}   \| \nabla \vro\|_{H_{tan}^2} \Big)\| Z^{\al}  \nabla \vro \|_{L^2} \\
        \lesssim &\; \sqrt{\me(t)} \md(t),\\
        V_6 = &\; \sum_{0\le \be \le \al} C_{\al}^{\be}  \i Z^{\be} u \cdot  Z^{\al-\be}  \nabla u \cdot   Z^{\al} \nabla \vro \ dx \\
        \lesssim &\; \sum_{0 <  \be < \al} \| Z^{\be} u\|_{L^2}^{\f14} \| \p_3 Z^{\be} u\|_{L^2}^{\f14} \| \p_1 Z^{\be} u\|_{L^2}^{\f14}  \| \p_{13} Z^{\be} u\|_{L^2}^{\f14}   \|Z^{\al-\be}  \nabla u\|_{L^2}^{\f12}\|\p_2 Z^{\al-\be}  \nabla u\|_{L^2}^{\f12}\| Z^{\al}  \nabla \vro \|_{L^2} \\
        \lesssim &\; \sqrt{\me(t)} \md(t).
\end{align*}
Similarly, we have
\beqq \bal
V_7 =  \sum_{0\le \be \le \al} C_{\al}^{\be}  \i Z^{\be} (\f{\vro}{1+\vro}) \cdot  Z^{\al-\be}  (\Delta_h u +\var \p_3^2 u+ \nabla \du)\cdot   Z^{\al} \nabla \vro \ dx
\lesssim  \sqrt{\me(t)} \md(t),
\dal \deqq
which, together with the estimates from $V_{1}$ to $V_{6}$, yields that
\beqq \bal
				&\; \frac{d}{dt}
					\i Z^{\al} u \cdot Z^{\al} \nabla  \vro \ dx
					+  \|Z^{\al} \nabla \vro\|_{L^2}^2 \\
					\le &\; \nu \|Z^{\al} \nabla \vro\|_{L^2}^2 +  \|(\du, \nabla \du)\|_{H^{k}_{co}}^2+ \| (\nabla_h u, \Delta_h u)\|_{H_{co}^{k}}^2 +
					\var \|\p_3^2 u \|_{H^{k}_{co}}^2
					+\sqrt{\me(t)}\md(t).
\dal \deqq
Using the smallness of $\nu$, we finish the proof of this lemma.
            \end{proof}
Finally, we will establish the estimate for $\var \| \p_3^2 \vro\|_{H_{co}^2}^2$.
            \begin{lemm}\label{lemma37}
				For any smooth solution $(\vro,u)$ of equation \eqref{eqr},
				it holds that
                \beq \label{3901}\bal
              &\;  \var(1+\var) \frac{d}{dt}  \Big(\i \f{1}{1+\vro} | \p_3^2 \vro|^2 \ dx + \sum_{  |\al|=2 } \i \f{1}{1+\vro} |Z^{\al} \p_3^2 \vro|^2 \ dx \Big)
    				+  \var\|\p_3^2 \vro \|_{H_{co}^{2}}^2 \\
				\lesssim &\; \|\p_t \du\|_{H_{tan}^2}^2 + \| Z_3 \p_t \p_3^2 u_3\|_{L^2}^2  + \| \p_t \p_3^2 u_3\|_{H_{tan}^1}^2
                +\|(\Delta_h u,\nabla_h \du, \nabla \vro)\|_{H_{co}^3}^2   + \sqrt{\me(t)} \md(t).
                \dal \deq
            \end{lemm}
	\begin{proof}
   From the equation \eqref{eqr}, we have
   \beq \label{3902}\bal
   &\; \p_t( 1+ \vro) + ( 1+ \vro) \, \du + u\cdot \nabla (1+\vro) = 0,\\
   &\; ( 1+ \vro) \p_t u + ( 1+ \vro) u \cdot \nabla u - \Delta_h u -\var \p_3^2 u - \nabla \du + ( 1+ \vro)^2 \nabla \vro = 0.
   \dal \deq
   Applying ${\rm div}$-operator to the equality $\eqref{3902}_2$, we have
   \beq \label{3903}
  {\rm div} (( 1+ \vro) \p_t u + ( 1+ \vro) u\cdot \nabla u ) - 2 \Delta_h \du - (1+\var) \p_3^2 \du + 2 (1+\vro) |\nabla \vro|^2 +  (1+\vro)^2 \Delta \vro = 0,
   \deq
   where we have used that $\Delta \du = \p_3^2 \du +\Delta_h \du $. Using the equation $\eqref{3902}_1$, we have
   \beqq
  - \p_3^2 \du =\p_3^2 ( \p_t \ln (1+\vro) + u \cdot \nabla \ln (1+\vro) ).
   \deqq
   Substituting above equality to the equation \eqref{3903}, we have
   \beqq \bal
   &\; (1+\var) \p_t  \p_3^2  \ln (1+\vro) + (1+\vro)^2 \p_3^2 \vro \\ = &\;
   -(1+\var)  \p_3^2 ( u \cdot \nabla \ln (1+\vro) ) - {\rm div} (( 1+ \vro) \p_t u + ( 1+ \vro) u \cdot \nabla u ) \\
   &+ 2 \Delta_h \du - 2 (1+\vro) |\nabla \vro|^2  -(1+\vro)^2 \Delta_h \vro \\
   := &\; F(x,t).
  \dal  \deqq
  Using the fact that
   \beqq
            \p_t \p_3^2 \ln(1+\vro) = \f{2}{(1+\vro)^3} \p_t \vro |\p_3 \vro|^2 - \f{2}{(1+\vro)^2} \p_3 \vro \,  \p_t \p_3 \vro -\f{1}{(1+\vro)^2} \p_t \vro \, \p_3^2 \vro + \f{1}{1+\vro} \p_t \p_3^2 \vro,
            \deqq
            and for $0 \le |\al| \le 2$, applying $Z^{\al}$-operator to the above equality, we have
  \beqq \bal
  &\; \f{1+\var}{1+\vro} \p_t Z^{\al} \p_3^2 \vro + (1+ \vro)^2 Z^{\al} \p_3^2 \vro\\
  = &\;   - (1+\var) Z^{\al} \Big(\f{2}{(1+\vro)^3} \p_t \vro |\p_3 \vro|^2 - \f{2}{(1+\vro)^2} \p_3 \vro \,  \p_t \p_3 \vro -\f{1}{(1+\vro)^2} \p_t \vro \, \p_3^2 \vro \Big)\\
  &\; - (1+\var) \sum_{ 0 < \be \le \al } C_{\al}^{\be} Z^{\be} (\f{1}{1+\vro}) Z^{\al- \be} \p_t \p_3^2 \vro -\sum_{ 0 < \be \le \al } C_{\al}^{\be} Z^{\be} ((1+ \vro)^2 ) Z^{\al- \be} \p_3^2 \vro + Z^{\al} F.
  \dal \deqq
  Multiplying above equality  by $\var Z^{\al} \p_3^2  \vro$ and integrating over $\mathbb{R}_+^3$, we can obtain that
             \beqq \bal
            &\; \frac{d}{dt} \frac{\var(1+\var)}{2} \i \f{1}{1+\vro} |Z^{\al} \p_3^2 \vro|^2 \ dx
						+\var \| (1+ \vro) Z^{\al} \p_3^2 \vro\|_{L^2}^2 \\
         = &\; - 2\var(1+\var) \i Z^{\al} \Big( \f{1}{(1+\vro)^3} \p_t \vro |\p_3 \vro|^2 \Big) \,  Z^{\al} \p_3^2  \vro \ dx   + 2\var(1+\var) \i Z^{\al} \Big( \f{1}{(1+\vro)^2} \p_3 \vro \,  \p_t \p_3 \vro \Big) \,  Z^{\al} \p_3^2  \vro \ dx \\
         &\;  + \var(1+\var) \i Z^{\al} \Big( \f{1}{(1+\vro)^2} \p_t \vro \, \p_3^2 \vro \Big) \,  Z^{\al} \p_3^2  \vro \ dx  +\f{\var(1+\var)}{2} \i  \p_t (\f{1}{1+\vro}) \,   |Z^{\al} \p_3^2  \vro|^2 \ dx\\
         &\; - \var(1+\var) \sum_{ 0 < \be \le \al } C_{\al}^{\be} \i Z^{\be} (\f{1}{1+\vro}) Z^{\al- \be} \p_t \p_3^2 \vro\,  Z^{\al} \p_3^2  \vro \ dx  \\
         &\;-\var \sum_{ 0 < \be \le \al } C_{\al}^{\be} \i  Z^{\be} ((1+ \vro)^2 ) Z^{\al- \be} \p_3^2 \vro \,Z^{\al} \p_3^2  \vro \ dx  + \var \i Z^{\al} F \,  Z^{\al} \p_3^2  \vro \ dx\\
         : = &\; \sum_{i=1}^{7} VI_{i},
            \dal \deqq
            where we have used the basis fact that
            \beqq \bal
            \i \f{1}{1+\vro}  \p_t Z^{\al} \p_3^2 \vro \,  Z^{\al} \p_3^2  \vro \ dx 
          = \f{d}{dt} \f12 \i \f{1}{1+\vro} |Z^{\al} \p_3^2 \vro|^2 \ dx - \f12 \i  \p_t (\f{1}{1+\vro}) \,   |Z^{\al} \p_3^2  \vro|^2 \ dx .
            \dal \deqq
            Using the equation $\eqref{eqr}_1$, it is easy to check that
            \beqq \bal
            VI_4 \lesssim &\; \var \| \p_t  \vro\|_{L^{\infty}} \| Z^{\al} \p_3^2  \vro\|_{L^2}^2 \lesssim   \var (\| (1+\vro) \du \|_{L^{\infty}} + \| u \cdot \nabla \vro \|_{L^{\infty}} ) \| Z^{\al} \p_3^2  \vro\|_{L^2}^2 \\
            \lesssim &\; \sqrt{\me(t)} \md(t).
            \dal \deqq
            {\bf We first deal with the case $| \al |=0$.} In this case, $VI_5 = 0$ and $VI_6=0$. Similar to the estimate of $VI_4$, we have
            \beqq \bal
            VI_1 + VI_3 \lesssim  &\; \var \| \p_t  \vro\|_{L^{\infty}} \| \p_3^2  \vro\|_{L^2}(\|  \p_3^2  \vro\|_{L^2} +\| \p_3  \vro\|_{L^2} \|\p_3 \vro\|_{L^{\infty}} )
            \lesssim \sqrt{\me(t)} \md(t),\\
            VI_2 \lesssim &\; \var  \| \p_3^2  \vro\|_{L^2} \| \p_3 \vro \|_{L^2}^{\f12} \| \p_3^2 \vro\|_{L^2}^{\f12} \| \p_t \p_3 \vro\|_{L^2}^{\f14}\| \nabla_h \p_t \p_3 \vro\|_{L^2}^{\f12} \| \nabla_h^2 \p_t \p_3 \vro\|_{L^2}^{\f14} \\
            \lesssim &\; \sqrt{\me(t)} \md(t),
            \dal \deqq
            where we have used the estimate \eqref{pro-p3tvro} of the term $\| \p_t \p_3 \vro\|_{H_{co}^2}$.
            Similarly, we can deal with nonlinear term $VI_7$ as follows:
            \beqq
            VI_7 \le \nu \var \| \p_3^2 \vro\|_{L^2}^2 + C_{\nu} (\| \p_t \du\|_{L^2}^2 + \| \Delta_h \du\|_{L^2}^2  + \| \Delta_h \vro\|_{L^2}^2  )   +\sqrt{\me(t)} \md(t),
            \deqq
            which, combining the estimates from $VI_1$ to $VI_6$ and using the smallness of $\nu$, yields that
            \beq \label{3904}
  \var(1+\var) \frac{d}{dt}  \i \f{1}{1+\vro} | \p_3^2 \vro|^2 \ dx
    				+\var \| (1+ \vro) \p_3^2 \vro\|_{L^2}^2
				\lesssim  \| (\p_t \du, \Delta_h \du, \Delta_h \vro)\|_{L^2}^2   + \sqrt{\me(t)} \md(t).
            \deq
            {\bf Finally, we deal with the case $| \al |=2$.}  Using the estimate \eqref{pro-ptvro} of $ \p_t \vro$, applying the anisotropic inequality \eqref{ie:Sobolev}, for $i=1,2,3$, we have
           \beqq \bal
            VI_1 = &\; - 2 \var(1+\var) \sum_{0 \le \be \le \al} C_{\al}^{\be} \i  Z^{\be}\Big(\f{|\p_3 \vro|^2 }{(1+\vro)^3}\Big) Z^{\al-
            \be} \p_t \vro \, Z^{\al} \p_3^2 \vro \ dx\\
            \lesssim &\; \var \| Z^{\al} \p_3^2 \vro\|_{L^2} \Big(\| \p_3 \vro\|_{L^{\infty}}^2 \| Z^{\al} \p_t \vro\|_{L^2}  + \| Z^{e_i} \Big(\f{|\p_3 \vro|^2 }{(1+\vro)^3} \Big)\|_{H_{tan}^1}^{\f12} \| \p_3 Z^{e_i}\Big(\f{|\p_3 \vro|^2}{(1+\vro)^3}\Big)\|_{H_{tan}^1}^{\f12} \| Z^{\al-e_i} \p_t \vro\|_{H_{tan}^1}    \\
            &\; +\| Z^{\al} \Big(\f{|\p_3 \vro|^2 }{(1+\vro)^3} \Big)\|_{L^2}^{\f12} \| \p_3 Z^{\al}\Big(\f{|\p_3 \vro|^2}{(1+\vro)^3} \Big)\|_{L^2}^{\f12}  \| \p_t \vro\|_{H_{tan}^2} \Big) \\
            \lesssim &\; \sqrt{\me(t)} \md(t),
            \dal \deqq
            and similarly, we can check
            \beqq \bal
         VI_2 +VI_3 + VI_6 \lesssim \sqrt{\me(t)} \md(t),
            \dal \deqq
            where we have used the estimate \eqref{pro-p3tvro} of $\p_3 \p_t \vro$.
            Applying $\p_3^2$-operator to the equation $\eqref{eqr}_1$, we have
            \beq \label{equ-pt33vro} \bal
           - \p_t \p_3^2 \vro = &\; \p_3^2 ((1+\vro) \du + u \cdot \nabla \vro)  = (1+\vro) \p_3^2 \du + 2 \p_3 \vro \, \p_3 \du + \p_3^2 \vro \, \du +\p_3^2 ( u \cdot \nabla \vro).
            \dal \deq
            Moreover, from the equation \eqref{equ-p3du} of $\p_3 \du$, we have
            \beqq
            \p_3^2 \du =\f{1}{1+\var} \p_3\Big( (1+ \vro) \p_t u_3 + (1+ \vro) u \cdot \nabla u_3  - \Delta_h u_3 + \var \p_3 \nabla_h \cdot u_h + (1+\vro)^2 \p_3 \vro \Big).
            \deqq
            Substituting above equality into \eqref{equ-pt33vro}, we have
            \beqq \bal
            - \p_t \p_3^2 \vro = &\;  \f{1+\vro}{1+\var} \p_3 \vro \, \p_t u_3 + \f{(1+\vro)^2}{1+\var} \p_t \p_3 u_3+ \f{1+\vro}{1+\var}   \p_3 \Big(  (1+ \vro) u \cdot \nabla u_3  - \Delta_h u_3 + \var \p_3 \nabla_h \cdot u_h + (1+\vro)^2 \p_3 \vro \Big) \\
            &\; + 2 \p_3 \vro \, \p_3 \du + \p_3^2 \vro \, \du +\p_3^2 ( u \cdot \nabla \vro).
            \dal \deqq
            Using the anisotropic type inequality \eqref{ie:Sobolev}, we can check
            \beqq \bal
            \var^{\f12} \| \p_t \p_3^2 \vro\|_{H_{co}^1} \lesssim \sqrt{\md(t)},
            \dal \deqq
            here we omit the proof for simplicity.
            Thus, we can deal with the term $VI_{5}$ as follows:
            \beqq
            VI_5 \lesssim  \var \sum_{ 0 < \be \le \al } \|Z^{\be} (\f{1}{1+\vro})\|_{L^{\infty}} \| Z^{\al-\be} \p_t \p_3^2 \vro\|_{L^2} \| Z^{\al} \p_3^2  \vro \|_{L^2} \lesssim \sqrt{\me(t)} \md(t).
            \deqq
            It remains to estimate the term $VI_7$. \begin{align*}
           VI_7 = &\; -\var (1+\var) \i Z^{\al} \p_3^2 ( u \cdot \nabla \ln (1+\vro) ) \,  Z^{\al} \p_3^2  \vro \ dx - \var \i Z^{\al}
            {\rm div} (( 1+ \vro) \p_t u) \,  Z^{\al} \p_3^2  \vro \ dx \\
           &\; +\var \i Z^{\al} \Big( -{\rm div} (( 1+ \vro) u \cdot \nabla u) + 2 \Delta_h \du - 2 (1+\vro) |\nabla \vro|^2  -(1+\vro)^2 \Delta_h \vro \Big) \,  Z^{\al} \p_3^2  \vro \ dx,\\
           : =&\; \sum_{i=1}^{3} VI_{7,i}.
            \end{align*}
            First, we decompose the term $VI_{7,1}$ into six terms as follows:
            \begin{align*}
            &\; VI_{7,1} \\ =
            &\;  -\var (1+\var) \i Z^{\al} \p_3^2 (  u_h \cdot \f{\nabla_h \vro}{1+\vro}) \,  Z^{\al} \p_3^2 \vro \ dx  -\var (1+\var) \i Z^{\al} \p_3^2 ( u_3 \, \f{\p_3 \vro}{1+\vro}) \,  Z^{\al} \p_3^2  \vro \ dx\\
            = &\;  -\var (1+\var) \sum_{ 0 \le \be \le  \al} C_{\al}^{\be} \i \Big(Z^{\be} \p_3^2 (\f{u_h}{1+\vro}) \cdot \nabla_h  Z^{\al-\be} \vro  + Z^{\be} \p_3  (\f{u_h}{1+\vro}) \cdot \p_3 \nabla_h Z^{\al-\be}  \vro \Big) \,  Z^{\al} \p_3^2  \vro \ dx \\
            &\;  -\var (1+\var) \sum_{ 0 \le \be <  \al} C_{\al}^{\be} \i Z^{\be}  (\f{u_h}{1+\vro}) \cdot \nabla_h Z^{\al-\be} \p_3^2 \vro \,  Z^{\al} \p_3^2  \vro \ dx  -\var (1+\var) \i  \f{u_h}{1+\vro} \cdot \nabla_h Z^{\al} \p_3^2   \vro \,  Z^{\al} \p_3^2  \vro \ dx \\
            &\; -\var (1+\var) \sum_{ 0 \le \be \le  \al} C_{\al}^{\be} \i \Big(Z^{\be} \p_3^2 (\f{u_3}{1+\vro}) \, Z^{\al-\be} \p_3 \vro  + Z^{\be} \p_3 (\f{u_3}{1+\vro}) \, Z^{\al-\be} \p_3^2 \vro \Big) \,  Z^{\al} \p_3^2  \vro \ dx \\
            &\; -\var (1+\var) \sum_{ 0 \le \be <  \al} C_{\al}^{\be} \i Z^{\be} (\f{u_3}{1+\vro}) \,  Z^{\al-\be} \p_3^3 \vro \,  Z^{\al} \p_3^2  \vro \ dx  -\var (1+\var) \i \f{u_3}{1+\vro} \, Z^{\al} \p_3^3 \vro \,  Z^{\al} \p_3^2  \vro \ dx \\
            = &\;\sum_{i=1}^{6} VI_{7,1,i}.
            \end{align*}
            We only give the proof of $VI_{7,1,3}$ , $VI_{7,1,5}$  and $VI_{7,1,6}$, and the other terms can be estimated similarly. 
            If $\al_3 \ge 1$, integrating by parts and using the property of the commutator \eqref{properity-p3}, we have
            \begin{align*}
            &\; VI_{7,1,3} +VI_{7,1,6} \\
            = &\; -\var (1+\var) \i  \f{1}{1+\vro} (\, u_h \cdot \nabla_h  Z^{\al} \p_3^2  \vro  + u_3 \, \p_3  Z^{\al} \p_3^2 \vro) \cdot Z^{\al} \p_3^2 \vro \ dx \\
            &\; -\var (1+\var) \i  \f{1}{1+\vro} \, u_3 \,(Z^{\al} \p_3 - \p_3 Z^{\al}) \p_3^2 \vro  \cdot Z^{\al} \p_3^2 \vro \ dx \\
            = &\; \var (1+\var) \i  \Big(\f{1}{1+\vro} \, \du + u \cdot \nabla (\f{1}{1+\vro}) \Big)|Z^{\al} \p_3^2 \vro|^2 \ dx  \\
            &\; -\var (1+\var) \i  \f{1}{1+\vro} \, u_3 \,(Z^{\al} \p_3 - \p_3 Z^{\al}) \p_3^2 \vro  \cdot Z^{\al} \p_3^2 \vro (\chi + 1- \chi)^2 dx  \\
             \lesssim &\; \var \| Z^{\al} \p_3^2 \vro\|_{L^2}^2 (  \| \du\|_{L^{\infty}} + \| u\|_{L^{\infty}} \| \nabla \vro \|_{L^{\infty}}) + \var \| Z^{\al} \p_3^2 \vro\|_{L^2} \| \p_3^2 \vro\|_{H_{co}^2} (\|u_3\|_{L^{\infty}} + \|\p_3 u_3\|_{L^{\infty}}) \\
             \lesssim &\; \sqrt{\me(t)} \md(t).
            \end{align*}
            If $\al_3 =0$, similarly we have
            \beqq  \bal
            VI_{7,1,3} +VI_{7,1,6}  =  \var (1+\var) \i  \Big(\f{1}{1+\vro} \, \du + u \cdot \nabla (\f{1}{1+\vro}) \Big)|Z^{\al} \p_3^2 \vro|^2 \ dx  \lesssim  \sqrt{\me(t)} \md(t).
            \dal \deqq
            Thus, combining the estimates for the two cases, we have
            \beqq  \bal
            VI_{7,1,3} +VI_{7,1,6}  \lesssim \sqrt{\me(t)} \md(t).
            \dal \deqq
            Now we estimate the term $VI_{7,1,5}$. If $\be_3=0$, using the boundary condition $u_3|_{x_3=0}=0 $ and similar to the estimate \eqref{ah1-II14} of $II_{1,4}$, we can check
            \beqq \bal
VI_{7,1,5} \lesssim &\;  \var \sum_{ 0 \le \be <  \al}  \i Z^{\be} (\f{u_3}{1+\vro}) \,  Z^{\al-\be} \p_3^3 \vro \,  Z^{\al} \p_3^2  \vro \, (1-\chi)^2 \ dx  \\
&\; +   \var \sum_{ 0 \le \be <  \al} \i \int_{0}^{x_3}  Z^{\be} \p_3 (\f{u_3}{1+\vro}) \ dz \,  Z^{\al-\be} \p_3^3 \vro \,  Z^{\al} \p_3^2  \vro \,  \chi^2 \ dx\\
\lesssim &\; \var \| Z^{\al} \p_3^2 \vro\|_{L^2} \sum_{ 0 \le \be <  \al} \| Z^{\al-\be}Z_3  \p_3^2 \vro\|_{L^2}  (\|Z^{\be} (\f{u_3}{1+\vro})\|_{L^{\infty}}   + \|Z^{\be} \p_3 (\f{u_3}{1+\vro})\|_{L^{\infty}})\\
  \lesssim &\; \sqrt{\me(t)} \md(t).
 \dal \deqq
 Thus, we can obtain that
 \beqq \bal
 VI_{7,1}  \lesssim \sqrt{\me(t)} \md(t).
 \dal \deqq
Similarly, under the assumption \eqref{assumption}, we can obtain that
\beqq \bal
 VI_{7,2} \lesssim   &\; \var \| Z^{\al} \p_3^2  \vro\|_{L^2} \|Z^{\al} \p_t \du\|_{L^2} +  \sqrt{\me(t)} \md(t) \\
 \le   &\; \nu \var \| Z^{\al} \p_3^2  \vro\|_{L^2}^2 + C_{\nu} (\|\p_t \du\|_{H_{tan}^2}^2 + \| \p_t (\nabla_h  \cdot u_h)\|_{H_{co}^2}^2 + \| (Z_3^2 \p_t \p_3 u_3, \nabla_h Z_3 \p_t \p_3 u_3)\|_{L^2}^2) + \sqrt{\me(t)} \md(t),  \\
\le   &\; \nu \var \| Z^{\al} \p_3^2  \vro\|_{L^2}^2 + C_{\nu} \Big(\|\p_t \du\|_{H_{tan}^2}^2 + \| Z_3 \p_t \p_3^2 u_3\|_{L^2}  + \| \p_t \p_3^2 u_3\|_{H_{tan}^1}^2 \\
&\; + \|(\Delta_h u,\nabla_h \du, \nabla_h \vro)\|_{H_{co}^3}^2  + \var\| \p_3^2 u\|_{H_{co}^3}^2 \Big)  +  \sqrt{\me(t)} \md(t),  \\
VI_{7,3} \le  &\;\nu \var \| Z^{\al} \p_3^2  \vro\|_{L^2}^2 + C_{\nu} (\|\Delta_h \du\|_{H_{co}^2}^2 +  \| \nabla \vro\|_{H_{co}^3}^2)+  \sqrt{\me(t)} \md(t),
\dal \deqq
where we have used the following estimate by using the equation $\eqref{eqr}_2$,
\beqq
 \| \p_t u_h\|_{H_{co}^3}^2 \lesssim \|(\Delta_h u ,\nabla_h \du, \nabla_h \vro)\|_{H_{co}^3}^2 + \var\| \p_3^2 u\|_{H_{co}^3}^2 + \me(t) \md(t).
 \deqq
 Combining the estimates from $VI_{7,1}$ to $VI_{7,3}$, we obtain
 \beqq \bal
 VI_{7} \le &\; \nu \var \| Z^{\al} \p_3^2  \vro\|_{L^2}^2 + C_{\nu} \Big(\|\p_t \du\|_{H_{tan}^2}^2 + \| Z_3 \p_t \p_3^2 u_3\|_{L^2}^2  + \| \p_t \p_3^2 u_3\|_{H_{tan}^1}^2 \\
&\; +\|(\Delta_h u,\nabla_h \du, \nabla \vro)\|_{H_{co}^3}^2 \Big)+  \sqrt{\me(t)} \md(t),
 \dal \deqq
 which, together with the estimates from $VI_1$ to $VI_6$, yields that
 \beqq  \bal
  &\; \var(1+\var) \frac{d}{dt}  \i \f{1}{1+\vro} | Z^{\al} \p_3^2 \vro|^2 \ dx
    				+\var \| (1+ \vro) Z^{\al} \p_3^2 \vro\|_{L^2}^2  \\
				\lesssim &\; \|\p_t \du\|_{H_{tan}^2}^2 + \| Z_3 \p_t \p_3^2 u_3\|_{L^2}^2  + \| \p_t \p_3^2 u_3\|_{H_{tan}^1}^2
                +\|( \Delta_h u,\nabla_h \du, \nabla \vro)\|_{H_{co}^3}^2   + \sqrt{\me(t)} \md(t).
           \dal \deqq
  Combining the above estimate and $\eqref{3904}$, we finish the proof of this lemma.
	\end{proof}

        \subsection{$L^{\infty}$ estimate} \label{lines}
               The aim of this subsection is to establish the $L^\infty$ estimates for
	the quantities $w_h$ and $\p_3 \vro$ in $\me(t)$.  First of all, let us establish the estimate for  $ \| w_h\|_{L^{\infty}}$.
               \begin{lemm}\label{lemma-wh}
				For any smooth solution $(\vro,u)$ of equation \eqref{eqr}, under the assumption \eqref{assumption},
				it holds that
                \beq\label{wh01}\bal
                    \|w_h(t) \|_{L^\infty}  \lesssim \|w_h(0)\|_{L^\infty} + \dl^{^{\f{5}{6}}}.
                \dal \deq
            \end{lemm}
            \begin{proof}
                Recalling the definition of $w_h=(\nabla \times u)_h$, it solves the following equation:
                \beq \label{equ-ptwh} \bal
             &\p_t w_h + u \cdot \nabla w_h - \f{1}{1+\vro} (\Delta_h w_h + \var \p_3^2 w_h)\\
             &= w_h \du + w\cdot \nabla u_h  - [\nabla(\f{\vro}{1+\vro} ) \times (\Delta_h u + \var \p_3^2 u + \nabla \du)]_h,
                \dal \deq
                and the boundary condition is 
                $w_h|_{x_3=0}=0$.
                Applying the maximum principle for \eqref{equ-ptwh}, we can obtain that
		\beq \label{wh-inf}
		\bal
\|w_h(t) \|_{L^\infty}  \leq &\; \|w_h(0)\|_{L^\infty} + \int_{0}^t \|w_h \du + w\cdot \nabla u_h\|_{L^\infty}  \, d\tau  +\int_{0}^t \|[\nabla(\f{\vro}{1+\vro} ) \times \Delta_h u]_h\|_{L^\infty}  \, d\tau \\
&\;  +\int_{0}^t \|[\nabla(\f{\vro}{1+\vro} ) \times \nabla \du]_h\|_{L^\infty} \, d\tau + \var \int_{0}^t \|[\nabla(\f{\vro}{1+\vro} ) \times \p_3^2 u]_h\|_{L^\infty}  \, d\tau.
         \dal \deq
         We estimate the terms on the right-hand side of \eqref{wh-inf} in the sequence.
         \beq \label{wh02} \bal
&\int_{0}^t \|w_h \du + w\cdot \nabla u_h\|_{L^\infty}  \, d\tau\\
\lesssim &\; \sup_{ 0\le \tau \le t} \|w_h(\tau)\|_{L^\infty} \int_{0}^t \|\du\|_{L^\infty}  \, d\tau + \sup_{ 0\le \tau \le t} \|w_h(\tau)\|_{L^\infty} \int_{0}^t \|\nabla_h u_h\|_{L^\infty} \, d\tau \\
&\; + \sup_{ 0\le \tau \le t} \|\p_3 u_h(\tau)\|_{L^\infty} \int_{0}^t \|w_3\|_{L^\infty} \, d\tau \\
\lesssim &\; \sup_{ 0\le \tau \le t} \|(w_h, \nabla_h u)(\tau)\|_{L^\infty} \int_{0}^t \|(\du, \nabla_h u) (\tau)\|_{L^\infty} \, d\tau
\lesssim \; \dl^{\f{5}{6}},
         \dal \deq
         where, under the assumption \eqref{assumption} and using the inequality $\eqref{ie:Sobolev}_1$, we have used the following estimates
         \begin{align*} 
&\int_{0}^t \|\du (\tau)\|_{L^\infty}  \, d\tau  \\
\lesssim &\; \int_{0}^t \|(\du, \p_3 \du)\|_{L^2}^{\f14}  \|\nabla_h (\du, \p_3 \du)\|_{H_{tan}^1}^{\f34}  \, d\tau \\
\lesssim &\; \left\{\int_{0}^t \|(\du, \p_3 \du)\|_{L^2}^{2} (1+\tau)^{\sigma} \, d\tau\right\}^{\f18}  \left\{\int_{0}^t \|\nabla_h (\du, \p_3 \du)\|_{H_{tan}^1}^{2} (1+\tau)^{1+\sigma} \, d\tau\right\}^{\f38}\\
&\; \times \left\{\int_{0}^t (1+\tau)^{-\f{3+ 4\sigma}{4} } \, d\tau\right\}^{\f12} \\
\lesssim &\; \dl^{\f12},
   \end{align*}
   and similarly, we have 
   \begin{align*}
 \int_{0}^t \|\nabla_h u (\tau)\|_{L^\infty}  \, d\tau 
\lesssim  \dl^{\f12}.
         \end{align*}
         Notice that for any smooth functions $f(x)$ and $g(x)$,
         \beqq
         [f \times g]_h = \left[\begin{array}{*{4}{ll}}
          f_2  g_3 - f_3 g_2\\
          f_3 g_1 - f_1 g_3
         \end{array}\right].
         \deqq
         Thus, using the $L^{\infty}$ anisotropic inequality $\eqref{ie:Sobolev}_1$, we have
         \beq\label{wh04} \bal
&\; \int_{0}^t \|[\nabla(\f{\vro}{1+\vro} ) \times \nabla \du]_h\|_{L^\infty}  \, d\tau\\
\lesssim &\;   \int_{0}^t  \|\p_3 \vro  \|_{L^\infty} \|\Delta_h \du\|_{L^\infty}   \, d\tau + \int_{0}^t \|\nabla_h \vro\|_{L^\infty} \|\p_3 \du\|_{L^\infty}  \, d\tau\\
\lesssim &\; \sup_{ 0\le \tau \le t} \| \p_3 \vro (\tau) \|_{L^\infty}    \left\{ \int_{0}^t \|(\Delta_h u, \p_3 \Delta_h u)\|_{H_{tan}^2}^2   (1+\tau)^{1+\sigma} \, d\tau \right\}^{\f12}\left\{\int_{0}^t (1+\tau)^{-(1+ \sigma) } \, d\tau\right\}^{\f12} \\
&\; + \int_{0}^t  \|\nabla_h \vro\|_{L^2}^{\f18} \|\nabla \nabla_h \vro\|_{H_{tan}^1}^{\f78} \|\p_3 \du\|_{L^\infty}  \, d\tau\\
\lesssim &\; \dl^{\f{5}{6}} +\int_{0}^t \|\nabla_h \vro\|_{L^2}^{\f18} \|\nabla \nabla_h \vro\|_{H_{tan}^1}^{\f78} \|\p_3 \du\|_{L^\infty}  \, d\tau.
         \dal \deq
Now we deal with the term $\|\p_3 \du\|_{L^\infty}$. 
Using the equation of $u_3$ in $\eqref{equ-p32u3}_1$, we can check that
\beq  \label{equ-p3du}
(1+\var) \p_3 \du = (1+ \vro) \p_t u_3 + (1+ \vro) u \cdot \nabla u_3  - \Delta_h u_3 + \var \p_3 \nabla_h \cdot u_h + (1+\vro)^2 \p_3 \vro.
\deq
Furthermore, under the assumption \eqref{assumption}, we have
\beqq \bal
&\; \i \|\nabla_h \vro\|_{L^\infty} \, d\tau \lesssim \i \|\nabla_h \vro\|_{L^2}^{\f18} \|\nabla \nabla_h \vro\|_{H_{tan}^2}^{\f78} \, d\tau\\
\lesssim &\; \left\{ \int_{0}^t  \| \nabla_h \vro\|_{L^2}^2   (1+\tau)^{\sigma} \, d\tau \right\}^{\f{1}{16}} \left\{ \int_{0}^t  \| \nabla \nabla_h \vro\|_{H_{tan}^2}^2   (1+\tau)^{1+\sigma} \, d\tau \right\}^{\f{7}{16}} \left\{\int_{0}^t (1+\tau)^{-\f{7+8\sigma}{8} } \, d\tau\right\}^{\f12}\\
\lesssim &\; \dl^{\f12}.
\dal \deqq
Thus, we have
\begin{align} \label{wh-help} 
&\; \int_{0}^t \|\nabla_h \vro\|_{L^\infty} \|\p_3 \du\|_{L^\infty} \, d\tau \notag\\
\lesssim &\; (1+\sup_{ 0\le \tau \le t} \| u (\tau) \|_{L^{\infty}} ) \int_{0}^t \|(\nabla_h \vro, \nabla \nabla_h \vro)\|_{H_{tan}^2} (\| \nabla u_3\|_{L^{\infty}} + \|\Delta_h u_3\|_{L^{\infty}} +  \var \|\p_3 \nabla_h u\|_{L^{\infty}} )\, d\tau \notag\\
&\; +
\sup_{ 0\le \tau \le t} \| (\p_t u_3, \p_3 \vro) (\tau) \|_{L^{\infty}}  \int_{0}^t \|\nabla_h \vro\|_{L^\infty} \, d\tau \\
\lesssim &\;  (1+\sup_{ 0\le \tau \le t} \| u (\tau) \|_{L^{\infty}} ) \int_{0}^t \md_{tan} (\tau)\, d\tau  +
\sup_{ 0\le \tau \le t} (\| (\p_t u_3, \p_3 \p_t u_3)\|_{H_{tan}^2}  + \|\p_3 \vro (\tau) \|_{L^{\infty}})  \int_{0}^t \|\nabla_h \vro\|_{L^\infty} \, d\tau \notag\\
\lesssim &\; \dl^{\f{5}{6}}.\notag
\end{align}
Substituting the above estimates into \eqref{wh04}, we can  yield that
\beq\label{wh05}
\int_{0}^t \|[\nabla(\f{\vro}{1+\vro} ) \times \nabla \du]_h\|_{L^\infty}  \, d\tau \lesssim \dl^{\f{5}{6}}.
\deq
It remains to estimate the last term on the right-hand side of \eqref{wh-inf}. Under the assumption \eqref{assumption} and using the $L^{\infty}$ anisotropic inequality $\eqref{ie:Sobolev}_1$, we have
\beq \label{wh06} \bal
&\; \var \int_{0}^t \|[\nabla(\f{\vro}{1+\vro} ) \times \p_3^2 u]_h\|_{L^\infty}  \, d\tau \\
\lesssim &\; \var \int_{0}^t \|\p_3 \vro\|_{L^\infty} \| \p_3^2 u_h\|_{L^\infty}  \, d\tau + \var \int_{0}^t \|\nabla_h \vro\|_{L^\infty} \| \p_3^2 u_3\|_{L^\infty} \, d\tau\\
\lesssim &\; \var \int_{0}^t  \|\p_3 \vro\|_{L^2}^{\f18} \|\p_3 \nabla_h \vro\|_{H_{tan}^1}^{\f38} \|\p_3^2 \vro\|_{H_{tan}^2}^{\f12} \|\p_3^2 u_h\|_{H_{tan}^2}^{\f12} \|\p_3^3 u_h\|_{H_{tan}^2}^{\f12} d\tau \\
&\; + \var \int_{0}^t \|\nabla_h \vro\|_{L^\infty} \| \p_3 \du \|_{L^\infty} \, d\tau +  \var \int_{0}^t \|\nabla_h \vro\|_{L^\infty} \| \p_3 \nabla_h u \|_{L^\infty} \, d\tau \\
\lesssim &\;
 \sup_{ 0\le \tau \le t} \me(\tau)^{\f12}  \left\{ \int_{0}^t  \| \p_3 \vro\|_{L^2}^2   (1+\tau)^{\sigma} \, d\tau \right\}^{\f{1}{16}}
 \left\{ \int_{0}^t  \| \p_3 \nabla_h \vro\|_{H_{tan}^1}^2   (1+\tau)^{1+\sigma} \, d\tau \right\}^{\f{3}{16}} \\
&\; \times \left\{ \int_{0}^t  \var \| \p_3^2 \vro\|_{H_{tan}^2}^2  \, d\tau \right\}^{\f{1}{4}} \left\{ \int_{0}^t  \var \| \p_3^2 u_h \|_{H_{tan}^2}^2  \, d\tau \right\}^{\f{1}{4}}
\left\{\int_{0}^t (1+\tau)^{-\f{3+4\sigma}{4} } \, d\tau\right\}^{\f14} +  \int_{0}^{t} \md(\tau) d \tau+  \dl^{\f{5}{6}}\\
\lesssim &\; \dl^{\f{5}{6}},
\dal \deq
where we have used the following estimate by the equation \eqref{equ-p32u3}
\beqq \bal
 \var \| \p_3^3 u_h \|_{H_{tan}^2}= &\; \| \p_3 \Big( (1+\vro) \p_t u_h + (1+ \vro) u \cdot \nabla u_h  - \Delta_h u_h -  \nabla_h \du + (1+\vro)^2 \nabla_h \vro \Big) \|_{H_{tan}^2} \\
 \lesssim &\; \me(t)^{\f12} + \md(t)^{\f12},
\dal \deqq
and the estimate \eqref{wh-help}. 
Similarly, we can check
\beq\label{wh03} \bal
 \int_{0}^t \|[\nabla(\f{\vro}{1+\vro} ) \times \Delta_h u]_h\|_{L^\infty}  \, d\tau
\lesssim \dl^{\f{5}{6}}.
         \dal \deq
Substituting the estimates from \eqref{wh02}, \eqref{wh05}, \eqref{wh06} and \eqref{wh03} into \eqref{wh-inf}, we finish the proof of this lemma.
\end{proof}

Next, we will estimate $\| \p_3 \vro \|_{L^{\infty}}$.
\begin{lemm}\label{lemma39}
				For any smooth solution $(\vro,u)$ of equation \eqref{eqr}, under the assumption \eqref{assumption},
				it holds that
                \beq \label{vro01} \bal
              &\;  \| \p_3 \vro(t) \|_{L^{\infty}} \lesssim \| \p_3 \vro_0 \|_{L^{\infty}}  + \dl^{\f12}.
                \dal \deq
            \end{lemm}
            \begin{proof}
            It is easy to check that $\p_3 \vro$ solves the following equation:
            \beqq  \bal
            \p_t \p_3 \vro + u \cdot \nabla \p_3 \vro  + (1+\vro) \p_3 \du = - \p_3 \vro \,  \du - \p_3 u \cdot \nabla \vro.
            \dal \deqq
            Substituting the equation \eqref{equ-p3du} of $\p_3 \du$ into above equation, we have
            \beqq \bal
            \p_t \p_3 \vro + u \cdot \nabla \p_3 \vro  + \f{1}{1+\var} \p_3 \vro = &\; - \f{1+\vro}{1+\var} \Big( (1+ \vro) \p_t u_3 + (1+ \vro) u \cdot \nabla u_3  - \Delta_h u_3 + \var \p_3 \nabla_h \cdot u_h \Big) \\
            &\; - \p_3 \vro\, \du - \p_3 u \cdot \nabla \vro - \f{1}{1+\var} \Big((1+\vro)^3 -1 \Big) \p_3 \vro\\
            := &\; G(t).
            \dal \deqq
            Thus we have
            \beqq
        \p_t (e^{\f{1}{1+\var} t} \p_3 \vro  )+ u \cdot \nabla (e^{\f{1}{1+\var} t} \p_3 \vro) = e^{\f{1}{1+\var} t} G(t).
            \deqq
             Applying the maximum principle for above transport equation, we have
            \begin{equation*}
        e^{\f{1}{1+\var} t} \| \p_3 \vro(t) \|_{L^{\infty}} \lesssim \| \p_3 \vro_0 \|_{L^{\infty}}  + \int_{0}^t e^{\f{1}{1+\var} \tau} \|G(\tau)\|_{L^{\infty}} \, d\tau,
         \end{equation*}
         which yields that
         \begin{equation}
           \label{p3vro-inf}
         \| \p_3 \vro(t) \|_{L^{\infty}} \lesssim \| \p_3 \vro_0 \|_{L^{\infty}}  + \int_{0}^t \|G(\tau)\|_{L^{\infty}} \, d\tau.
         \end{equation}
         Now we estimate the term $\int_{0}^t \|G(\tau)\|_{L^{\infty}} \, d\tau$. From the equality $\eqref{equ-p3du}$, we can obtain
         \beqq
         \p_t u_3 = - u \cdot \nabla u_3 + \f{1}{1+\vro} \Delta_h u_3 -\f{\var}{1+\vro}  \p_3 \nabla_h \cdot u_h+ \f{1+\var}{1+\vro} \p_3 \du - (1+\vro) \p_3 \vro,
         \deqq
         thus we can check that
         \beq \label{pro-ptu3} \bal
         \| \p_t u_3\|_{L^2}^2
         \lesssim  \md_{tan}(t), ~~~\|
         \nabla_h \p_t u_3\|_{H_{tan}^1}^2 \lesssim \bmd_{tan}(t).
         \dal \deq
         Thus, using the interpolation inequality \eqref{inter-1} and under the assumption \eqref{assumption}, we have
          \begin{align*}
        &\; \int_{0}^t \|\p_t u_3\|_{L^{\infty}} \, d\tau\\
        \lesssim &\;  \int_{0}^t \|\p_t u_3\|_{L^2}^{\f18} \| \nabla_h \p_t u_3\|_{H_{tan}^1}^{\f38} \| \p_t \p_3 u_3 \|_{L^2}^{\f18} \| \nabla_h \p_t \p_3 u_3 \|_{L^2}^{\f14} \| \nabla_h^2 \p_t \p_3 u_3 \|_{L^2}^{\f18}  \, d\tau \\
        \lesssim &\;  \int_{0}^t \|\p_t u_3\|_{L^2}^{\f{3}{16}} \| \nabla_h \p_t u_3\|_{H_{tan}^1}^{\f{9}{16}} \| \p_t \p_3^2 u_3 \|_{H_{tan}^2}^{\f14}  \, d\tau \\
        \lesssim &\; \left\{ \int_{0}^t \|\p_t u_3\|_{L^2}^2 (1+\tau)^{\sigma} \, d\tau \right\}^{\f{3}{32}} \left\{ \int_{0}^t  \| \nabla_h \p_t u_3\|_{H_{tan}^1}^2 (1+\tau)^{1+\sigma} \, d\tau \right\}^{\f{9}{32}}  \left\{ \int_{0}^t \| \p_t \p_3^2 u_3 \|_{H_{tan}^2}^2  \, d\tau \right\}^{\f{1}{8}}
\\
        &\; \times \left\{\int_{0}^t (1+\tau)^{-\f{9+12\sigma}{16} }\, d\tau \right\}^{\f{1}{2}}
\lesssim \dl^{\f12},
\end{align*}
and
\begin{align*}
\var \int_{0}^t \|\p_3 \nabla_h u\|_{L^{\infty}} \, d\tau  \lesssim &\;\var^{\f34} \left\{ \int_{0}^t \|\p_3 \nabla_h u\|_{H_{tan}^2}^2 (1+\tau)^{\sigma} \, d\tau \right\}^{\f{1}{4}} \left\{ \var \int_{0}^t \|\p_3^2 \nabla_h u\|_{H_{tan}^2}^2 (1+\tau)^{1+\sigma} \, d\tau \right\}^{\f{1}{4}}\\
&\; \times
\left\{\int_{0}^t (1+\tau)^{-\f{1+2\sigma}{2} }\, d\tau \right\}^{\f{1}{2}}\lesssim \var^{\f34} \dl^{\f12}.
        \end{align*}
         Similarly, we can check that
         \beqq \bal
          &\;\int_{0}^t (\|u \cdot \nabla u_3\|_{L^{\infty}} +\|\Delta_h u\|_{L^{\infty}}) \, d\tau  \lesssim \dl^{\f12},\\
 &\; \int_{0}^t \Big( \|  \p_3 \vro\, \du\|_{L^{\infty}} + \|\p_3 u \cdot \nabla \vro\|_{L^{\infty}}  + \|\f{1}{1+\var} ((1+\vro)^3 -1 ) \p_3 \vro\|_{L^{\infty}} \Big) \, d\tau \lesssim \dl^{\f56}.
        \dal \deqq
        Combining above $L^{\infty}$ estimates of the term $G(t)$, we have
        \beqq
        \int_{0}^t \|G(\tau)\|_{L^{\infty}} \, d\tau \lesssim \dl^{\f12}.
        \deqq
        Substituting the above inequality into \eqref{p3vro-inf}, we finish the proof of this lemma.
            \end{proof}
            
			\subsection{Time derivative estimate}\label{timer}
In this subsection, we will establish the time derivative estimates in $\me(t)$.
 \begin{lemm}\label{lemma310}
				For any smooth solution $(\vro,u)$ of equation \eqref{eqr},
				it holds that 
                \begin{align} 
                &\; \frac{d}{dt} (\|  \p_t u\|_{L^2}^2 +\|  \p_t u\|_{\dot{H}_{tan}^2}^2) + \|\nabla_h  \p_t u \|_{H_{tan}^2}^2\notag \\
				 &\; +\var \|\p_3  \p_t u \|_{H_{tan}^2}^2
    				+  \| \p_t \du \|_{H_{tan}^2}^2 \lesssim \|\du\|_{H_{tan}^2}^2 +  \sqrt{\me(t)} \md(t),\label{3601}\\ 
                  &\;  \frac{d}{dt} (\| \p_t w_h\|_{L^2}^2+\| \p_t w_h\|_{\dot{H}_{tan}^2}^2 + \| Z_3 \p_t w_h\|_{L^2}^2 ) + \|\nabla_h  \p_t w_h \|_{H_{tan}^2}^2 +\|\nabla_h Z_3 \p_t w_h \|_{L^{2}}^2 \notag\\
    				&\; +  \var(\|\p_3 \p_t w_h \|_{H_{tan}^2}^2 + \|\p_3 Z_3 \p_t w_h \|_{L^{2}}^2)
				\lesssim  \sqrt{\me(t)} \md(t)\label{3701},
                 \end{align}
                 and 
                \beq\label{3801} \bal
                  &\; \frac{d}{dt} (\|\p_t \p_3 u_3  \|_{L^2}^2 + \|\p_t \p_3 u_3  \|_{\dot{H}_{tan}^2}^2+ \| Z_3 \p_t \p_3 u_3 \|_{L^2}^2) + \| \nabla \p_t \p_3 u_3\|_{H_{tan}^2}^2 + \| Z_3 \nabla \p_t \p_3 u_3\|_{L^{2}}^2 \\
				\lesssim &\; \|\nabla_h \p_t w_h\|_{H_{tan}^1}^2 + \| \nabla_h \p_t u \|_{H_{tan}^2}^2+  \| \p_3 \du \|_{H_{co}^2}^2 + \sqrt{\me(t)} \md(t).
                \dal \deq
            \end{lemm}
\begin{proof}
    We shall present the proof of the estimate for \eqref{3801} only, as the proofs for \eqref{3601}-\eqref{3701} follow essentially the same strategy and are thus omitted. 
Recall that $u_3$ solves the following equality:
    \beq\label{equ-ptu3} \bal
  &\; \p_t u_3 - \Delta_h u_3 - (1+\var) \p_3^2 u_3 - \p_3 \nabla_h \cdot u_h + \p_3 \vro \\
   =&\;  - u \cdot \nabla u_3 - \f{\vro}{1+\vro} (\Delta_h u_3 + (1+\var) \p_3^2 u_3 + \p_3 \nabla_h \cdot u_h) - \vro \, \p_3 \vro,
    \dal \deq
    here we have used $\p_3 \du = \p_3^2 u_3 + \p_3 \nabla_h \cdot u_h$.
    For $0 \le |\ah| \le 2$, applying $ Z^{\ah} \p_t$-operator to the above equality and multiplying  by $-Z^{\ah} \p_t \p_3^2   u_3$ and integrating over $\mathbb{R}_+^3$, we can obtain that
            \beqq \bal
            &\; \frac{d}{dt} \frac{1}{2}\i | Z^{\ah} \p_t \p_3 u_3 |^2 dx
						+\| Z^{\ah}  \nabla_h  \p_t \p_3 u_3\|_{L^2}^2 + (1+\var)\| Z^{\ah} \p_t \p_3^2 u_3\|_{L^2}^2 \\
						=&\; - \i Z^{\ah} \p_t \p_3 \nabla_h \cdot u_h \cdot Z^{\ah} \p_t \p_3^2 u_3 \ dx+ \i Z^{\ah} \p_t (u \cdot \nabla u_3 ) \cdot Z^{\ah} \p_t \p_3^2 u_3 \ dx \\
                        &\; +  \i Z^{\ah}  \Big(\p_t (\f{\vro}{1+\vro}) (\Delta_h u_3 + (1+\var) \p_3^2 u_3 + \p_3 \nabla_h \cdot u_h) \Big) \cdot Z^{\ah} \p_t \p_3^2 u_3 \ dx \\
                        &\;+  \i Z^{\ah}  \Big( \f{\vro}{1+\vro} \p_t(\Delta_h u_3 + (1+\var) \p_3^2 u_3 + \p_3 \nabla_h \cdot u_h) \Big) \cdot Z^{\ah} \p_t \p_3^2 u_3 \ dx \\
                        &\; +\i Z^{\ah} (\p_t \vro \, \p_3 \vro +  (1+\vro) \, \p_t \p_3 \vro ) \cdot Z^{\ah} \p_t \p_3^2 u_3 \ dx\\
                        := &\; \sum_{i=1}^{5} VII_{i}.
            \dal  \deqq
            Using H\"older inequality, we have
            \beqq
            VII_{1} \le \nu \|Z^{\ah} \p_t \p_3^2 u_3  \|_{L^2}^2 + C_{\nu} \|(Z^{\ah} \p_t w_h, Z^{\ah} \nabla_h \p_t u) \|_{L^2}^2.
            \deqq
            {\bf Now, we deal with the case $|\ah|=0$.} In this case, by Lemma \ref{lemm:sobolev-ie}, we have
            \beqq \bal
            VII_2 \lesssim &\; \| \p_t \p_3^2 u_3 \|_{L^2} \Big( \| \p_t u_h\|_{L^2}^{\f12} \| \p_1 \p_t u_h\|_{L^2}^{\f12} \| \nabla_h u_3\|_{H_{tan}^1}^{\f12} \|\p_3 \nabla_h u_3\|_{H_{tan}^1}^{\f12} \\
           &\; + \| \p_t u_3\|_{L^2}^{\f12} \| \p_1 \p_t u_3\|_{L^2}^{\f12} \| \p_3 u_3\|_{H_{tan}^2} + \| u \|_{H_{tan}^2} \| \nabla \p_t u_3\|_{L^2}^{\f12} \| \p_3 \nabla \p_t u_3\|_{L^2}^{\f12}  \Big)\\
           \lesssim &\; \sqrt{\me(t)} \md(t),\\
           VII_{3} \lesssim &\; \| \p_t \p_3^2 u_3 \|_{L^2} \| \p_t \vro\|_{L^2}^{\f12} \| \p_t \vro\|_{L^2}^{\f12}  ( \| \Delta_h u_3\|_{H_{tan}^2} + \| \p_3^2 u_3\|_{H_{tan}^2}  + \| \p_3 \nabla_h u_3\|_{H_{tan}^2} )\\
             \lesssim &\; \sqrt{\me(t)} \md(t),
            \dal \deqq
         and similarly, using the estimate \eqref{pro-p3tvro}, we have
        \beqq
       VII_4  +VII_5 \le \nu \| \p_t \p_3^2 u_3 \|_{L^2}^2+ C_{\nu} \|\p_3 \du\|_{L^2}^2 +  \sqrt{\me(t)} \md(t) .
       \deqq
       Thus, using the smallness of $\nu$, we obtain that
       \beq \label{3802}
       \frac{d}{dt} \i |  \p_t \p_3 u_3 |^2 dx
						+\|  \nabla_h  \p_t \p_3 u_3\|_{L^2}^2 + (1+\var)\| \p_t \p_3^2 u_3\|_{L^2}^2  \lesssim \|(\p_t w_h,  \nabla_h \p_t u ,\p_3 \du) \|_{L^2}^2 + \sqrt{\me(t)} \md(t).
       \deq
        {\bf Next, we deal with the case $|\ah|=2$.}
        Using Lemma \ref{lemm:sobolev-ie}, we can easily check that
        \beqq \bal
        VII_2+ VII_3+ VII_4+ VII_5  \lesssim \nu \| \nabla_h^2 \p_t \p_3^2 u_3 \|_{L^2}^2+ C_{\nu} \|\nabla_h^2 \p_3 \du\|_{L^2}^2 + \sqrt{\me(t)} \md(t),
        \dal \deqq
        here we omit the proof for simplicity. Thus, in this case, we have
         \beq \label{3803} \bal
       &\; \frac{d}{dt} \i | \nabla_h^2 \p_t \p_3 u_3 |^2 dx
						+\|  \nabla_h^3  \p_t \p_3 u_3\|_{L^2}^2 + (1+\var)\|  \nabla_h^2 \p_t \p_3^2 u_3\|_{L^2}^2  \\
                       \lesssim  &\;  \| \nabla_h^2 ( \p_t w_h,  \nabla_h \p_t u ,\p_3 \du) \|_{L^2}^2 + \sqrt{\me(t)} \md(t).
       \dal \deq
      {\bf Finally, we will give the estimate for the term $Z_3 \p_t \p_3 u_3$.}
      Applying $ Z_3 \p_t \p_3  $-operator to the equality \eqref{equ-ptu3} and multiplying  by $Z_3 \p_t \p_3 u_3$ and integrating over $\mathbb{R}_+^3$, we can obtain that
      \beqq \bal
            &\; \frac{d}{dt} \frac{1}{2}\i | Z_3 \p_t \p_3 u_3 |^2 dx
						+\|  Z_3  \nabla_h  \p_t \p_3 u_3\|_{L^2}^2 + (1+\var)\|  Z_3 \p_t \p_3^2 u_3\|_{L^2}^2 \\
						=&\;  (1+\var) \i (Z_3 \p_3 -\p_3 Z_3) \p_t \p_3^2 u_3 \cdot Z_3 \p_t \p_3 u_3 \ dx  + (1+\var)\i Z_3 \p_t \p_3^2 u_3 \cdot (Z_3 \p_3 -\p_3 Z_3) \p_t \p_3 u_3 \ dx\\
                        &\; + \i Z_3 \p_t \p_3 \nabla_h \cdot u_h \cdot Z_3 \p_t \p_3 u_3 \ dx- \i Z_3 \p_t \p_3(u \cdot \nabla u_3 ) \cdot Z_3 \p_t \p_3 u_3\ dx \\
                        &\; -   \i Z_3 \p_3  \Big(\p_t (\f{\vro}{1+\vro}) (\Delta_h u_3+(1+\var) \p_3^2 u_3 + \p_3 \nabla_h \cdot u_h) \Big) \cdot Z_3 \p_t \p_3 u_3 \ dx \\
                        &\;-  \i Z_3 \p_3 \Big( \f{\vro}{1+\vro} \p_t(\Delta_h u_3 + (1+\var) \p_3^2 u_3 +\p_3 \nabla_h \cdot u_h) \Big) \cdot Z_3 \p_t \p_3 u_3 \ dx \\
                        &\; -\i Z_3 \p_3(\p_t \vro \, \p_3 \vro +  (1+\vro) \, \p_t \p_3 \vro ) \cdot Z_3 \p_t \p_3 u_3\ dx\\
                        := &\; \sum_{i=6}^{12} VII_{i},
            \dal  \deqq
            where we have used that
            \beqq \bal
         &\; -\i Z_3 \p_t \p_3^3 u_3 \cdot Z_3 \p_t \p_3 u_3 \ dx \\
         = &\; -\i (Z_3 \p_3 -\p_3 Z_3) \p_t \p_3^2 u_3 \cdot Z_3 \p_t \p_3 u_3 \ dx  - \i \p_3 Z_3 \p_t \p_3^2 u_3 \cdot  Z_3 \p_t \p_3 u_3 \ dx\\
        = &\; -\i (Z_3 \p_3 -\p_3 Z_3) \p_t \p_3^2 u_3 \cdot Z_3 \p_t \p_3 u_3 \ dx  + \|  Z_3 \p_t \p_3^2 u_3\|_{L^2}^2 - \i Z_3 \p_t \p_3^2 u_3 \cdot (Z_3 \p_3 -\p_3 Z_3) \p_t \p_3 u_3 \ dx.
        \dal \deqq
        It is easy to check that $Z_3 \p_3 -\p_3 Z_3=- \p_3 \vf \, \p_3$, and then integrating by parts, we have
        \beqq \bal
       VII_6 = &\; -(1+\var)\i \p_3 \vf \, \p_t \p_3^3 u_3  \cdot Z_3 \p_t \p_3 u_3 \ dx = (1+\var) \i \p_t \p_3^2 u_3 \p_3 (\p_3 \vf \,   \cdot Z_3 \p_t \p_3 u_3 )\ dx \\
        = &\; (1+\var) \i \p_t \p_3^2 u_3 (\p_3^2 \vf \,   \cdot Z_3 \p_t \p_3 u_3 + |\p_3 \vf|^2 \, \p_t \p_3^2 u_3 + \p_3 \vf  \, \cdot Z_3 \p_t \p_3^2 u_3 ) \ dx \\
        \le &\; \nu \|Z_3 \p_t \p_3^2 u_3\|_{L^2}^2 + C_{\nu} \|\p_t \p_3^2 u_3 \|_{L^2}^2,
        \dal \deqq
        and similarly, by H\"older inequality, we have
        \beqq \bal
      VII_7 \le &\; \nu \|Z_3 \p_t \p_3^2 u_3\|_{L^2}^2 + C_{\nu} \|\p_t \p_3^2 u_3 \|_{L^2}^2,\\
      VII_8 \le &\;  \nu \|Z_3 \p_t \p_3^2 u_3  \|_{L^2}^2 + C_{\nu} \|(\nabla_h \p_t w_h, \nabla_h^2 \p_t u, \p_t \p_3^2 u_3) \|_{L^2}^2.
        \dal \deqq
        Integrating by parts and using anisotropic inequalities \eqref{ie:Sobolev}, we have
        \begin{align*}
       VII_{9} = &\; \i  \p_3 \p_t ( u \cdot \nabla u_3) \, \p_3( \vf^2 \p_3 \p_t \p_3 u_3 )\ dx\\
       = &\; \i  Z_3 \p_t ( u \cdot \nabla u_3) \, Z_3 \p_t \p_3^2 u_3 \ dx + 2 \i  \p_3 \p_t ( u \cdot \nabla u_3) \,\vf \,  \p_3 \vf \,    \p_t \p_3^2 u_3 \ dx \\
       = &\; \i  Z_3 \p_t ( u \cdot \nabla u_3) \, Z_3 \p_t \p_3^2 u_3 \ dx - 2 \i  \p_t ( u \cdot \nabla u_3)  \p_3 (\,\vf \,  \p_3 \vf \,    \p_t \p_3^2 u_3) \ dx \\
       = &\; \i  Z_3 \p_t ( u \cdot \nabla u_3) \, Z_3 \p_t \p_3^2 u_3 \ dx - 2 \i  \p_t ( u \cdot \nabla u_3)  (|\p_3  \vf |^2 + \vf \, \p_3^2 \vf)  \p_t \p_3^2 u_3 \ dx \\
      &\;  - 2 \i  \p_t ( u \cdot \nabla u_3) \,\p_3  \vf  Z_3 \p_t \p_3^2 u_3 \ dx \\
       \lesssim &\; \sqrt{\me(t)} \md(t).
        \end{align*}
        Similarly, we can check
        \beqq \bal
        VII_{10} +VII_{11} +VII_{12}  \le \nu \|Z_3 \p_t \p_3^2 u_3  \|_{L^2}^2 + C_{\nu} (\|Z_3 \p_3 \du \|_{L^2}^2 + \| \p_t \p_3^2 u_3\|_{L^2}^2)+ \sqrt{\me(t)} \md(t),
        \dal \deqq
        which, together with the estimates from $VII_{6}$  to $VII_{12}$ and using the smallness of $\nu$, yields that
        \beqq \bal
        &\; \frac{d}{dt} \| Z_3 \p_t \p_3 u_3 \|_{L^2} + \| Z_3 \nabla \p_t \p_3 u_3\|_{L^{2}}^2 \\
				\lesssim &\; \|\nabla_h^2 \p_t u\|_{L^2}^2  +  \|\nabla_h \p_t w_h\|_{L^{2}}^2 + \|  \p_t \p_3^2 u_3 \|_{L^2}^2 + \|Z_3 \p_3 \du\|_{L^2}^2 +  \sqrt{\me(t)} \md(t).
        \dal \deqq
        Multiplying the above equality by small positive suitable constant $\kappa$, together with estimate \eqref{3802}, we have
        \beqq \bal
        &\; \frac{d}{dt} \| (\p_t \p_3 u_3, \kappa Z_3 \p_t \p_3 u_3 )\|_{L^2} + \|  \nabla \p_t \p_3 u_3 \|_{L^2}^2  + \kappa \| Z_3 \nabla \p_t \p_3 u_3\|_{L^{2}}^2 \\
				\lesssim &\; \|\nabla_h^2 \p_t u\|_{L^2}^2  +  \|\nabla_h \p_t w_h\|_{L^{2}}^2 + \|Z_3 \p_3 \du\|_{L^2}^2 +  \sqrt{\me(t)} \md(t).
        \dal \deqq
      This and  estimates \eqref{3803}, we finish the proof of \eqref{3801}. Thus, we finish the proof of this lemma.
\end{proof}

			\subsection{Negative derivative estimate}\label{negg}
			In this subsection, we will establish the estimates in negative
			Sobolev space that will play an important role in establishing
			the decay rate estimate, see Lemma \ref{lemma311} and Lemma \ref{lemma312}.
			\begin{lemm}\label{lemma311}
				Under the assumption \eqref{assumption},
				the smooth solution $(\vro,u)$ of equation \eqref{eqr} has the estimate
				\beq\label{J01}
				\begin{aligned}
					&\|\hs (\vro,u)(t)\|_{L^2}^2+
					\int_0^t (\|\hs(\nabla_h u, \du)\|_{L^2}^2 + \var \| \hs \p_3 u\|_{L^2}^2 )  d\tau
					\lesssim \|\hs (\vro_0, u_0)\|_{L^2}^2+\delta^2.
				\end{aligned}
				\deq
			\end{lemm}
			\begin{proof}
				The equation \eqref{eqr} yields directly
				\beq\label{J02}
				\begin{aligned}
					&\frac{d}{dt}\frac{1}{2}\i(|\hs \vro|^2+|\hs u|^2)dx
					+\i |\nabla_h \hs u|^2 dx+ \i | \hs \du|^2 dx
					+\var \i |\hs \p_3 u|^2 dx\\
					=&-\i \hs(u \cdot \nabla \vro) \cdot \hs \vro \ dx
					-\i \hs( \vro \, \du) \cdot \hs \vro \ dx -\i \hs(u \cdot \nabla u) \cdot \hs u \ dx\\
                    &-\i \hs (\vro \, \nabla \vro) \cdot \hs u \ dx  -\i \hs\Big(\f{\vro}{1+\vro} (\Delta_h u +\var \p_3^2 u+ \nabla \du)  \Big) \cdot \hs u \ dx  \\
					:=&\sum_{i=1}^{5} J_{i}.
				\end{aligned}
				\deq
				Using the inequalities \eqref{ie:Sobolev}  and \eqref{a16},
				we may deduce that
				\begin{align*}
					J_1
					=&-\int_{\mathbb{R}_+}\int_{\mathbb{R}^2}
					\hs(u_h \cdot \nabla_h \vro+u_3 \p_3 \vro)\cdot \hs \vro \ dx_h dx_3\\
					\lesssim
					&\int_{\mathbb{R}_+}(\|\hs(u_h \cdot \nabla_h \vro)\|_{L^2(\mathbb{R}^2)}
					+\|\hs(u_3 \p_3 \vro)\|_{L^2(\mathbb{R}^2)})
					\|\hs \vro\|_{L^2(\mathbb{R}^2)}dx_3\\
					\lesssim
					&\int_{\mathbb{R}_+}(\|u_h \cdot \nabla_h \vro\|_{L^{\frac{1}{\frac{1}{2}+\frac{s}{2}}}(\mathbb{R}^2)}
					+\|u_3 \p_3 \vro\|_{L^{\frac{1}{\frac{1}{2}+\frac{s}{2}}}(\mathbb{R}^2)})
					\|\hs \vro\|_{L^2(\mathbb{R}^2)}dx_3\\
					\lesssim
					&\int_{\mathbb{R}_+}(\|u_h\|_{L^{\frac{2}{s}}(\mathbb{R}^2)}
					\|\nabla_h \vro\|_{L^2(\mathbb{R}^2)}
					+\|u_3\|_{L^{\frac{2}{s}}(\mathbb{R}^2)}
					\|\p_3 \vro\|_{L^2(\mathbb{R}^2)})
					\|\hs \vro\|_{L^2(\mathbb{R}^2)}dx_3\\
					\lesssim
					&\left.\|\|u_h\|_{L^\infty(\mathbb{R}_+)}\right\|_{L^{\frac{2}{s}}(\mathbb{R}^2)}
					\|\nabla_h \vro\|_{L^2(\mathbb{R}_+^3)}
					\|\hs \vro\|_{L^2(\mathbb{R}_+^3)}+\left.\|\|u_3\|_{L^\infty(\mathbb{R}_+)}\right\|_{L^{\frac{2}{s}}(\mathbb{R}^2)}
					\|\p_3 \vro\|_{L^2(\mathbb{R}_+^3)}
					\|\hs \vro\|_{L^2(\mathbb{R}_+^3)}\\
					\lesssim
					&(\|u_h\|_{L^2}\|\p_2 u_h\|_{L^2}
					+\|\p_1 u_h\|_{L^2}\|\p_{12} u_h\|_{L^2})^{\frac{1-s}{2}}
					\|u_h\|_{L^2}^{\frac{2s-1}{2}}\|\p_3 u_h\|_{L^2}^{\frac{1}{2}}
					\|\nabla_h \vro\|_{L^2}\|\hs \vro\|_{L^2}\\
					&+(\|u_3\|_{L^2}\|\p_2 u_3\|_{L^2}
					+\|\p_1 u_3\|_{L^2}\|\p_{12} u_3\|_{L^2})^{\frac{1-s}{2}}
					\|u_3\|_{L^2}^{\frac{2s-1}{2}}\|\p_3 u_3\|_{L^2}^{\frac{1}{2}}
					\|\p_3 \vro\|_{L^2}\|\hs \vro\|_{L^2}\\
					\lesssim
					&\sqrt{\mathcal{E}_{tan}(t)}\sqrt{\md_{tan}(t)}
					\|\hs \vro\|_{L^2}
					+\mathcal{E}_{tan}(t)^{\frac34}
					\md_{tan}(t)^{\frac14}
					\|\hs \vro\|_{L^2}.
				\end{align*}
				Similarly, we can obtain the estimate
				\beqq
				\begin{aligned}
					J_2
					\lesssim
					&\sqrt{\mathcal{E}_{tan}(t)}\sqrt{\md_{tan}(t)}
					\|\hs \vro\|_{L^2},\\
					J_3
					\lesssim
					&\sqrt{\mathcal{E}_{tan}(t)}\sqrt{\md_{tan}(t)}
					\|\hs u\|_{L^2}
					+\mathcal{E}_{tan}(t)^{\frac34}
					\md_{tan}(t)^{\frac14}
					\|\hs u\|_{L^2},\\
                    J_5
					\lesssim
					&\sqrt{\mathcal{E}_{tan}(t)}\sqrt{\md_{tan}(t)}
					\|\hs u\|_{L^2}.
				\end{aligned}
				\deqq
                It remains to estimate $J_4$. Integrating by parts, we have
				\beqq \bal
				J_{4} = &  \f12 \i \hs (\vro^2) \cdot \hs \du \ dx \\
				\lesssim & (\|\vro\|_{L^2}\|\p_2 \vro\|_{L^2}
					+\|\p_1 \vro\|_{L^2}\|\p_{12} \vro\|_{L^2})^{\frac{1-s}{2}}
					\|\vro\|_{L^2}^{\frac{2s-1}{2}}\|\p_3 \vro\|_{L^2}^{\frac{1}{2}}
					\| \vro\|_{L^2} \|\hs \du\|_{L^2} \\
					\lesssim & \var  \|\hs \du\|_{L^2}^2 + C_{\var} \mathcal{E}_{tan}(t)^2.
				\dal \deqq
				Substituting the estimates of terms for $J_1$  through $J_4$
				into \eqref{J02}, using the smallness of $\var$ and integrating over $[0, t]$, we obtain
				\beqq
				\begin{aligned}
					&\frac{1}{2}\|\hs (\vro,u)(t)\|_{L^2}^2+
					\int_0^t (\|\hs(\nabla_h u, \du)\|_{L^2}^2 + \var \| \hs \p_3 u\|_{L^2}^2 )  d\tau\\
					\lesssim
					&\frac{1}{2}\|(\hs \vro_0, \hs u_0)\|_{L^2}^2
					+\underset{0\le \tau \le t}{\sup}\|(\hs \vro, \hs u)(\tau)\|_{L^2}
					\int_0^t \sqrt{\mathcal{E}_{tan}(\tau)}\sqrt{\md_{tan}(\tau)}d\tau\\
					&+\underset{0\le \tau \le t}{\sup}\|(\hs \vro, \hs u)(\tau)\|_{L^2}
					\int_0^t \mathcal{E}_{tan}(\tau)^{\frac34} \md_{tan}(\tau)^{\frac14} d\tau.
				\end{aligned}
				\deqq
				Using the assumption \eqref{assumption} and H\"{o}lder's inequality, we have
				\beq\label{J03}
				\begin{aligned}
					\int_0^t  \sqrt{\mathcal{E}_{tan}(\tau)}\sqrt{\md_{tan}(\tau)} d\tau
					&\le
					\left\{\int_0^t   \mathcal{E}_{tan}(\tau)
					(1+\tau)^{-\sigma} d\tau\right\}^{\frac{1}{2}}
					\left\{\int_0^t  \md_{tan}(\tau) (1+\tau)^{\sigma}
					d\tau\right\}^{\frac{1}{2}}\\
					&\lesssim
					\delta \left\{\int_0^t (1+\tau)^{-(s+\sigma)} d\tau\right\}^{\frac{1}{2}}
					\lesssim \delta, \\
					\int_0^t  \mathcal{E}_{tan}(\tau)^{\frac34} \md_{tan}(\tau)^{\frac14} d\tau
					\lesssim
					&\left\{\int_0^t  \mathcal{E}_{tan}(\tau)
					(1+\tau)^{-\frac13\sigma} d\tau\right\}^{\frac{3}{4}}
					\left\{\int_0^t  \md_{tan}(\tau) (1+\tau)^{\sigma}
					d\tau\right\}^{\frac{1}{4}}\\
					\lesssim&
					\delta \left\{\int_0^t (1+\tau)^{-(s+\frac13\sigma)} d\tau\right\}^{\frac{3}{4}}
					\lesssim \delta,
				\end{aligned}
				\deq
				where we use $\frac{16}{17}\le \sigma<s<1$ in the above estimate.
				Thus, we can obtain the following estimate
				\beqq
				\begin{aligned}
					&\|(\hs \vro, \hs u)(t)\|_{L^2}^2
					+
					\int_0^t (\|\hs(\nabla_h u, \du)\|_{L^2}^2 + \var \| \hs \p_3 u\|_{L^2}^2 )  d\tau
					\lesssim
					\|(\hs \vro_0, \hs u_0)\|_{L^2}^2
					+\delta^2.
				\end{aligned}
				\deqq
				Therefore, we complete the proof of this lemma.
			\end{proof}
			
			Next, we establish the 
            estimate of vertical derivative
			of density and velocity field in negative Sobolev space.
			\begin{lemm}\label{lemma312}
				Under the assumption \eqref{assumption},
				the smooth solution $(\vro,u)$ of equation \eqref{eqr} has the estimate
				\beq\label{K01}
				\begin{aligned}
					&\|\hs \p_3(\vro, u)(t)\|_{L^2}^2
					+\int_0^t (\|\hs \p_3(\nabla_h u, \du)\|_{L^2}^2 + \var \| \hs \p_3^2 u\|_{L^2}^2 ) d\tau
					\lesssim 
					\|\hs \p_3(\vro, u)(0)\|_{L^2}^2+\delta^2.
				\end{aligned}
				\deq
			\end{lemm}
			\begin{proof}
            Applying $\hs \p_3$-operator to $\eqref{eqr}_1$ and multiplying by $\hs \p_3 \vro$ and integrating over $\mathbb{R}_+^3$, we have
                \begin{equation*} \label{hsp3-vro}
					\begin{aligned}
						\frac{d}{dt}\frac{1}{2}\i| \hs \p_3 \vro|^2 \ dx+ \i  \hs \p_3 \du \, \hs \p_3 \vro \ dx
						=\i \hs \p_3  (-u\cdot \nabla \vro -\vro \, \du )\, \hs \p_3 \vro \ dx.
					\end{aligned}
				\end{equation*}
                Next, applying $\hs$-operator to $\eqref{eqr}_2$ and multiplying by $-\p_3^2 \hs u$ and integrating over $\mathbb{R}_+^3$, together with above inequality, we have
				\begin{equation}\label{K02}
					\begin{aligned}
						&\frac{d}{dt}\frac{1}{2}\i(|\hs \wr|^2 + |\hs\wu|^2)dx+\i | \hs \p_3 \nabla_h u|^2 dx +
						 \i | \hs \p_3 \du|^2 dx +\var\i  | \hs \p_3^2 u|^2 dx\\
						=&-\i \hs \p_3 (u \cdot \nabla \vro) \cdot \hs \p_3  \vro \ dx
					-\i \hs \p_3 ( \vro \, \du) \cdot \hs \p_3  \vro \ dx  + \i \hs ( u \cdot \nabla u) \cdot \hs \p_3^2 u \ dx \\
                &
				+\i \hs (\vro \, \nabla \vro) \cdot \hs \p_3^2 u \ dx  +\i \hs\Big(\f{\vro}{1+\vro} (\Delta_h u +\var \p_3^2 u+ \nabla \du)  \Big) \cdot \hs \p_3^2  u \ dx  \\
						:= &\sum_{i=1}^{5}K_i,
					\end{aligned}
				\end{equation}
                where we have used the fact that
                \beqq \bal
              &\;  \i \hs \nabla \du \cdot \hs \p_3^2 u \ dx  = \i \hs \p_3 \du \cdot \hs \p_3^2 u_3 \ dx +   \i \hs \nabla_h \du \cdot \hs \p_3^2 u_h \ dx  \\
               = &\; \i \hs \p_3 \du \cdot \hs \p_3^2 u_3 \ dx +   \i \hs \p_3 \du \cdot \hs \p_3 \nabla_h \cdot u_h \ dx  \\
               = &\; \i | \hs \p_3 \du |^2 dx,
                \dal \deqq
                and
                \beqq \bal
              &\;  \i \hs \nabla \vro \cdot \hs \p_3^2 u \ dx  = \i \hs \p_3 \vro \cdot \hs \p_3^2 u_3 \ dx +   \i \hs \nabla_h \vro \cdot \hs \p_3^2 u_h \ dx  \\
               = &\; \i \hs \p_3 \vro \cdot \hs \p_3^2 u_3 \ dx +   \i \hs \p_3 \vro \cdot \hs \p_3 \nabla_h \cdot u_h \ dx  \\
               = &\; \i \hs \p_3 \vro \, \hs \p_3 \du \ dx.
                \dal \deqq
				The H\"{o}lder's inequality and  inequality \eqref{a16} in Lemma \ref{H-L} yield directly
				\begin{equation}\label{K03}
					\begin{aligned}
						K_1 = &\int_{\mathbb{R}^3} \hs( \p_3 u \cdot \nabla  \vro+ u\cdot \nabla \p_3 \vro ) \cdot \hs \p_3 \vro \ dx \\
						\lesssim & \int_{\mathbb{R}_+}
						(\|\p_3 u_h \cdot \nabla_h \vro \|_{L^{\frac{1}{\frac12+\frac{s}{2}}}(\mathbb{R}^2)} +\|\p_3 u_3 \partial_3 \vro \|_{L^{\frac{1}{\frac12+\frac{s}{2}}}(\mathbb{R}^2)}
						+ \|u_h \cdot \nabla_h \p_3 \vro \|_{L^{\frac{1}{\frac12+\frac{s}{2}}}(\mathbb{R}^2)}
						+\|u_3 \partial_3^2 \vro \|_{L^{\frac{1}{\frac12+\frac{s}{2}}}(\mathbb{R}^2)})\\
						&\; \times
						\|\hs \p_3 \vro \|_{L^2(\mathbb{R}^2)}dx_3   .
					\end{aligned}
				\end{equation}
				Similar to the estimate of term $J_1$, it is easy to check that
				\begin{equation}\label{K04}
					\begin{aligned}
						&~~~~\int_{\mathbb{R}_+}
						(\|\p_3 u_h \cdot \nabla_h \vro \|_{L^{\frac{1}{\frac12+\frac{s}{2}}}(\mathbb{R}^2)} +\|\p_3 u_3 \partial_3 \vro \|_{L^{\frac{1}{\frac12+\frac{s}{2}}}(\mathbb{R}^2)}
						+ \|u_h \cdot \nabla_h \p_3 \vro \|_{L^{\frac{1}{\frac12+\frac{s}{2}}}(\mathbb{R}^2)})
						\|\hs \p_3 \vro\|_{L^2(\mathbb{R}^2)}dx_3\\
						&\lesssim
						\left(
						\left\|\|\nabla_h \vro\|_{L^\infty(\mathbb{R}_+)}\right\|_{L^{\frac{2}{s}}}
						\|\partial_3 u_h\|_{L^2}+\left\|\|\p_3 u_3\|_{L^\infty(\mathbb{R}_+)}\right\|_{L^{\frac{2}{s}}}
						\|\partial_3 \vro \|_{L^2}+\left\|\|u_h\|_{L^\infty(\mathbb{R}_+)}\right\|_{L^{\frac{2}{s}}}
						\|\nabla_h \p_3 \vro \|_{L^2}
						\right)
						\|\hs \p_3 \vro\|_{L^2}\\
						&\lesssim \sqrt{\mathcal{E}_{tan}(t)} \sqrt{\md_{tan}(t)}
						\|\hs \p_3 \vro\|_{L^2}.
					\end{aligned}
				\end{equation}
				Now, let us deal with the difficult term $\int_{\mathbb{R}_+}
				\|u_3 \partial_3^2 \vro \|_{L^{\frac{1}{\frac12+\frac{s}{2}}}(\mathbb{R}^2)}
				\|\hs \p_3 \vro \|_{L^2(\mathbb{R}^2)}dx_3$.
				Indeed, we may check that
				\begin{equation*}
					\begin{aligned}
						u_3 \p_3^2 \vro
						&=u_3 \p_3^2 \vro \, \chi(x_3)+u_3 \p_3^2 \vro \, (1-\chi(x_3))\\
						&=\int_0^{x_3} \partial_3 u_3 \, dz \, \p_3^2 \vro \, \chi(x_3)
						+\varphi^{-1} u_3 \,\varphi \p_3^2 \vro \,  (1-\chi(x_3)),
					\end{aligned}
				\end{equation*}
				which yields directly
				\begin{equation*}
					|u_3 \p_3^2 \vro|
					\lesssim \|\partial_3 u_3\|_{L^\infty(\mathbb{R}_+)}|Z_3\p_3 \vro|
					+|u_3||Z_3 \p_3 \vro|.
				\end{equation*}
                Then, this decomposition yields directly
				\begin{equation}\label{K05}
					\begin{aligned}
						&~~~~\int_{\mathbb{R}_+}
						\|u_3 \partial_3^2 \vro\|_{L^{\frac{1}{\frac12+\frac{s}{2}}}(\mathbb{R}^2)}
						\|\hs \p_3 \vro \|_{L^2(\mathbb{R}^2)}dx_3\\
						&\lesssim\int_0^{+\infty}
						\left(\left\|\|\partial_3 u_3\|_{L^\infty(\mathbb{R}_+)}\right\|_{L^{\frac2s}(\mathbb{R}^2)}
						+\left\|\|u_3\|_{L^\infty(\mathbb{R}_+)}\right\|_{L^{\frac2s}(\mathbb{R}^2)}\right)
						\|Z_3 \p_3 \vro\|_{L^2(\mathbb{R}^2)}
						\|\Lambda_h^{-s} \p_3 \vro\|_{L^2(\mathbb{R}^2)}dx_3\\
						&\lesssim\left(\left\|\|\partial_3 u_3\|_{L^\infty(\mathbb{R}_+)}\right\|_{L^{\frac2s}(\mathbb{R}^2)}
						+\left\|\|u_3\|_{L^\infty(\mathbb{R}_+)}\right\|_{L^{\frac2s}(\mathbb{R}^2)}\right)
						\|Z_3 \p_3 \vro\|_{L^2}\|\Lambda_h^{-s} \p_3 \vro\|_{L^2}\\
						&\lesssim\left(\left\|\|\partial_3 u_3\|_{L^\infty(\mathbb{R}_+)}\right\|_{L^{\frac2s}(\mathbb{R}^2)}
						+\left\|\|u_3\|_{L^\infty(\mathbb{R}_+)}\right\|_{L^{\frac2s}(\mathbb{R}^2)}\right)
						\|\Lambda_h^{-s} \p_3 \vro\|_{L^2}\\
						&\quad \quad \times \left(\|\p_3 \vro\|_{L^2}
						+\|\p_3 \vro\|_{L^2}^{\frac34}\|Z_3^3 \p_3 \vro\|_{L^2}^{\frac14}
						+\|\p_3 \vro\|_{L^2}^{\frac23}\|Z_3^3 \p_3 \vro\|_{L^2}^{\frac13} \right)\\
						&\lesssim \|(\p_3 \vro, Z_3^3 \p_3 \vro )\|_{L^2}^{\frac13}
						\mathcal{E}_{tan}(t)^{\frac{7}{12}}
						\md_{tan}(t)^{\frac14}
						\|\hs \p_3 \vro\|_{L^2}.
					\end{aligned}
				\end{equation}
				Substituting estimates \eqref{K04} and \eqref{K05}  into \eqref{K03}, we have
				\beqq
				\begin{aligned}
					K_1
					\lesssim&\;
					\sqrt{\mathcal{E}_{tan}(t)} \sqrt{\md_{tan}(t)}
					\|\hs \p_3 \vro\|_{L^2} +\|(\p_3 \vro, Z_3^3 \p_3 \vro )\|_{L^2}^{\frac13}
					\mathcal{E}_{tan}(t)^{\frac{7}{12}}
					\md_{tan}(t)^{\frac14}
					\|\hs  \p_3 \vro\|_{L^2}.
				\end{aligned}
				\deqq
				Similarly, it is easy to check that
				\beqq
				\begin{aligned}
					K_2
					\lesssim
					&\sqrt{\mathcal{E}_{tan}(t)} \sqrt{\md_{tan}(t)}
					\|\hs \p_3 \vro\|_{L^2},\\
                    \end{aligned}
                    \deqq
                    and integrating by parts, we have
				\begin{align*}
					K_3 = &\; \i \hs( u \cdot \nabla u_3  ) \cdot \hs \p_3^2 u_3 \ dx +  \i \hs( u \cdot \nabla u_h  ) \cdot \hs \p_3^2 u_h \ dx \\
                    = &\;- \i \hs \p_3 ( u \cdot \nabla u_3  ) \cdot \hs \p_3 u_3 \ dx -  \i \hs \p_3 ( u \cdot \nabla u_h  ) \cdot \hs \p_3 u_h \ dx \\
					\lesssim &\;
					\sqrt{\mathcal{E}_{tan}(t)} \sqrt{\md_{tan}(t)}
					\|\hs \p_3 u\|_{L^2} +\|(\p_3 u, Z_3^3 \p_3 u)\|_{L^2}^{\frac13}
					\mathcal{E}_{tan}(t)^{\frac{7}{12}}
					\md_{tan}(t)^{\frac14}
					\|\hs  \p_3 u\|_{L^2},\\
					K_4 = &\; \f12 \i \hs \p_3 (|\vro|^2 ) \cdot \hs \p_3^2 u_3 \ dx  + \f12 \i \hs \nabla_h (|\vro|^2 ) \cdot \hs \p_3^2 u_h \ dx \\
                    = &\; \f12 \i \hs \p_3 (|\vro|^2 ) \cdot \hs \p_3^2 u_3 \ dx  + \f12 \i \hs \p_3 (|\vro|^2 ) \cdot \hs \p_3 \nabla_h \cdot u_h \ dx \\
                     = &\; \i \hs(\vro \, \p_3 \vro) \cdot \hs \p_3 \du \ dx \\
					\le
					&\; \nu \|\hs \p_3 \du\|_{L^2}^2 +  C_{\nu} \mathcal{E}_{tan}(t)^2.
				\end{align*}
                Integrating by part, we have
                \begin{align*}
					K_5  = &\; \i \hs \Big(\f{\vro}{1+\vro} (\Delta_h u_h + \nabla_h \du)  \Big) \cdot \hs \p_3^2 u_h \ dx  +\var \i \hs \Big(\f{\vro}{1+\vro} \p_3^2 u_h) \cdot \hs \p_3^2 u_h \ dx\\
                    &\; + \i \hs \Big(\f{\vro}{1+\vro} (\Delta_h u_3 +\var \p_3^2 u_3 + \p_3 \du)  \Big) \cdot \hs \p_3^2 u_3 \ dx\\
                    = &\; -\i \hs \p_3\Big(\f{\vro}{1+\vro} (\Delta_h u_h + \nabla_h \du)  \Big) \cdot \hs \p_3 u_h \ dx +\var \i \hs \Big(\f{\vro}{1+\vro} \p_3^2 u_h) \cdot \hs \p_3^2 u_h \ dx \\
                    &\; + \i \hs \Big(\f{\vro}{1+\vro} (\Delta_h u_3 +\var \p_3^2 u_3 + \p_3 \du)  \Big) \cdot \hs \p_3^2 u_3 \ dx.
                    \end{align*}
				Using the H\"{o}lder's inequality and inequality \eqref{a16} in Lemma \ref{H-L} yield directly
				\begin{align*}
					K_5  
					\lesssim &\;
                    \| \hs \p_3 u_h\|_{L^2} \Big( \left(\| \vro\|_{L^2}\|\partial_{1}  \vro\|_{L^2}
						+\|\p_{2} \vro\|_{L^2}\| \p_{12} \vro\|_{L^2}\right)^{\frac{1-s}{2}}
						\|\vro\|_{L^2}^{\frac{2s-1}{2}}
						\| \partial_3   \vro\|_{L^2}^{\frac12}
						\| \p_3 (\Delta_h u_h,\nabla_h \du)\|_{L^2} \\
                       &\;  +\| \p_3 \vro\|_{L^2} \big( \left(\| \Delta_h u_h\|_{L^2}\|\partial_{1}  \Delta_h u_h\|_{L^2}
						+\|\p_{2} \Delta_h u_h\|_{L^2}\| \p_{12} \Delta_h u_h\|_{L^2}\right)^{\frac{1-s}{2}}
						\|\Delta_h u_h\|_{L^2}^{\frac{2s-1}{2}}
						\| \partial_3 \Delta_h u_h\|_{L^2}^{\frac12}\\
                        &\;  +\left(\| \nabla_h \du\|_{L^2}\|\partial_{1}  \nabla_h \du\|_{L^2}
						+\|\p_{2} \nabla_h \du\|_{L^2}\| \p_{12} \nabla_h \du\|_{L^2}\right)^{\frac{1-s}{2}}
						\|\nabla_h \du\|_{L^2}^{\frac{2s-1}{2}}
						\| \partial_3 \nabla_h \du\|_{L^2}^{\frac12}
						\big) \Big)\\
						 &\; + \left(\| \vro\|_{L^2}\|\partial_{1}  \vro\|_{L^2}
						+\|\p_{2} \vro\|_{L^2}\| \p_{12} \vro\|_{L^2}\right)^{\frac{1-s}{2}}
						\|\vro\|_{L^2}^{\frac{2s-1}{2}}
						\| \partial_3   \vro\|_{L^2}^{\frac12}\\
                        &\;\times
						\Big (
                        \| (\Delta_h u_3, \var \p_3^2 u_3, \p_3 \du)\|_{L^2}  \|(\hs \p_3 \du, \hs \p_3 \nabla_h u)\|_{L^2} + \var \|\p_3^2 u_h\|_{L^2} \| \hs \p_3^2 u_h\|_{L^2} \Big)\\
                    \le &\; \nu (\|(\hs \p_3 \du, \hs \p_3 \nabla_h u)\|_{L^2}^2 + \var \|\hs \p_3^2 u\|_{L^2}^2 ) + C_{\nu} \me_{tan}(t)  \md_{tan}(t) + \sqrt{\mathcal{E}_{tan}(t)} \sqrt{\md_{tan}(t)}
					\|\hs \p_3 u\|_{L^2}.
				\end{align*}
				Substituting the estimates of terms for  $K_1$
				through $K_5$ into \eqref{K02} and using the smallness of $\nu$, we have
				\begin{align}\label{K06}
					&\frac{1}{2} \|\hs \p_3(\vro, u)(t)\|_{L^2}^2
					+\int_0^t (\|\hs \p_3(\nabla_h u, \du)\|_{L^2}^2 + \var \| \hs \p_3^2 u\|_{L^2}^2 ) d\tau \notag\\
					\lesssim
					&\frac{1}{2}\|\hs \p_3 (\vro, u)(0)\|_{L^2}^2
					+\underset{0\le \tau \le t}{\sup}\|\hs \p_3 (\vro, u)(\tau)\|_{L^2}
					\int_0^t \mathcal{E}_{tan}(\tau)^{\f12}
					\md_{tan} (\tau)^{\f12}
					d\tau \notag\\
					&+\underset{0\le \tau \le t}{\sup}
					\|\hs \p_3 (\vro, u)(\tau)\|_{L^2}
					\underset{0\le \tau \le t}{\sup}
					\|(\p_3 \vro, \p_3 u, Z_3^3 \p_3 \vro,  Z_3^3 \p_3 u)(\tau)\|_{L^2}^{\frac13}
					\int_0^t  \mathcal{E}_{tan}(\tau)^{\frac{7}{12}}
					\md_{tan}(\tau)^{\frac14} d \tau \notag\\
			         &\; + \int_0^t (\mathcal{E}_{tan}(\tau)^{2} +\mathcal{E}_{tan}(\tau) \md_{tan}(\tau) ) d\tau.
				\end{align}
				Using the assumption \eqref{assumption}, it is easy to check that
				\begin{equation}\label{K07}
					\begin{aligned}
						\int_0^t \mathcal{E}_{tan}(\tau)^{\frac{7}{12}}
						\md_{tan}(\tau)^{\frac14} d\tau
						&\lesssim
						\bigg\{\int_0^t \md_{tan}(\tau)
						(1+\tau)^{\sigma}d\tau\bigg\}^{\frac{1}{4}}
						\bigg\{\int_0^t \mathcal{E}_{tan}(\tau)^{\frac{7}{9}}(1+\tau)^{-\frac13 \sigma}
						d\tau\bigg\}^{\frac{3}{4}}\\
						&\lesssim
						\delta^{\frac56}\bigg\{\int_0^t (1+\tau)^{-(\frac{7}{9}s+\frac13\sigma)}d\tau
						\bigg\}^{\frac{1}{2}}
						\lesssim \delta^{\frac56},
					\end{aligned}
					\end{equation}
                    and
                    \beqq \bal
                    \int_0^t (\mathcal{E}_{tan}(\tau)^{2} +\mathcal{E}_{tan}(\tau) \md_{tan}(\tau) ) d\tau  \lesssim \delta^2,
                    \dal \deqq
				where we have used the condition $\frac{16}{17}\le\sigma<s<1$.
				Substituting the above estimate and \eqref{J03} into \eqref{K06}, then we have
				\beqq
				\begin{aligned}
					&\|(\hs\wr, \hs \wu)(t)\|_{L^2}^2
					+\int_0^t \|(\nabla_h \hs \wu, \hs \p_3 \du)\|_{L^2}^2 d\tau + \var \int_0^t \| \hs \p_3^2 u \|_{L^2}^2 d\tau\\
					\lesssim
					&\|(\hs\wr, \hs \wu)(0)\|_{L^2}^2+\delta^2.
				\end{aligned}
				\deqq
				Therefore, we complete the proof of this lemma.
			\end{proof}

			\subsection{Decay in time estimate}\label{dect}
			In this subsection, we will establish the decay rate estimate
			for density and velocity field. This decay in time estimate
			will help us to close the energy estimate.
			First of all, we establish the uniform estimate for $\me(t)$.
            \begin{lemm}\label{lemma313}
				Under the assumption \eqref{assumption},
				the smooth solution $(\vro,u)$ of equation \eqref{eqr} has the estimate
				\beq\label{31001}
				\begin{aligned}
					\me(t) + \int_{0}^{t} \md(\tau) d \tau
				\lesssim C_0,
				\end{aligned}
				\deq
				where the constant $C_0$ is defined in \eqref{co}.
			\end{lemm}
            \begin{proof}
                Combining the estimates $\eqref{3101-1}$ and \eqref{3301},  we can apply the induction and the equivalence relation \eqref{est-equivalence} 
                to obtain that
				\beqq \bal
				&\; \| (\vro, u)(t)\|_{H_{co}^m}^2 + \| \p_3 (\vro, u)(t)\|_{H_{co}^{m-1}}^2 +\int_0^t \|(\nabla_h u, \du)(\tau)\|_{H_{co}^m}^2  d\tau + \var \int_0^t \|\p_3 u(\tau)\|_{H_{co}^m}^2 d \tau\\
                &\; +\int_0^t \|\p_3 (\nabla_h u, \du)(\tau)\|_{H_{co}^{m-1}}^2  d\tau + \var \int_0^t \|\p_3^2 u(\tau)\|_{H_{co}^{m-1}}^2  d \tau 
				\\
                \lesssim  &\;  \| (\vro, u)(0)\|_{H_{co}^m}^2 +\|\p_3 (\vro, u)(0)\|_{H_{co}^{m-1}}^2 +\delta^{\frac43}.
				\dal \deqq
                Multiplying the estimate \eqref{3501} by the small constant $\kappa_1$ and integrating over $[0,t]$, together with the above inequality, we can obtain
                \begin{align*}
                &\; \| (\vro, u)(t)\|_{H_{co}^m}^2 + \| \p_3 (\vro, u)(t)\|_{H_{co}^{m-1}}^2 + \kappa_1 \sum_{|\al| \le m-1}  \i Z^{\al} u \cdot Z^{\al} \nabla  \vro \ dx
                +\int_0^t \|(\nabla_h u, \du)(\tau)\|_{H_{co}^m}^2  d\tau \\  
                &\;+ \var \int_0^t \|\p_3 u(\tau)\|_{H_{co}^m}^2 d \tau 
				 +\int_0^t \|\p_3 (\nabla_h u, \du)(\tau)\|_{H_{co}^{m-1}}^2  d\tau + \var \int_0^t \|\p_3^2 u(\tau)\|_{H_{co}^{m-1}}^2  d \tau + \kappa_1 \int_0^t \|\nabla \vro(\tau)\|_{H_{co}^{m-1}}^2  d \tau\\
                \lesssim  &\; \kappa_1 \int_{0}^{t} (\|(\du, \nabla \du)\|_{H^{m-1}_{co}}^2+ \| (\nabla_h u, \Delta_h u)\|_{H_{co}^{m-1}}^2 +
					\var \|\p_3^2 u \|_{H^{m-1}_{co}}^2 ) d \tau + \kappa_1 \sum_{|\al| \le m-1}  \i Z^{\al} u_0 \cdot Z^{\al} \nabla  \vro_0 \ dx \\
                    &\; + \| (\vro, u)(0)\|_{H_{co}^m}^2 +\|\p_3 (\vro, u)(0)\|_{H_{co}^{m-1}}^2 +\delta^{\frac43},
				\end{align*}
                which, using the smallness of $\kappa_1$, yields that
                \begin{align*}
                &\; \| (\vro, u)(t)\|_{H_{co}^m}^2 + \| \p_3 (\vro, u)(t)\|_{H_{co}^{m-1}}^2 +\int_0^t \|(\nabla_h u, \du)(\tau)\|_{H_{co}^m}^2  d\tau + \var \int_0^t \|\p_3 u(\tau)\|_{H_{co}^m}^2 d \tau 
				\\  
                &\; +\int_0^t \|\p_3 (\nabla_h u, \du)(\tau)\|_{H_{co}^{m-1}}^2  d\tau + \var \int_0^t \|\p_3^2 u(\tau)\|_{H_{co}^{m-1}}^2  d \tau + \int_0^t \|\nabla \vro(\tau)\|_{H_{co}^{m-1}}^2  d \tau\\
                \le &\; C (\| (\vro, u)(0)\|_{H_{co}^m}^2 +\|\p_3 (\vro, u)(0)\|_{H_{co}^{m-1}}^2 )+C\delta^{\frac43}.
				\end{align*}
                 In a similar way, combining the above estimate and \eqref{3901}, \eqref{3601}, \eqref{3701} and \eqref{3801}, we can deduce that
                 \beq\label{31002} \bal
&\; \| (\vro, u)(t)\|_{H_{co}^m}^2 + \| \p_3 (\vro, u)(t)\|_{H_{co}^{m-1}}^2 + \| \p_t  (u, w_h,  \p_3 u_3)(t)\|_{H_{tan}^2}^2 \\
                     &\;
					+  \| Z_3 \p_t ( w_h, \p_3 u_3 ) (t)\|_{L^2}^2 + \var \| \p_3^2 \vro(t) \|_{H_{co}^2}^2 +\int_0^t \md(\tau) d\tau \\
				\le &\; C(\| (\vro, u) (0)\|_{H_{co}^m}^2 + \| \p_3 (\vro, u)(0)\|_{H_{co}^{m-1}}^2 + \| \p_t  (u, w_h,  \p_3 u_3)(0)\|_{H_{tan}^2}^2
					 \\
                     &\;  \quad +  \| Z_3 \p_t ( w_h, \p_3 u_3 ) (0)\|_{L^2}^2+ \var \| \p_3^2 \vro(0) \|_{H_{co}^2}^2)+C\delta^{\frac43}.
                 \dal \deq
                Moreover, 
                the combination of estimates \eqref{wh01} and \eqref{vro01} yields
                \beqq
				\begin{aligned}
					\| w_h\|_{L^{\infty}}^{3} + \| \p_3 \vro\|_{L^{\infty}}^{3}
					\le
					C( \| w_h(0)\|_{L^{\infty}}^{3}+ \| \p_3 \vro(0) \|_{L^{\infty}}^{3}+\delta^{\f{3}{2}}),
				\end{aligned}
				\deqq
                and the combination of estimates \eqref{J01} and \eqref{K01} yields directly
				\beqq
				\begin{aligned}
					\|(\hs(\vro,u), \hs \p_3 (\vro,u))(t)\|_{L^2}^2
					\le
					C(\|(\hs(\vro,u), \hs \p_3 (\vro,u))(0)\|_{L^2}^2+\delta^2),
				\end{aligned}
				\deqq
				which, together with estimate \eqref{31002}, yields directly
				\beq\label{31003}
				\me(t) + \int_{0}^{t} \md(\tau) d \tau
				\le CC_0,
				\deq
				where the constant $C_0$ is defined by
				\beq\label{co}
				C_0:=\me(0) +\delta^{\f43}.
				\deq
                Thus, we finish the proof of this lemma.
            \end{proof}
            Next, we  establish the decay rate estimate for $\me_{tan}(t)$ and $\md_{tan}(t)$.

			\begin{lemm}\label{lemma314}
				Under the assumption \eqref{assumption},
				the smooth solution $(\vro,u)$ of equation \eqref{eqr} has the estimate
				\beq\label{31101}
				(1+t)^s \mathcal{E}_{tan}(t) + \int_0^t (1+\tau)^{\sigma} \md_{tan}(\tau)d\tau \lesssim C_0,
				\deq
				where the constant $C_0$ is defined in \eqref{co}.
			\end{lemm}
			\begin{proof}
				From estimates $\eqref{3101-1}$ and $\eqref{3201-1}$, we have
				\beq\label{31102} \bal
				&\; \frac{d}{dt} (\|(\vro,u)(t)\|_{L^2}^2+\|(\vro,u)(t)\|_{\dot{H}_{tan}^{m}}^2)
				+\|(\nabla u,\du)\|_{H^{m}_{tan}}^2 + \var \| \p_3 u \|_{H^{m}_{tan}}^2  
				\lesssim  \sqrt{\me(t)}\md_{tan}(t),
				\dal \deq
				and
				\beq\label{31103} \bal
				&\; \frac{d}{dt}(\|\p_3(\vro,u)(t)\|_{L^2}^2+\|\p_3 (\vro,u)(t)\|_{\dot{H}_{tan}^{m-1}}^2) 
				+\|\p_3(\nabla u,\du)\|_{H^{m-1}_{tan}}^2 + \var \| \p_3^2 u \|_{H^{m-1}_{tan}}^2 
				\lesssim (\me(t)^{\f13}+\me(t)^{\f12})\md_{tan}(t).
				\dal \deq
				From the estimates \eqref{3501} and \eqref{3701}, we have
				\beq\label{31104}
				\begin{aligned}
					&\frac{d}{dt}
					\i ( u \cdot \nabla \vro +  \sum_{|\ah|=m-1} Z^{\ah} u \cdot \nabla Z^{\ah} \vro ) \ dx
					+\|\nabla \vro\|_{H^{m-1}_{tan}}^2  \\
					\lesssim
					&\|(\du, \nabla \du, \Delta_h u)\|_{H^{m-1}_{tan}}^2+
					\var \|\p_3^2 u \|_{H^{m-1}_{tan}}^2 
					+\sqrt{\me(t)}\md_{tan}(t).
				\end{aligned}
				\deq
				Then, for some small positive suitable constant $\kappa_2>0$,
				we can deduce from estimates \eqref{31102}-\eqref{31104} that
				\beq\label{31105}
				\begin{aligned}
					&\frac{d}{dt}\widetilde{\mathcal{E}}_{tan}(t)
					+2\kappa_2 \md_{tan}(t)
					\lesssim (\me(t)^{\f13}+\me(t)^{\f12})\md_{tan}(t),
				\end{aligned}
				\deq
				where the quantity $\widetilde{\mathcal{E}}_{tan}(t)$ is defined as
				\beqq
				\begin{aligned}
					\widetilde{\mathcal{E}}_{tan}(t)
					:=&\|(\vro,u)(t)\|_{L^2}^2 + \|(\vro,u)(t)\|_{\dot{H}_{tan}^{m}}^2 +\|\p_3(\vro,u)(t)\|_{L^2}^2+\|\p_3 (\vro,u)(t)\|_{\dot{H}_{tan}^{m-1}}^2 \\
                    &+ 2\kappa_2 (\i u \cdot \nabla \vro \ dx +  \sum_{|\ah|=m-1} \i Z^{\ah} u \cdot \nabla Z^{\ah} \vro \ dx).
				\end{aligned}
				\deqq
				Due to the smallness of $\kappa_2$, it is easy to check that
				$\widetilde{\mathcal{E}}_{tan}(t)$ is equivalent to $\mathcal{E}_{tan}(t)$.
				Thus, due to the a priori assumption \eqref{assumption}, we have
				\beq\label{31106}
				\frac{d}{dt}\widetilde{\mathcal{E}}_{tan}(t)
				+\kappa_2 \md_{tan}(t)
				\le 0.
				\deq
				On the other hand, due to the inequality
				\beqq
				\|(\vro, u, \wr, \wu)\|_{L^2}
				\lesssim \|\hs(\vro, u, \wr, \wu)\|_{L^2}^{\frac{1}{1+s}}
				\|\nabla_h(\vro, u, \wr, \wu)\|_{L^2}^{\frac{s}{1+s}},
				\deqq
				it is easy to check that
				\beq\label{31107}
				\begin{aligned}
					\widetilde{\mathcal{E}}_{tan}(t)
					\lesssim
					\me_{tan}(t)
					\lesssim
					&(\|\hs(\vro, u, \wr, \wu)\|_{L^2}^2
					+\|\nabla_h(\vro,u)\|_{H^{m-2}_{tan}}^2
					+\|\nabla_h(\wr, \wu)\|_{H^{m-3}_{tan}}^2)^{\frac{1}{1+s}}\\
					& \times(\|\nabla_h(\vro, u)\|_{H^{m-2}_{tan}}^2
					+\|\nabla_h(\wr, \wu)\|_{H^{m-3}_{tan}}^2)^{\frac{s}{1+s}}\\
					\lesssim
					&C_0^{\frac{1}{1+s}} \md_{tan}(t)^{\frac{s}{1+s}},
				\end{aligned}
				\deq
				where we have used the estimate \eqref{31003} in the last inequality.
				The combination of \eqref{31106} and \eqref{31107} yields
				\beqq
				\frac{d}{dt}\widetilde{\mathcal{E}}_{tan}(t)
				+\kappa_2 C_0^{-\frac{1}{s}}
				\widetilde{\mathcal{E}}_{tan}(t)^{1+\frac{1}{s}}\le 0,
				\deqq
				which yields directly the decay estimate
				\beq\label{31108}
				\mathcal{E}_{tan}(t)
				\lesssim
				\widetilde{\mathcal{E}}_{tan}(t)
				\lesssim C_0(1+t)^{-s}.
				\deq
			Finally, we will establish the time integration
			of  $\md_{tan}(t)$ with the suitable weight $(1+\tau)^{\sigma}$.
				For $0<\sigma<s$, multiplying \eqref{31106} by $(1+t)^{\sigma}$, we have
				\beqq
				\begin{aligned}
					&\frac{d}{dt}[(1+t)^{\sigma}\widetilde{\mathcal{E}}_{tan}(t)]
					+\kappa_2 (1+t)^{\sigma} \md_{tan}(t)
					\le \sigma(1+t)^{\sigma-1}\widetilde{\mathcal{E}}_{tan}(t).
				\end{aligned}
				\deqq
				Integrating the above inequality over $[0, t]$
				and using the uniform estimate \eqref{31105}, we have
				\beq\label{est-sigma-s}
				\begin{aligned}
					&\;(1+t)^{\sigma}\widetilde{\mathcal{E}}_{tan}(t)
					+\kappa_2 \int_0^t  (1+\tau)^{\sigma} \md_{tan}(\tau)d\tau\\
					\le
					&\;\widetilde{\mathcal{E}}_{tan}(0)
					+\sigma\int_0^t (1+\tau)^{\sigma-1}
					\widetilde{\mathcal{E}}_{tan}(\tau)d\tau\\
					\le
					&\;\widetilde{\mathcal{E}}_{tan}(0)
					+\sigma\underset{0\le \tau \le t}{\sup}
					\left[\widetilde{\mathcal{E}}_{tan}(\tau)(1+\tau)^s\right]
					\int_0^t (1+\tau)^{\sigma-1-s}d\tau\\
					\le
					&\;\widetilde{\mathcal{E}}_{tan}(0)
					+\frac{\sigma}{s-\sigma}
					\left[1-(1+t)^{\sigma-s}\right]
					\underset{0\le \tau \le t}{\sup}
					\left[\widetilde{\mathcal{E}}_{tan}(\tau)(1+\tau)^s\right]\\
					\lesssim
					&\;C_0.
				\end{aligned}
				\deq
				Combining the above estimate and \eqref{31108}, we complete the proof of this lemma.
			\end{proof}
            Finally, we establish the decay rate estimate for $\bme_{tan}(t)$ and $\bmd_{tan}(t)$.

			\begin{lemm}\label{lemma315}
				Under the assumption \eqref{assumption},
				the smooth solution $(\vro,u)$ of equation \eqref{eqr} has the estimate
				\beq\label{31201}
				(1+t)^{1+\sigma} \bme_{tan}(t) + \int_0^t (1+\tau)^{1+\sigma} \bmd_{tan}(\tau)d\tau \lesssim C_0,
				\deq
				where the constant $C_0$ is defined in \eqref{co}.
			\end{lemm}
            \begin{proof}
                From estimates $\eqref{3101-2}$ and $\eqref{3201-2}$, we have
				\beq\label{31202} \bal
				&\; \frac{d}{dt}(\|\nabla_h (\vro,u)(t)\|_{L^2}^2  + \|\nabla_h (\vro,u)(t)\|_{\dot{H}_{tan}^{m-2}}^2 )
				+\|\nabla_h(\nabla u,\du)\|_{H^{m-2}_{tan}}^2 \\
                &\; + \var \| \nabla_h \p_3 u \|_{H^{m-2}_{tan}}^2  
				\lesssim  \sqrt{\me(t)}\bmd_{tan}(t),
				\dal \deq
				and
				\beq\label{31203} \bal
				&\; \frac{d}{dt}(\|\nabla_h \p_3 (\vro,u)(t)\|_{L^2}^2 + \|\nabla_h \p_3 (\vro,u)(t)\|_{\dot{H}_{tan}^{m-3}}^2)
				+\|\nabla_h \p_3(\nabla u,\du)\|_{H^{m-3}_{tan}}^2 \\
                &\; +\var \| \nabla_h \p_3^2 u \|_{H^{m-3}_{tan}}^2   
				\lesssim (\me(t)^{\f13}+\me(t)^{\f12})\bmd_{tan}(t).
				\dal \deq
				From the estimates \eqref{3401}, we have
				\beq\label{31204}
				\begin{aligned}
					&\frac{d}{dt}
					\i ( \nabla_h u \cdot \nabla \nabla_h \vro +  \sum_{|\ah|=m-2} Z^{\ah} u \cdot \nabla Z^{\ah} \vro ) \ dx
					+\|\nabla_h \nabla \vro\|_{H^{m-3}_{tan}}^2  \\
					\lesssim
					&\|\nabla_h (\du, \nabla \du, \Delta_h u)\|_{H^{m-3}_{tan}}^2+
					\var \|\nabla_h \p_3^2 u \|_{H^{m-3}_{tan}}^2 
					+\sqrt{\me(t)}\bmd_{tan}(t).
				\end{aligned}
				\deq
				Then, for some small positive suitable constant $\kappa_3>0$,
				we can deduce from estimates \eqref{31202}-\eqref{31204} that
				\beqq
				\begin{aligned}
					&\frac{d}{dt}\widetilde{\bme}_{tan}(t)
					+2\kappa_3 \bmd_{tan}(t)
					\lesssim (\me(t)^{\f13}+\me(t)^{\f12})\bmd_{tan}(t),
				\end{aligned}
				\deqq
				where the quantity $\widetilde{\bme}_{tan}(t)$ is defined as
				\beqq
				\begin{aligned}
					\widetilde{\bme}_{tan}(t)
					:=&\|\nabla_h (\vro,u)(t)\|_{L^2}^2 + \|\nabla_h (\vro,u)(t)\|_{\dot{H}_{tan}^{m-2}}^2 +\|\nabla_h \p_3 (\vro,u)(t)\|_{L^2}^2 +\|\nabla_h \p_3 (\vro,u)(t)\|_{\dot{H}_{tan}^{m-3}}^2 \\
					&+ 2\kappa_3 (\i \nabla_h u \cdot \nabla \nabla_h \vro \ dx +  \sum_{|\ah|=m-2} \i Z^{\ah} u \cdot \nabla Z^{\ah} \vro \ dx).
				\end{aligned}
				\deqq
				Due to the smallness of $\kappa_3$, it is easy to check that
				$\widetilde{\bme}_{tan}(t)$ is equivalent to $\bme_{tan}(t)$.
				Thus, due to the a priori assumption \eqref{assumption}, we have
				\beqq
				\frac{d}{dt}\widetilde{\bme}_{tan}(t)
				+\kappa_3 \bmd_{tan}(t)
				\le 0.
				\deqq
				Multiplying the above inequality by $(1+t)^{1+\sigma}$, we have
				\beqq \bal
				\frac{d}{dt}[(1+t)^{1+\sigma}\widetilde{\bme}_{tan}(t)]
				+\kappa_3 (1+t)^{1+\sigma}\bmd_{tan}(t)\le (1+\sigma)(1+t)^\sigma \widetilde{\bme}_{tan}(t).
				\dal \deqq
				Integrating the above inequality over $[0,t]$ and using the decay estimate \eqref{31101}, we have
				\beqq
				\begin{aligned}
					&\;(1+t)^{1+\sigma}\widetilde{\bme}_{tan}(t)
					+\kappa_3 \int_0^t  (1+\tau)^{1+\sigma} \bmd_{tan}(\tau)d\tau\\
					\le
					&\;\widetilde{\bme}_{tan}(0)
					+(1+\sigma)\int_0^t 
					\widetilde{\bme}_{tan}(\tau)(1+\tau)^{\sigma}d\tau\\
					\le
					&\;\widetilde{\bme}_{tan}(0)
					+(1+\sigma)
					\int_0^t \md_{tan}(\tau)(1+\tau)^{\sigma}d\tau\\
					\lesssim
					&\;C_0,
				\end{aligned}
				\deqq
                here we can check that $\widetilde{\bme}_{tan}(t) \lesssim \md_{tan}(t)$. Thus, we finish the proof of this lemma.
                            \end{proof}

			\subsection{Global in time uniform regularity} \label{lastt}
			
			In this subsection, we will give the proof of Proposition \ref{main_pro}.
			Indeed,  the combination of estimates \eqref{31001}, \eqref{31101}  and \eqref{31201} 
			yields directly
            \beqq
			\underset{0\le \tau \le t}{\sup}\me(\tau) + \int_{0}^{t} \md(\tau) d \tau 
			\le C\me(0)+C\delta^{\frac43},
			\deqq
            and
            \beqq \bal
			&\; \underset{0\le \tau \le t}{\sup}[(1+\tau)^{s}{\mathcal{E}}_{tan}(\tau)]
			+\int_0^t (1+\tau)^{\sigma} \md_{tan}(\tau)d\tau
			\le C(\me(0)+\delta^{\f43}),\\
            &\; \underset{0\le \tau \le t}{\sup}[(1+\tau)^{1+\sigma}{\bme}_{tan}(\tau)]
			+\int_0^t (1+\tau)^{1+\sigma} \bmd_{tan}(\tau)d\tau
			\le C(\me(0)+\delta^{\f43}),
			\dal \deqq
			where the constants $(\sigma, s)$ satisfy $\frac{16}{17}\le\sigma<s<1$.
			Then, we can obtain the estimate
			\beqq
			\begin{aligned}
				 &\; \underset{0\le \tau \le t}{\sup}\me(\tau)
					+\underset{0\le \tau \le t}{\sup}[(1+\tau)^s \mathcal{E}_{tan}(\tau)] +\underset{0\le \tau \le t}{\sup}[(1+\tau)^{1+\sigma} \bme_{tan}(\tau) ]
					 \\
                    &\;+\int_0^t \md(\tau)d\tau  +\int_0^t (1+\tau)^{\sigma} \md_{tan}(\tau)d\tau
					+ \int_0^t (1+\tau)^{1+\sigma} \bmd_{tan}(\tau)d\tau
				\le
				 3C(\me(0)
				+\delta^{\frac43}).
			\end{aligned}
			\deqq
			Now choose the small constant
			\beqq\label{choose_delta}
			\delta=12 C \me(0) \le \min\{1,\frac{1}{(12 C)^3}\},
			\deqq
			then we have
			\beqq
			\begin{aligned}
				&\; \underset{0\le \tau \le t}{\sup}\me(\tau)
					+\underset{0\le \tau \le t}{\sup}[(1+\tau)^s \mathcal{E}_{tan}(\tau)] +\underset{0\le \tau \le t}{\sup}[(1+\tau)^{1+\sigma} \bme_{tan}(\tau) ]
					 \\
                    &\;+\int_0^t \md(\tau)d\tau  +\int_0^t (1+\tau)^{\sigma} \md_{tan}(\tau)d\tau
					+ \int_0^t (1+\tau)^{1+\sigma} \bmd_{tan}(\tau)d\tau\\
				\le &\;
				 3C \me(0)
				+3C\delta^{\frac43} 
                \le  
				 \frac{\delta}{4}+\frac{\delta}{4}=\frac{\delta}{2},
			\end{aligned}
			\deqq
			which implies the estimate \eqref{close_assumption}.
			Therefore, we complete the proof of Proposition \ref{main_pro}.

			\begin{proof}[\textbf{Proof of Theorem \ref{main_result_one}}]
				Suppose the assumptions in Theorem \ref{main_result_one} hold,
				due to the horizontal dissipative structure, one can establish the uniform local-in-time well-posedness for the equation \eqref{eqr}.
				Next, we use the standard continuity argument to show the global well-posedness.
				From the local existence result
				and smallness assumption of initial condition,
				it holds
				\beqq \bal
				&\; \underset{0\le \tau \le t}{\sup}\me(\tau)
					+\underset{0\le \tau \le t}{\sup}[(1+\tau)^s \mathcal{E}_{tan}(\tau)] +\underset{0\le \tau \le t}{\sup}[(1+\tau)^{1+\sigma} \bme_{tan}(\tau) ]
					 \\
                    &\;+\int_0^t \md(\tau)d\tau  +\int_0^t (1+\tau)^{\sigma} \md_{tan}(\tau)d\tau
					+ \int_0^t (1+\tau)^{1+\sigma} \bmd_{tan}(\tau)d\tau
				 \le \delta,
				\dal \deqq
				for all $t\in [0, T_0)$ and $\delta=12 C \me(0) $
				and $C$ is a positive constant independent of time $t$ and
				parameter $\var$.
				Set
				\beq\label{criterion} \bal
				T^*:= &\; \underset{T_0}{\sup}
				\Big\{T_0~|
				\underset{0\le \tau \le t}{\sup}\me(\tau)
					+\underset{0\le \tau \le t}{\sup}[(1+\tau)^s \mathcal{E}_{tan}(\tau)] +\underset{0\le \tau \le t}{\sup}[(1+\tau)^{1+\sigma} \bme_{tan}(\tau) ]
					 \\
                    &\;+\int_0^t \md(\tau)d\tau  +\int_0^t (1+\tau)^{\sigma} \md_{tan}(\tau)d\tau
					+ \int_0^t (1+\tau)^{1+\sigma} \bmd_{tan}(\tau)d\tau
				 \le \delta,
				\quad \forall ~ t\in [0, T_0)\Big\},
				\dal \deq
				we claim that $T^*=+\infty$. Otherwise, applying the estimate
				\eqref{close_assumption}
				and the local-in-time existence result,
				there exists a positive constant $T^{**}$ such that $T^{**}>T^{*}$,
				it holds that for any $T\in [T^{*}, T^{**})$,
				\beqq \bal
				&\; \underset{0\le \tau \le t}{\sup}\me(\tau)
					+\underset{0\le \tau \le t}{\sup}[(1+\tau)^s \mathcal{E}_{tan}(\tau)] +\underset{0\le \tau \le t}{\sup}[(1+\tau)^{1+\sigma} \bme_{tan}(\tau) ]
					 \\
                    &\;+\int_0^t \md(\tau)d\tau  +\int_0^t (1+\tau)^{\sigma} \md_{tan}(\tau)d\tau
					+ \int_0^t (1+\tau)^{1+\sigma} \bmd_{tan}(\tau)d\tau\le \delta,
				\quad \forall ~ t\in [0, T).
				\dal \deqq
				This contradicts the definition of $T^*$ in \eqref{criterion}.
				Therefore, we can deduce that $T^*=+\infty$.
				Therefore, we complete the proof of Theorem \ref{main_result_one}.
			\end{proof}

            \section{Convergence of solution}\label{asymptotic-behavior}
					In this section, we will obtain the  vanishing vertical viscosity limit to the equation \eqref{eqr0} based on the uniform estimate established in previous section.
		Thanks to Theorem \ref{main_result_one}, the global uniform estimate \eqref{uniform_estimate} holds for any $t>0$,
		which yields directly
		\beq \label{est-u-vro-compact}
		\sup_{0\le \tau \le t} (\|(\vro^{\var},u^{\var})\|_{H_{co}^{m}}^2 +\|(\nabla \vro^{\var},\nabla u^{\var})\|_{H_{co}^{m-1}}^2) \le C \delta_0.
		\deq	
		On the other hand, due to the equation \eqref{eqr} and the global uniform estimate \eqref{uniform_estimate}, it holds
		\beq\label{est-ptu-vro}  \bal
        \int_{0}^{t} \| \p_t \vro^{\var}\|_{H_{co}^{m-1}}^2 d \tau \lesssim  &\; \int_{0}^{t} ( \|  \du^{\var} \|_{H_{co}^{m-1}}^2 +  \|  u^{\var} \cdot \nabla \vro^{\var} \|_{H_{co}^{m-1}}^2 + \|  \vro^{\var} \, \du^{\var}\|_{H_{co}^{m-1}}^2 )d \tau \le C,\\
		\int_{0}^{t} \| \p_t u^{\var}\|_{H_{co}^{m-1}}^2 d \tau
		\lesssim  &\;  \int_{0}^{t} (  \|  (\Delta_h u^{\var}, \nabla \du^{\var},  \nabla \vro^{\var})\|_{H_{co}^{m-1}}^2 + \var^2 \|  \p_3^2 u^{\var}\|_{H_{co}^{m-1}}^2+ \|  u^{\var} \cdot \nabla u^{\var}\|_{H_{co}^{m-1}}^2 )d \tau\\
        &\; +  \int_{0}^{t} \| \vro^{\var} \, \nabla \vro^{\var}\|_{H_{co}^{m-1}}^2 d \tau +  \int_{0}^{t} \| \f{\vro^{\var}}{1+\vro^{\var}}( \Delta_h u^{\var} +  \nabla \du^{\var} + \var \p_3^2 u^{\var}) \|_{H_{co}^{m-1}}^2  d \tau  \\
		\le  &\; C.
		\dal \deq
		Consequently, from the estimates \eqref{est-u-vro-compact}, \eqref{est-ptu-vro} and the  strong compactness argument(see Lemma 4 in \cite{Simon1990})
		that $\vro^{\var}$ and $u^\var$ are compact in $C([0,t]; H^{m-1}_{co,loc})$ for any $t>0$.
		In particular, there exist a sequence $\var_n \rightarrow 0^+$
		and $\vro^0, u^0 \in C([0,t]; H^{m-1}_{co,loc})$ such that
		\beqq \bal
		&\; \vro^{\var_n}
		\rightarrow   \vro^0 \quad \text{in} \, C([0,t]; H^{m-1}_{co,loc}) \,
		\text{as} \ {\var_n} \rightarrow 0^+,\\
        &\; u^{\var_n}
		\rightarrow   u^0 \quad \text{in} \, C([0,t]; H^{m-1}_{co,loc}) \,
		\text{as} \ {\var_n} \rightarrow 0^+.
		\dal \deqq
         With above information, passing the limit $\var_n \to 0^+$ in \eqref{eqr}, we can check that $\vro^0$ and $u^0$ solve the equation $\eqref{eqr0}_1$ and $\eqref{eqr0}_2$ respectively as well as the
initial data of $(\vro^0, u^0)$ is the limit of $(\vro^{\var},u^{\var})(0)$, which yields the condition \eqref{eqr-i2}.
         Furthermore, from the uniform estimate \eqref{uniform_estimate}, $u^{\var}$ is bounded in $L^2([0, t]; H^{1})$. Without loss of generality, the subsequence $u^{\var}$ also converges weakly in $L^2([0, t]; H^{1})$ to $u^0$, which yields the boundary condition  $\eqref{eqr0}_3$.
        Thus we end the proof of Theorem \ref{main_result_two}.
        
		\section*{Acknowledgments}
		Jincheng Gao was partially supported by the National Key Research and Development Program of China(2021YFA1002100), Guangdong Special Support Project (2023TQ07A961) and Guangzhou Science and Technology Program (2024A04J6410).
		Lianyun Peng was partially supported by a fellowship award (PolyU RFS2122-1S05).
		Jiahong Wu was partially supported by the
		National Science Foundation of the United States (DMS 2104682, DMS 2309748).
			Zheng-an Yao was partially supported by
			National Key Research and Development Program of China (2020YFA0712500).	

			\begin{appendices}
				\section{Some useful inequalities}\label{usefull-inequality}	
		Now let us state some  anisotropic Sobolev inequalities used frequently in our paper.
				\begin{lemm}\label{lemm:sobolev-ie}
					For any suitable functions $(f(x), g(x), h(x))$ defined on $\mathbb{R}^3_+$ and different numbers $i,j,k \in \{1, 2, 3\}$,
					the following estimates hold
					\beq \label{ie:Sobolev}
					\bal
					\|f\|_{L^\infty} &\; \lesssim \|f\|_{L^2}^{\f18}\|\p_1 f\|_{L^2}^{\f18}\|\p_2 f\|_{L^2}^{\f18}\|\p_{12} f\|_{L^2}^{\f18}
					\|\p_3 f\|_{L^2}^{\f18}\|\p_{13} f\|_{L^2}^{\f18}\|\p_{23} f\|_{L^2}^{\f18}\|\p_{123} f\|_{L^2}^{\f18},\\
					\int_{\mathbb{R}^3_+} |f g h |\, dx
					&\; \lesssim \|f\|_{L^2}
					\|g\|_{L^2}^{\frac12}
					\|\p_i g\|_{L^2}^{\frac12}
					\|h\|_{L^2}^{\frac14}
					\|\p_j h\|_{L^2}^{\frac14}
					\|\p_{k} h\|_{L^2}^{\frac14}
					\|\p_{jk} h\|_{L^2}^{\frac14},\\
					\int_{\mathbb{R}^3_+} |f g h |\, dx
					&\; \lesssim \|f\|_{L^2}^{\frac12}
					\|\p_1 f\|_{L^2}^{\frac12}
					\|g\|_{L^2}^{\frac12}
					\|\p_2 g\|_{L^2}^{\frac12}
					\|h\|_{L^2}^{\frac12}
					\|\p_3 h\|_{L^2}^{\frac12},\\
					\|Z_3 f\|_{L^2}
					&\lesssim \|f\|_{L^2}^{\frac23}\|(f,Z_3^3 f)\|_{L^2}^{\frac13},\\
					\|Z_3 f\|_{L^2}
					&\lesssim \|f\|_{L^2}^{\frac34}\|(f,Z_3^4 f)\|_{L^2}^{\frac14},\\
                    \left\|\|f\|_{L^\infty(\mathbb{R}_+)}\right\|_{L^{\frac2s}(\mathbb{R}^2)}
					&\lesssim \left(\|f\|_{L^2}\|\partial_2 f\|_{L^2}
					+\|\partial_1 f\|_{L^2}\|\partial_{12}f\|_{L^2}\right)^{\frac{1-s}{2}}
					\|f\|_{L^2}^{\frac{2s-1}{2}}
					\|\partial_3 f\|_{L^2}^{\frac12}.
					\dal
					\deq
				\end{lemm}
				
				\begin{proof}
					First of all, one can follow the idea in \cite[Lemma 1.2]{Wu2021Advance} to establish the inequalities $\eqref{ie:Sobolev}_1$-$\eqref{ie:Sobolev}_3$.
					Let us give the proof of inequality $\eqref{ie:Sobolev}_4$ and $\eqref{ie:Sobolev}_6$.
					Indeed, integrating by part, it holds
					\begin{equation*}
						\|Z_3 f\|_{L^2}^2=-\int_{\mathbb{R}_+^3}(\varphi' Z_3 f+Z_3^2 f)f dx
						\lesssim \|Z_3 f\|_{L^2}\|f\|_{L^2}
						+\|Z_3^2 f\|_{L^2}\|f\|_{L^2},
					\end{equation*}
					which yields directly
					\begin{equation}\label{a10}
						\|Z_3 f\|_{L^2}
						\lesssim \|f\|_{L^2}
						+\|f\|_{L^2}^{\frac12}\|Z_3^2 f\|_{L^2}^{\frac12}.
					\end{equation}
					Applying estimate \eqref{a10} twice, it is easy to check that
					\begin{equation*}
						\begin{aligned}
							\|Z_3^2 f\|_{L^2}
							&\lesssim \|Z_3 f\|_{L^2}
							+\|Z_3 f\|_{L^2}^{\frac12}\|Z_3^3 f\|_{L^2}^{\frac12}\\
							&\lesssim \|f\|_{L^2}
							+\|f\|_{L^2}^{\frac12}\|Z_3^2 f\|_{L^2}^{\frac12}
							+\left(\|f\|_{L^2}
							+\|f\|_{L^2}^{\frac12}\|Z_3^2 f\|_{L^2}^{\frac12}\right)^{\frac12}
							\|Z_3^3 f\|_{L^2}^{\frac12}\\
							&\lesssim \frac12 \|Z_3^2 f\|_{L^2}+\|f\|_{L^2}
							+\|f\|_{L^2}^{\frac12}\|Z_3^3 f\|_{L^2}^{\frac12}
							+\|f\|_{L^2}^{\frac13}\|Z_3^3 f\|_{L^2}^{\frac23},
						\end{aligned}
					\end{equation*}
					which implies directly
					\begin{equation}\label{a11}
						\|Z_3^2 f\|_{L^2}
						\lesssim \|f\|_{L^2}
						+\|f\|_{L^2}^{\frac12}\|Z_3^3 f\|_{L^2}^{\frac12}
						+\|f\|_{L^2}^{\frac13}\|Z_3^3 f\|_{L^2}^{\frac23}.
					\end{equation}
					Then, the combination of estimates \eqref{a10} and \eqref{a11} yields directly
					\begin{equation*}
						\begin{aligned}
							\|Z_3 f\|_{L^2}
							&\lesssim \|f\|_{L^2}
							+\|f\|_{L^2}^{\frac12}
							\left\{\|f\|_{L^2}
							+\|f\|_{L^2}^{\frac12}\|Z_3^3 f\|_{L^2}^{\frac12}
							+\|f\|_{L^2}^{\frac13}\|Z_3^3 f\|_{L^2}^{\frac23}\right\}^{\frac12}\\
							&\lesssim \|f\|_{L^2}
							+\|f\|_{L^2}^{\frac34}\|Z_3^3 f\|_{L^2}^{\frac14}
							+\|f\|_{L^2}^{\frac23}\|Z_3^3 f\|_{L^2}^{\frac13},
						\end{aligned}
					\end{equation*}
					which yields the estimate $\eqref{ie:Sobolev}_4$. Similarly, we can obtain the estimate $\eqref{ie:Sobolev}_5$, thus we omit the proof.
					Finally, let us establish the estimate $\eqref{ie:Sobolev}_6$.
					Indeed, it holds
					\begin{equation}\label{a6}
						\begin{aligned}
							\left\|\|f\|_{L^\infty(\mathbb{R}_+)}\right\|_{L^{\frac2s}(\mathbb{R}^2)}^{\frac2s}
							\lesssim \left\|\|f\|_{L^2(\mathbb{R}_+)}^{\frac12}
							\|\partial_3 f\|_{L^2(\mathbb{R}_+)}^{\frac12}\right\|_{L^{\frac2s}(\mathbb{R}^2)}^{\frac2s}
							\lesssim \|\partial_3 f\|_{L^2}^{\frac1s}
							\left\|\|f\|_{L^2(\mathbb{R}_+)}\right\|_{L^{\frac{2}{2s-1}}
								(\mathbb{R}^2)}^{\frac1s},
						\end{aligned}
					\end{equation}
					and
					\begin{equation}\label{a7}
						\begin{aligned}
							\left\|\|f\|_{L^2(\mathbb{R}_+)}\right\|_{L^{\frac{2}{2s-1}}(\mathbb{R}^2)}
							^{\frac{2}{2s-1}}
							&\lesssim \left\|\|f\|_{L^\infty(\mathbb{R}^2)}\right\|_{L^2(\mathbb{R}_+)}
							^{\frac{4-4s}{2s-1}}\|f\|_{L^2}^2\\
							&\lesssim \left(\|f\|_{L^2}^{\frac12}\|\partial_2 f\|_{L^2}^{\frac12}
							+\|\partial_1 f\|_{L^2}^{\frac12}\|\partial_{12}f\|_{L^2}^{\frac12}\right)
							^{\frac{4-4s}{2s-1}}\|f\|_{L^2}^2.
						\end{aligned}
					\end{equation}
					The combination of estimates \eqref{a6} and \eqref{a7} gives
					\begin{equation*}
						\begin{aligned}
							\left\|\|f\|_{L^\infty(\mathbb{R}_+)}\right\|_{L^{\frac2s}(\mathbb{R}^2)}
							&\lesssim \|\partial_3 f\|_{L^2}^{\frac12}
							\left\|\|f\|_{L^2(\mathbb{R}_+)}\right\|_{L^{\frac{2}{2s-1}}
								(\mathbb{R}^2)}^{\frac12}\\
							&\lesssim \left(\|f\|_{L^2}\|\partial_2 f\|_{L^2}
							+\|\partial_1 f\|_{L^2}\|\partial_{12}f\|_{L^2}\right)^{\frac{1-s}{2}}
							\|f\|_{L^2}^{\frac{2s-1}{2}}
							\|\partial_3 f\|_{L^2}^{\frac12}.
						\end{aligned}
					\end{equation*}
					Therefore, we complete the proof of this lemma.
				\end{proof}
				
				Next, we introduce the Hardy-Littlewood-Sobolev inequality
				in \cite[pp. 119, Theorem 1]{Stein1970} as follow.
				\begin{lemm}\label{H-L}
					Let $0<\alpha<2 , 1<p<q<\infty, \frac{1}{q}+\frac{\alpha}{2}=\frac{1}{p}$, then
					\begin{equation}\label{a16}
						\|\Lambda_h^{-\alpha}f\|_{L^q(\mathbb{R}^2)}\lesssim \|f\|_{L^p(\mathbb{R}^2)}.
					\end{equation}
				\end{lemm}
				In this paper, taking $q=2$ in \eqref{a16}, then
				$
				p=\frac{1}{\frac12+\frac{\alpha}{2}}=\frac{2}{1+\alpha}
				$
				should satisfy the condition
				$$
				1<p=\frac{2}{1+\alpha} <2.
				$$
				This implies the index $\alpha \in (0, 1)$.
			\end{appendices}

		\section*{Data Availability}
		Data sharing is not applicable to this article as no new data were created or analysed in this study.

		\section*{Conflict of interest}
		The authors declared that they have no Conflict of interest to this work.

			\phantomsection

		\end{sloppypar}

\addcontentsline{toc}{section}{\refname}
		\begin{thebibliography}{99}
			
		
			
			\bibitem{BeiraodaVeiga2010}
			H.~Beir\~{a}o~da Veiga and F.~Crispo.
			\newblock Sharp inviscid limit results under {N}avier type boundary conditions.
			{A}n {$L^p$} theory.
			\newblock {\em J. Math. Fluid Mech.}, 12(3):397--411, 2010.
			
			\bibitem{BeiraodaVeiga2011}
			H.~Beir\~{a}o~da Veiga and F.~Crispo.
			\newblock Concerning the {$W^{k,p}$}-inviscid limit for 3-{D} flows under a
			slip boundary condition.
			\newblock {\em J. Math. Fluid Mech.}, 13(1):117--135, 2011.
			
			
            \bibitem{Bresch-Jabin-2018}
			D.~Bresch and P.E.~Jabin.
             Global existence of weak solutions for compressible Navier-Stokes equations:
            thermodynamically unstable pressure and anisotropic viscous stress tensor.
           \newblock {\em Ann. of Math.}, 188(2):577--684, 2018.
           
            \bibitem{new-Cao-Wu2025}
            C.S.~Cao and J.H.~Wu.
            \newblock Stability of the 3D Navier-Stokes equations
            with anisotropic dissipation.
            \newblock {\em Nonlinearity}, 38(9): No.095012, 2025.

			

			
			\bibitem{ani-CDGG2000}
			J.-Y.~Chemin, B.~Desjardins, I.~Gallagher and E.~Grenier. \newblock Fluids with anisotropic viscosity. \newblock{\em Math. Model. Numer. Anal.}, 34(2):315--335, 2000.
			
			
			
			
			\bibitem{Chemin2007}
			J.-Y.~Chemin and P.~Zhang.
			\newblock On the global wellposedness to the 3-{D} incompressible anisotropic
			{N}avier-{S}tokes equations.
			\newblock{\em Comm. Math. Phys.}, 272(2):529--566, 2007.
			
			
			\bibitem{Constantin1986}
			P.~Constantin.
			\newblock Note on loss of regularity for solutions of the {$3$}-{D}
			incompressible {E}uler and related equations.
			\newblock {\em Comm. Math. Phys.}, 104(2):311--326, 1986.
			
			\bibitem{Constantin1988}
			P.~Constantin and C.~Foias.
			\newblock {\em Navier-{S}tokes equations}.
			\newblock Chicago Lectures in Mathematics. University of Chicago Press,
			Chicago, IL, 1988.
			
			\bibitem{ConstantinWu1995}
			P.~Constantin and J.H.~Wu.
			\newblock Inviscid limit for vortex patches.
			\newblock {\em Nonlinearity}, 8(5):735--742, 1995.

            \bibitem{Fan-Li-Li2022}
            X.Y.~Fan, J.X.~Li and J.~Li.
           \newblock Global existence of strong and weak solutions to 2D compressible Navier-Stokes system in bounded domains with large data and vacuum. 
           \newblock {\em Arch. Ration. Mech. Anal.}, 245(1):239--278, 2022. 
    
			\bibitem{FHWW2025}
			Z.F.~Feng, G.Y.~Hong, Y.H.~Wang and J.H.~Wu.
			\newblock
			Stability and decay rates for the 3D compressible Navier-Stokes 
             equations with eddy diffusion.
			\newblock {\em  preprint}.
			

		
			\bibitem{ConstantinWu1996}
			P.~Constantin and J.H.~Wu.
			\newblock The inviscid limit for non-smooth vorticity.
			\newblock {\em Indiana Univ. Math. J.}, 45(1):67--81, 1996.
			
		
			\bibitem{Gao2024}
			J.C.~Gao, J.H.~Wu, Z.A.~Yao and X.~Yin.
			\newblock Global uniform regularity for the 3D incompressible MHD equations
			with slip boundary condition near an equilibrium.
			\newblock {\em  arXiv:2408.14123}.

            \bibitem{Guo-Wang2012}
            Y.~Guo and Y.J.~Wang.
            \newblock Decay of dissipative equations and negative Sobolev spaces. 
           \newblock {\em  Comm. Partial Differential Equations}, 37(12):2165--2208, 2012.
			


         \bibitem{Huang-Li-Xin-Siam}
          X.D.~Huang, J.~Li  and Z.P.~Xin.
          Serrin type criterion for the three-dimensional viscous compressible flows.
          \newblock {\em SIAM J. Math. Anal.}, 43:1872--1886, 2011. 


          \bibitem{Huang-Li-Xin-CMP}
          X.D.~Huang, J.~Li  and Z.P.~Xin.
          Blowup criterion for the compressible flows with vacuum states.
         \newblock {\em  Comm. Math. Phys.}, 301:23--35, 2011. 

            

            \bibitem{Huang-Li-Xin2012}
            X.D.~Huang, J.~Li and Z.P.~Xin.
            \newblock Global well-posedness of classical solutions with large oscillations and vacuum to the three-dimensional isentropic compressible Navier-Stokes equations.
            \newblock {\em Commun. Pure Appl. Math.}, 65:549--585, 2012.

            \bibitem{Huang-Li2018}
            X.D.~Huang and J.~Li. 
            \newblock Global classical and weak solutions to the three-dimensional full compressible Navier-Stokes system with vacuum 
            and large oscillations. 
            \newblock {\em Arch. Ration. Mech. Anal.}, 227(3):995--1059, 2018.

            \bibitem{Hoff-Zumbrun1995}
             D.~Hoff and K.~Zumbrun.
            \newblock  Multi-dimensional diffusion waves for the Navier-Stokes equations of compressible flow. 
           \newblock {\em Indiana Univ. Math. J.}, 44(2):603--676, 1995.

             \bibitem{Hong-Hou-Peng-Zhu2024}
             G.Y.~Hong, X.H.~Hou, H.Y.~Peng and C.J.~Zhu.
             \newblock Global existence for a class of large solution to compressible Navier-Stokes equations with vacuum. 
             \newblock {\em Math. Ann.}, 388(2):2163--2194, 2024.

            
			\bibitem{Iftimie2002}
			D.~Iftimie.
			\newblock A uniqueness result for the {N}avier-{S}tokes equations with
			vanishing vertical viscosity.
			\newblock {\em SIAM J. Math. Anal.}, 33(6):1483--1493, 2002.
			
			\bibitem{ani-Iftimie-Planas2016}
			D.~Iftimie and G.~Planas.
			\newblock Inviscid limits for the Navier-Stokes equations with Navier friction boundary conditions. \newblock{\em Nonlinearity}, 19(4):899--918, 2006.
			
			\bibitem{Iftimie2011ARMA}
			D.~Iftimie and F.~Sueur.
			\newblock Viscous boundary layers for the {N}avier-{S}tokes equations with the
			{N}avier slip conditions.
			\newblock {\em Arch. Ration. Mech. Anal.}, 199(1):145--175, 2011.
			
			
            \bibitem{JTW2023}
            R.H.~Ji, L.~Tian and J.H.~Wu.
            \newblock 3D anisotropic Navier-Stokes equations in
            $\mathbb{T}^2 \times \mathbb{R}$: stability and large-time behaviour.
            \newblock {\em Nonlinearity}, 36(6):3219--3237, 2023.

            \bibitem{jWy2021}
            R.H.~Ji, J.H.~Wu and W.R.~Yang.
            \newblock Stability and optimal decay for the 3D
            Navier-Stokes equations with horizontal dissipation.
            \newblock {\em J. Differential Equations}, 290:57--77, 2021.
			
			 \bibitem{Kagei-Kobayashi2002}
            Y.~Kagei and T.~Kobayashi.
           \newblock  On large time behavior of solutions to the compressible Navier–Stokes equations in the half space in $R^3$. 
\newblock {\em Arch. Ration. Mech. Anal.}, 165:89--159, 2002.

			\bibitem{Kato1972}
			T.~Kato.
			\newblock Nonstationary flows of viscous and ideal fluids in {${\bf R}\sp{3}$}.
			\newblock {\em J. Funct. Anal.}, 9:296--305, 1972.
			
           \bibitem{Liu-Wang1998}
            T.P.~Liu and W.K.~Wang.
            \newblock The pointwise estimates of diffusion wave for the 
            Navier-Stokes systems in odd multi-dimensions.
            \newblock {\em Comm. Math. Phys.}, 196(1):145--173, 1998.
          
          

            

            
			
			\bibitem{Liu-Paica-Zhang-2020}
             Y.L.~Liu, M.~Paicu and P.~Zhang
             \newblock  Global well-posedness of 3-D anisotropic Navier-Stokes
             system with small unidirectional derivative.
             \newblock {\em  Arch. Ration. Mech. Anal.}, 238(2):805--843, 2020.
	
			
			\bibitem{Masmoudi2007CMP}
			N.~Masmoudi.
			\newblock Remarks about the inviscid limit of the {N}avier-{S}tokes system.
			\newblock {\em Comm. Math. Phys.}, 270(3):777--788, 2007.

            \bibitem{Matsumura-Nishida1980}
            A.~Matsumura and T.~Nishida.
            The initial value problem for the equations of motion of viscous and
            heat-conductive gases.
            \newblock {\em J. Math. Kyoto Univ.}, 20:67--104, 1980.

            \bibitem{Matsumura-Nishida1983}
             A.~Matsumura and T.~Nishida.
             The initial boundary value problems for the equations of motion of
             compressible and heat-conductive fluids.
             \newblock {\em Commun. Math. Phys.}, 89:445--464, 1983.	
             		
			\bibitem{Masmoudi2012ARMA}
			N.~Masmoudi and F.~Rousset.
			\newblock Uniform regularity for the {N}avier-{S}tokes equation with {N}avier
			boundary condition.
			\newblock {\em Arch. Ration. Mech. Anal.}, 203(2):529--575, 2012.

            \bibitem{Masmoudi-Sun-Wang-Zhang}
            N.~Masmoudi, C.Z.~Sun, C.~Wang and Z.F.~Zhang
            Incompressible and vanishing vertical viscosity limit for the compressible Navier-Stokes system with Dirichlet boundary conditions.
            \newblock {\em arXiv:2501.05063}.

            
			\bibitem{1987Geophysical}
			J.~Pedlosky.
			\newblock {\em Geophysical Fluid Dynamics, Second Edition}.
			\newblock Geophysical Fluid Dynamics, Second Edition, 1987.			

			\bibitem{Paicu2005-2}
			M.~Paicu.
			\newblock Anisotropic Navier-Stokes equation in critical spaces.
			\newblock {\em Rev. Mat. Iberoamericana}, 21(1):179--235, 2005.
			
			\bibitem{Paicu2005}
			M.~Paicu.
			\newblock \'{E}quation periodique de {N}avier-{S}tokes sans viscosit\'{e} dans
			une direction.
			\newblock {\em Comm. Partial Differential Equations}, 30(7-9):1107--1140, 2005.

            \bibitem{Sun-Wang-Zhang-JMPA}
			Y.Z.~Sun, C.~Wang and Z.F.~Zhang.
            A Beale-Kato-Majda blow-up criterion for the 3-D compressible Navier-Stokes equations.
             \newblock {\em J. Math. Pures Appl.}, 
             95:36--47, 2011. 
			
            
			\bibitem{Simon1990}
			J.~Simon. Nonhomogeneous viscous incompressible fluids: existence of velocity, density, and pressure, \newblock{SIAM J. Math. Anal.}, 21(5):1093--1117, 1990.
						
			\bibitem{Stein1970}
			E.M.~Stein.
			\newblock {\em  Singular integrals and differentiability properties of
				functions}, volume No. 30 of { Princeton Mathematical Series}.
			\newblock {Princeton University Press, Princeton, NJ}, 1970. 

            
			
			\bibitem{Temam-Ziane-2004}
             R.~Temam and M.~Ziane.
             Some mathematical problems in geophysical fluid dynamics.
             \newblock{\em Handbook of mathematical fluid dynamics},
             Vol. III, North-Holland, Amsterdam, 535--657, 2004.

			
            \bibitem{Wen-2025}
            H.Y.~Wen.
            \newblock Global wellposedness of compressible Navier-Stokes equations with
            vacuum and smallness on scaling invariant quantity in $\mathbb{R}^3$.
            \newblock {\em Adv. Math.},  482: No.110628,  2025.

         \bibitem{Wen-Zhu-Adv}
          H.Y.~Wen and C.J.~Zhu.
          Blow-up criterions of strong solutions to 3D compressible Navier-Stokes equations with vacuum.
        \newblock {\em Adv. Math.}, 248:534--572, 2013.



           \bibitem{Wen-2017}
           H.Y.~Wen and C.J.~Zhu.
          \newblock  Global solutions to the three-dimensional full compressible Navier-Stokes equations with vacuum at infinity in some classes of large data. 
          \newblock {\em SIAM J. Math. Anal.}, 49(1):162--221, 2017. 


			\bibitem{Wang-2016}
			Y.~Wang.
			Uniform regularity and vanishing dissipation limit for the
			full compressible Navier-Stokes system in three dimensional
			bounded domain.
			\newblock {\em Arch. Ration. Mech. Anal.}, 221(3):1345--1415, 2016.

			\bibitem{Wang-Xin-Yan2015}
             Y.~Wang, Z.P.~Xin and Y.~Yong.
            Uniform regularity and vanishing viscosity limit for the compressible
            Navier-Stokes with general Navier-slip boundary conditions
            in three-dimensional domains.
           \newblock {\em SIAM J. Math. Anal.}, 47(6):4123--4191, 2015.

			

			
			
			\bibitem{Wu2021Advance}
			J.H.~Wu and Y.~Zhu.
			\newblock Global solutions of 3{D} incompressible {MHD} system with mixed
			partial dissipation and magnetic diffusion near an equilibrium.
			\newblock {\em Adv. Math.}, 377:107466, 2021.
            



			\bibitem{Xiao2007}
			Y.L.~Xiao and Z.P.~Xin.
			\newblock On the vanishing viscosity limit for the 3{D} {N}avier-{S}tokes
			equations with a slip boundary condition.
			\newblock {\em Comm. Pure Appl. Math.}, 60(7):1027--1055, 2007.
			
			\bibitem{Xiao2013}
			Y.L.~Xiao and Z.P.~Xin.
			\newblock On the inviscid limit of the 3{D} {N}avier-{S}tokes equations with
			generalized {N}avier-slip boundary conditions.
			\newblock {\em Commun. Math. Stat.}, 1(3):259--279, 2013.
			
			
			\bibitem{Zhang-2009}
			T.~Zhang.
			Global wellposed problem for the 3-D incompressible
			anisotropic Navier-Stokes equations in an anisotropic space.
			\newblock {\em Comm. Math. Phys.}, 287(1):211--224, 2009.
			
			
		\end{thebibliography}
	\end{document}